\documentclass[reqno]{amsart}
\usepackage{setspace}
\usepackage{amsmath,amsthm,amssymb} %各种数学公式宏包
\usepackage{color}
\usepackage{comment}
\usepackage[margin=1in]{geometry}
\usepackage{graphicx}
\usepackage{float}
\usepackage{mathrsfs}
\usepackage{yfonts}
\usepackage{enumitem}
\usepackage{lscape}
\usepackage{subfigure}
\usepackage{bm}
\usepackage{tikz}
\usetikzlibrary{arrows.meta} % 箭头样式库
\usepackage[numbers,sort&compress]{natbib}

\numberwithin{figure}{section}
\numberwithin{equation}{section}
\newtheorem{rhp}{Riemann-Hilbert Problem}[section]
\newtheorem{thm}{Theorem}[section]
\newtheorem{lem}{Lemma}[section]
\newtheorem{prop}{Proposition}[section]
\theoremstyle{remark}
\newtheorem{rmk}{\bf Remark}[section]

\newcommand{\RNum}[1]{\uppercase\expandafter{\romannumeral #1\relax}}

\begin{document}

\title{\bf Long-time asymptotics of the defocusing mKdV equation\\with step initial data}

\author{Deng-Shan Wang}
\address{Deng-Shan Wang: Laboratory of Mathematics and Complex Systems (Ministry of Education), School of Mathematical Sciences, Beijing Normal University, Beijing 100875, China}
\email{dswang@bnu.edu.cn}
\author{Ding Wen}
\address{Ding Wen: Laboratory of Mathematics and Complex Systems (Ministry of Education), School of Mathematical Sciences, Beijing Normal University, Beijing 100875, China}
\email{wd2001@mail.bnu.edu.cn}


\date{\today}


\begin{abstract}

This work investigates the long-time asymptotics of solution to defocusing modified Korteweg-de Vries equation with a class of step initial data. A rigorous asymptotic analysis is conducted on the associated Riemann-Hilbert problem by applying
Deift-Zhou nonlinear steepest descent method. In this process, the construction of odd-symmetry $g$-function is generalized and the method of genus reduction on the Riemann-theta function is proposed via conformal transformation and symmetries. It is revealed that for sufficiently large time, the solution manifests a tripartite spatiotemporal structure, i.e., in the left plane-wave region, the solution decays to a modulated plane wave with oscillatory correction; in the central dispersive shock wave region, the solution is governed by a modulated elliptic periodic wave; in the right plane wave region, the solution converges exponentially to a constant. The results from the long-time asymptotic analysis have been shown to match remarkably well with that obtained by direct numerical simulations.

\end{abstract}

\maketitle

\setcounter{tocdepth}{1}

\tableofcontents

\section{Introduction}

In this paper, we investigate the Cauchy problem of the defocusing modified Korteweg-de Vries (mKdV) equation 
\begin{equation}\label{mkdv}
q_t(x,t)-6q^2(x,t)q_x(x,t)+q_{xxx}(x,t)=0, \quad (x,t)\in\mathbb{R}\times \mathbb{R}_{\ge 0},
\end{equation} 
with a sharp constant step initial data $q(x,0)$ of the form
\begin{equation}\label{initial}
	q(x,0)=\begin{cases}
		q_l, \quad x< 0, \\ q_r, \quad x> 0, 
	\end{cases}
\end{equation}
where $q_l$ and $q_r$ are two real constants. The main focus of this work is to derive the precise asymptotic expression of the unique classical solution as $t\to\infty$. In addition, to ensure well-posedness of the Cauchy problem for (\ref{mkdv}) with (\ref{initial}), the following integrability condition must be imposed: 
\begin{equation}
	||q(\cdot,t)-q_r||_{L^1(\mathbb{R}^+)} + ||q(\cdot,t)-q_l||_{L^1(\mathbb{R}^-)} < \infty, \quad \forall t\ge 0.
\end{equation}
\par
Noticing that the transformation $q(x,t)\mapsto-q(x,t)$ keeps the equation (\ref{mkdv}) unchanged, it suffices to consider the case of $q_l\ge 0$.  However, when the values of $q_l\ge 0$ and $|q_r|$ are fixed, the cases $q_r> 0$ and $q_r< 0$ lead to a significantly different asymptotic analysis. Figure \ref{kink} shows the numerical simulations of the defocusing mKdV equation (\ref{mkdv}) with initial condition (\ref{initial}) for various parameters. It is observed that the case of $q_r< 0$ introduces an additional kink soliton region compared to the case of $q_r>0$. In the present work, we restrict ourselves to the case of $q_r\ge q_l\ge 0$, while a concise discussion of the cases $q_r\ge-q_l \ge 0$ and $-q_l \ge q_r\ge 0$ is provided in Remark \ref{a(0)=0}. For the case of $q_l\ge q_r \ge 0$, relevant results can be found in \cite{Xu2021}. 
\begin{figure}[htbp]
	\centering
	\subfigure[$q_l=0.2$, $|q_r|=0.8$, $t=15$]{
		\begin{minipage}[b]{0.45\textwidth}  % Adjusted width to fit both images
			\centering
			\includegraphics[width=\linewidth]{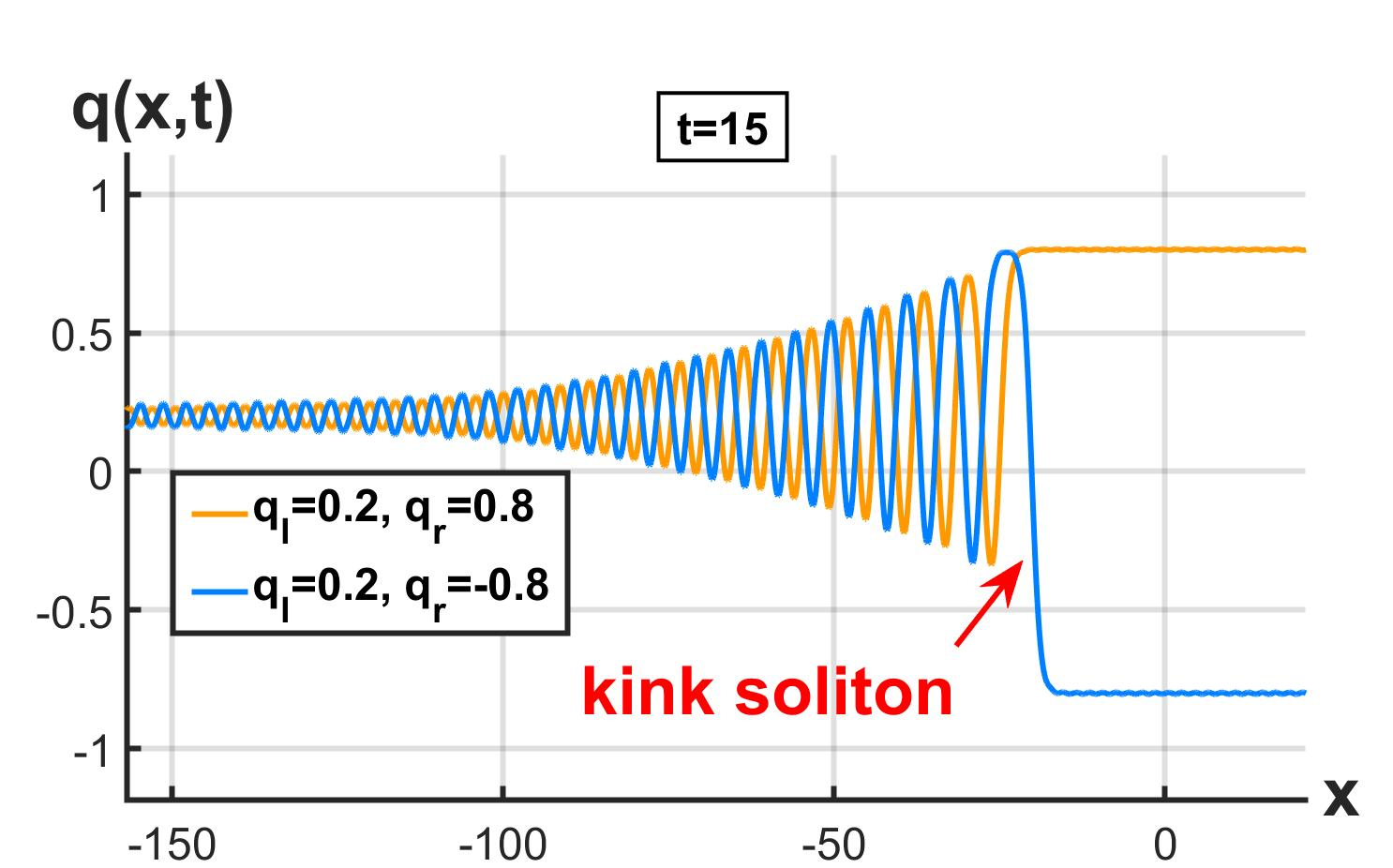}  % Use \linewidth to fit the minipage
		\end{minipage}
	}
	\hfill  % Adds horizontal space between subfigures
	\subfigure[$q_l=0.8$, $|q_r|=0.2$, $t=15$]{
		\begin{minipage}[b]{0.48\textwidth}  % Adjusted width to fit both images
			\centering
			\includegraphics[width=\linewidth]{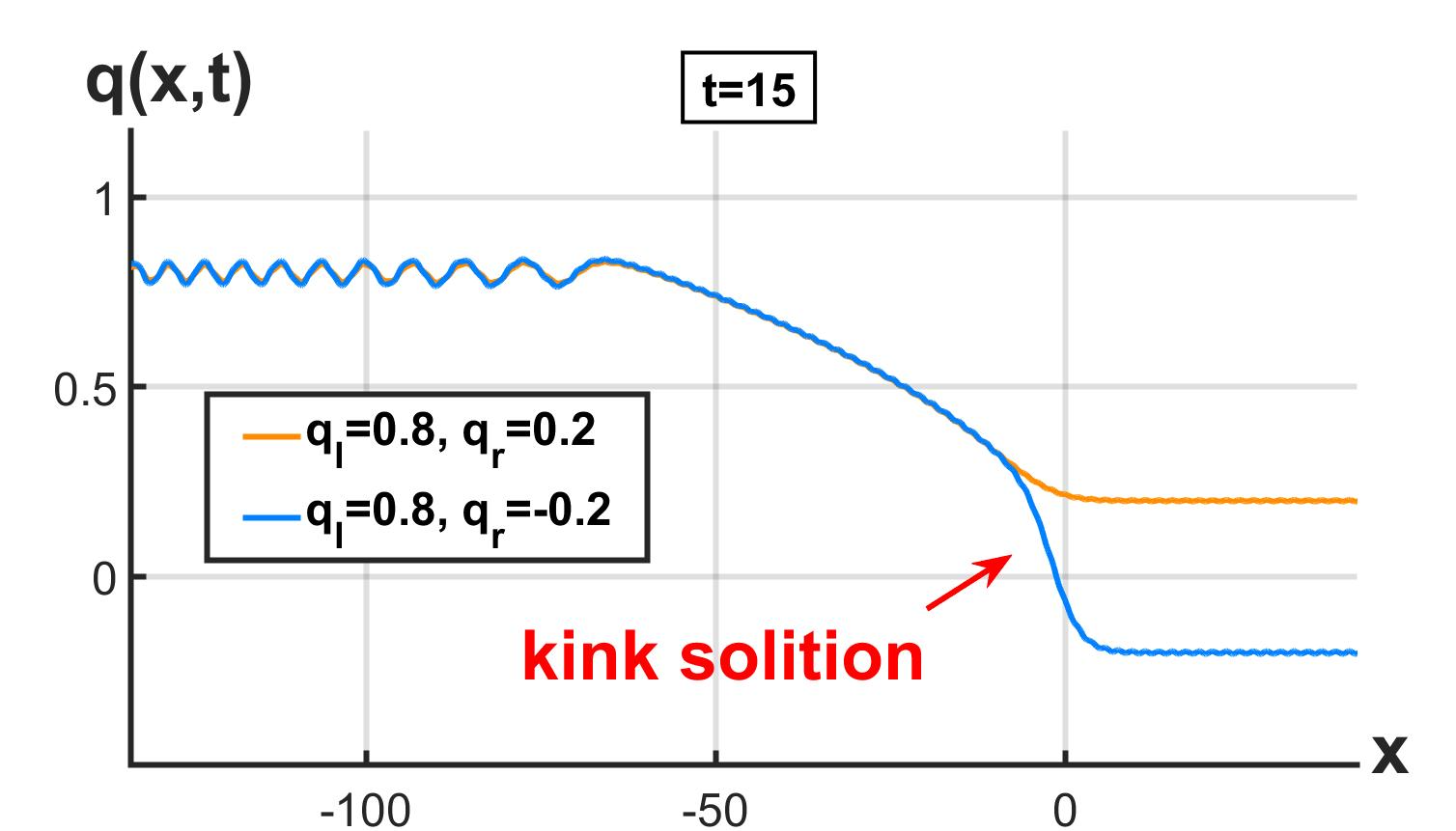}
		\end{minipage}
	}
	\caption{Direct numerical simulations of the defocusing mKdV equation (\ref{mkdv}) with initial condition (\ref{initial}) for various parameters.}  % Main caption
	\label{kink}  % Label for referencing
\end{figure}
\par
The defocusing mKdV equation (\ref{mkdv}) constitutes a fundamental mathematical model that emerges in diverse physical systems, including waves in shallow two-layer fluids \cite{Kakutani1978,JFM1984}, ion acoustic waves in a plasma with negative ions \cite{Watanabe1984}, and a two-electron-temperature plasma \cite{Tajiri1985}. Furthermore, the canonical Miura transformation \cite{Miura1968} 
\begin{equation}
	u(x,t)=q_x(x,t)-q^2(x,t)
\end{equation}
establishes the relation between solutions of (\ref{mkdv}) and the famous KdV equation
\begin{equation}
u_t(x,t)+6u(x,t)u_x(x,t)+u_{xxx}(x,t)=0. 
\end{equation}
\par
The investigation of long-time asymptotics for the Cauchy problem of integrable dispersive equations with vanishing initial data was pioneered by Ablowitz and Segur \cite{KdV1977} and Zakharov and Manakov \cite{Zakharov1976} through their seminal contributions to the inverse scattering transform. The nonlinear steepest descent method, introduced by Deift and Zhou \cite{DZ1993}, established a rigorous analytical framework for oscillatory Riemann-Hilbert problems, thereby enabling precise characterization of long-time asymptotics of the solution of the defocusing mKdV equation with decaying initial data. These foundational studies have established crucial theoretical frameworks for subsequent mathematical analysis of step-like initial value problems. 
\par
The seminal progress in long-time asymptotic analysis for the Cauchy problem with step-like initial data was first attained in the context of the KdV equation. This development stemmed from the foundational work of Gurevich and Pitaevskii \cite{KdV1973}, who, employing Whitham modulation theory, theoretically predicted the emergence of highly oscillatory structures now termed as dispersive shock waves. The rigorous mathematical justification of this phenomenon was subsequently provided by Khruslov \cite{Khruslov1976} via the inverse scattering method, with later extensions encompassing a broad class of integrable systems \cite{Kotlyarov1989}. 
\par
The rigorous asymptotic analysis within the framework of the Riemann-Hilbert method of step-like initial value problems for integrable systems was pioneered in the seminal works \cite{Buckingham2007,CMP2009}.  Since then, the long-time
asymptotics of integrable dispersive equations with step-like initial conditions
has been studied for KdV in \cite{JDE2016,KdV2013}, for the NLS equation in \cite{Biondini2021, IMRN2011, Study2021, Lenells2021, Jenkins2015, Jenkins2016}, and for the Camassa–Holm equation in \cite{Minakov2016}. For the mKdV equation, the analysis of the focusing version was initiated in the work \cite{Kotlyarov1989}, and later in \cite{Minakov2010, Minakov2012, Minakov2015, Minakov2019, Minakov2020} via the asymptotic analysis of the Riemann-Hilbert problems.  
\par
It is worth mentioning that Grava and Minakov \cite{Minakov2020} proposed a conformal transformation method to reduce a hyperelliptic (genus-two) model Riemann-Hilbert problem, characterized by the symmetry
\begin{equation}
	M(z)=\overline{M(-\bar{z})}=\sigma_2 M(-z)\sigma_2=\sigma_2\overline{M(\bar{z})}\sigma_2, 
\end{equation}
to an ellptic (genus-one) model problem. This reduction allowed for the derivation of the unique solution to the hyperelliptic problem, which is used to reconstruct the leading-order Jacobi elliptic wave solution to the focusing mKdV equation \cite{Minakov2020}. In Section \ref{section5} of this work, we generalize this approach to handle the sharp step initial problem of the  defocusing mKdV equation (\ref{mkdv}) which associates with a hyperelliptic (also genus-two) model Riemann-Hilbert problem with the symmetry
\begin{equation}
	M(z)=\overline{M(-\bar{z})}=\sigma_1 M(-z)\sigma_1=\sigma_1\overline{M(\bar{z})}\sigma_1, 
\end{equation} 
where $\sigma_1$ and $\sigma_2$ are Pauli matrices given in (\ref{Pauli-matrices}) below.
\par

\subsection{Main results and numerical comparisons}\
\par
Numerical simulations indicate that the solution to initial problem of the defocusing mKdV equation (\ref{mkdv})-(\ref{initial})  with $q_r>q_l>0$ develops characteristic spatiotemporal patterns. As illustrated in Figure \ref{3D}, for parameters $q_l=0.2, q_r=0.8$, spatial profiles of the solution exhibit a tripartite structure characterized by two constant-wave regions bordering a central oscillatory zone. 
\par

\begin{figure}[htbp]
	\centering
	\subfigure[Top view of the contour for $q(x,t)$ ]{
		\begin{minipage}[b]{0.4\textwidth}  % Adjusted width to fit both images
			\centering
			\includegraphics[width=\linewidth]{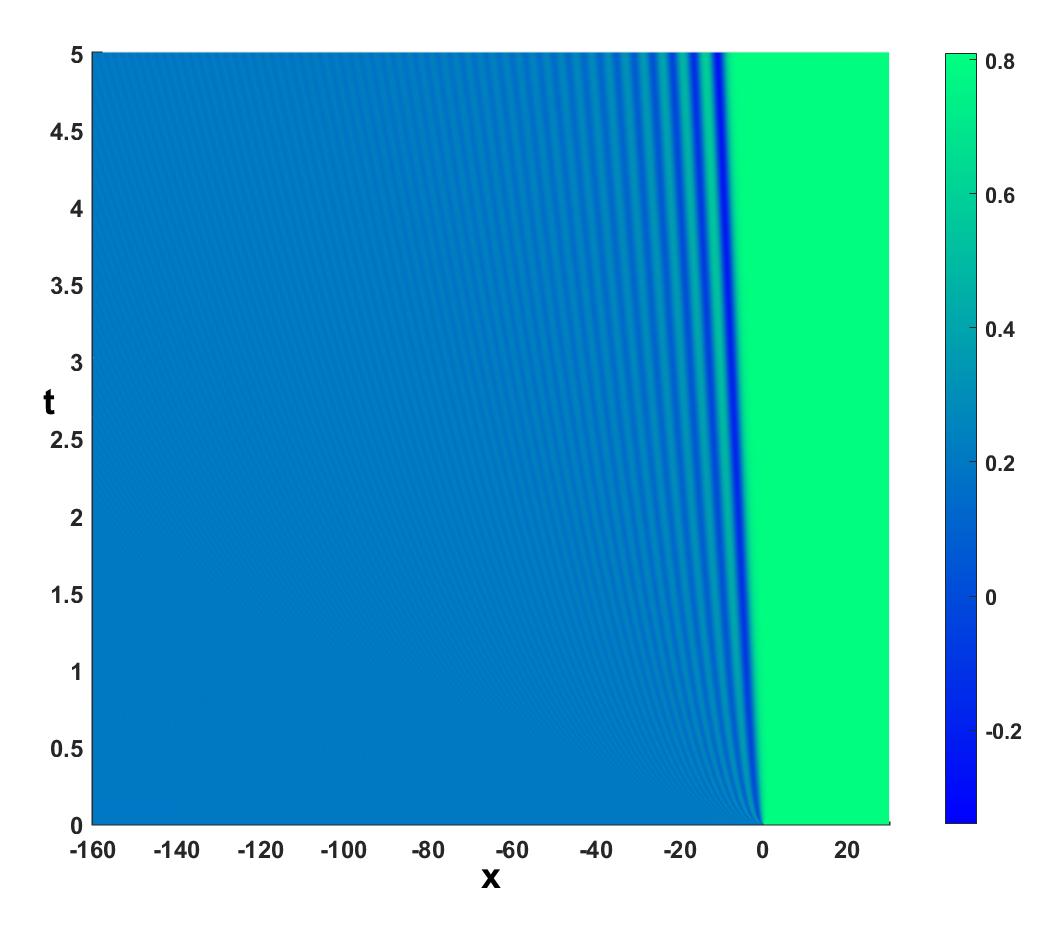}
		\end{minipage}
	}
	\hfill  % Adds horizontal space between subfigures
	\subfigure[Three-dimensional spatiotemporal structure of $q(x,t)$]{
		\begin{minipage}[b]{0.4\textwidth}  % Adjusted width to fit both images
			\centering
			\includegraphics[width=\linewidth]{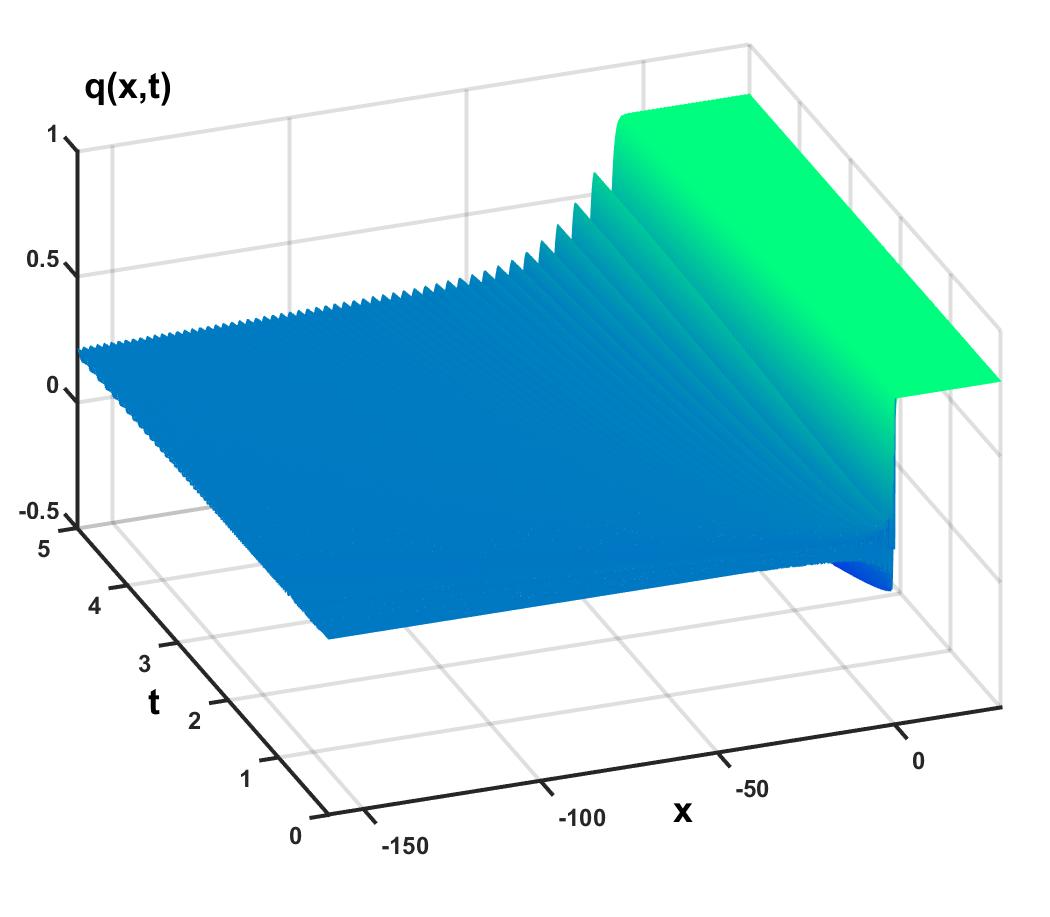}  % Use \linewidth to fit the minipage
		\end{minipage}
	}
	\caption{The evolution of $q(x,t)$ on the upper $(x,t)$-plane for $q_l=0.2$, $q_r=0.8$ and $t\in[0,15]$}  % Main caption
	\label{3D}  % Label for referencing
\end{figure}

To provide the rigorous mathematical characterization of these numerical observations, Theorem 1.1 presents the long-time asymptotics of the solution to the defousing mKdV equation (\ref{mkdv}) with initial data (\ref{initial}). This characterization, derived through rigorous asymptotic analysis within the framework of Riemann-Hilbert formulation, demonstrates that the observed three-region configuration constitutes an intrinsic structural feature of the solution, with the exception of two sufficiently narrow transition regions as specified in Remark \ref{transition}, see also Figure \ref{region}. 
\par

\begin{thm}\label{main} 
	Given two positive constants $M>\varepsilon>0$, the solution $q(x,t)$ to the defocusing mKdV equation (\ref{mkdv}) with initial data (\ref{initial}) with $q_r\ge q_l\ge0$ admits three distinct asymptotic behaviors, classified by the spatiotemporal region in the upper $(x,t)$ plane as shown in Figure \ref{region}.   
    \par
    {\em (1)} In the left-plane wave {\em (LPW)} region, i.e., Region  {\em \RNum{1}} for $6q_l^2t-12q_r^2t-M<x<6q_l^2t-12q_r^2t-\varepsilon$, the asymptotic solution behaves like
	\begin{equation}\label{left-plane}
		q(x,t)=q_l+t^{-\frac{1}{2}} q_{pc}(x,t)+\mathcal{O}(t^{-1}\ln t),
	\end{equation}
    where $q_{pc}(x,t)$ is expressed by
	\begin{equation}
		\begin{aligned}
			q_{pc}(x,t)=\sqrt{\frac{\nu(\xi)\sqrt{-\xi-\frac{q_l^2}{2}}}{3(-\xi+\frac{q_l^2}{2})}}\mathrm{cos}\left[16t\left(-\xi-\frac{q_l^2}{2}\right)^{\frac{3}{2}}-\nu(\xi)\ln \left(\frac{192t\left(-\xi+\frac{q_l^2}{2}\right)^2}{\sqrt{-\xi-\frac{q_l^2}{2}}}\right)+\phi(\xi)\right],
		\end{aligned}
	\end{equation}
	with
	\begin{equation}
		\begin{aligned}
			&\xi =\frac{x}{12t},\quad \nu(\xi)=-\frac{1}{2\pi}\ln(1-|r(z_l)|^2),\quad z_l(\xi)=\sqrt{\frac{q_l^2}{2}-\xi},\\
			&\phi(\xi):=\frac{\pi}{4}-\arg r(z_l)-\arg \Gamma(-i\nu(\xi))+\arg \chi^2(z_l), \\
		\end{aligned}
	\end{equation} 
	where $\chi(z_l)=\lim_{z\to z_l}\chi(z)$, and $r(z)$ and $\chi(z)$ are given by (\ref{a}) and (\ref{chi}), respectively. 
    \par
	 {\em (2)} In the middle dispersive shock wave {\em (DSW)} region, i.e., Region {\em\RNum{2}} for $6q_l^2t-12q_r^2t+\varepsilon<x<-4q_l^2t-2q_r^2t-\varepsilon$, the asymptotic solution is characterized by a modulated elliptic wave of the form
	\begin{equation}\label{shock}
		q(x,t)=q_{dsw}(x,t)+\mathcal{O}(t^{-1}), 
	\end{equation}
	where $q_{dsw}(x,t)$ is expressed by Jacobi elliptic functions 
    \begin{equation}\label{qdsw}
		q_{dsw}(x,t)=q_l+z_d-q_r
		+\frac{2(q_r-z_d)(q_r-q_l)\mathrm{cn}^2\left(\sqrt{q_r^2-q_l^2}((x-x_0)+\mathcal{V}t)|m\right)}{(q_r-q_l)-(q_r-z_d)\mathrm{sn}^2\left(\sqrt{q_r^2-q_l^2}((x-x_0)+\mathcal{V}t)|m\right)},
	\end{equation}
	with
	\begin{equation}
		\begin{aligned}
			x_0= \int_{q_l}^{z_d}\frac{\ln(a_+(\zeta)a_-^*(\zeta))\zeta }{\mathcal{R}_+(\zeta)}d\zeta,\quad\mathcal{V}=q_l^2+q_r^2+z_d^2,\quad m^2=\frac{q_r^2-z_d^2}{q_r^2-q_l^2}, 
		\end{aligned}
	\end{equation}
	where $z_d$ is determined by the equation (\ref{zd}), and $a(z)$ and $a^*(z)$ are defined in (\ref{a}).
    \par
	{\em (3)} In the right-plane wave {\em (RPW)} region, i.e., Region {\em \RNum{3}} for $x>-4q_l^2t-2q_r^2t+\varepsilon$, the solution is asymptotically described by
	\begin{equation}
		q(x,t)=q_r+\mathcal{O}(e^{-ct}),
	\end{equation}
	where $c>0$ is a positive real constant. 
\end{thm}
\par
In fact, if $z_d$ is a constant, the leading-order term $q_{dsw}(x,t)$ in (\ref{qdsw}) in Region {\RNum{2}} is the exact traveling wave periodic solution of the defocusing mKdV equation (\ref{mkdv}) in Appendix \ref{Appendix-A}. However, since the initial data (\ref{initial}) considered in this work is a sharp step function, $z_d$ has to be a movable Riemann invariant. So the asymptotic solution (\ref{shock}) behaves as a dispersive shock wave which is expressed by a modulated Jacobi elliptic functions. 
\par
To further verify the correctness of the asymptotic formulas in Theorem \ref{main}, it is necessary to compare the curves of the leading-order terms in Theorem \ref{main} with the corresponding numerical solutions for sufficiently large $t$ with special parameters $q_r$ and $q_l$. As shown in Figure \ref{compare}, perfect agreements are observed for two groups of parameters: (a) $q_l=0$, $q_r=0.5$;  (b) $q_l=0.2$, $q_r=0.8$. 
\par
\begin{figure}[ht]
	\centering
	\begin{tikzpicture}[scale=1.3]
		\draw[white, fill=blue!60] (-4.8,4.6) -- (-4.8,0.2) -- (0,0) -- cycle;
		\draw[white, fill=orange!60] (-4.6,4.8) -- (-2.1,4.8) -- (0,0) -- cycle;
		\draw[white, fill=blue!20!] (-1.9,4.8) -- (4.8,4.8) -- (4.8,0) -- (0,0) -- cycle;
		\draw[line width=1pt] (0,0) -- (-4.8,4.8) node[pos=1,above,font=\small]{$\xi=\frac{q_l^2}{2}-q_r^2$};
		\draw[line width=1pt] (0,0) -- (-2,4.8) node[pos=1,above,font=\small]{$\xi=-\frac{q_l^2}{3}-\frac{q_r^2}{6}$};
		\node[font=\small, text width=4cm, align=center] at (2.4,2.4) {\textbf{Region~\RNum{3} \\ Right-plane wave}};
		\node[font=\small, text width=4cm, align=center] at (-3.5,1.5) {\textbf{Region~\RNum{1}\\Left-plane wave}};
		\node[font=\small, text width=4cm, align=center] at (-2.6,4) {\textbf{Region~\RNum{2}\\Dispersive\\shock-wave}};
		\draw[gray,dashed,line width=1pt] (0,0) -- (-0.5,4.8) node[pos=1,above,font=\small]{$\xi=-\frac{q_r^2}{6}$};
		\draw[line width=1pt,-{Stealth[length=1.5mm, width=1.5mm]}] (-5,0) -- (5,0) node[pos=0.5,below]{$O$} node[pos=1,below]{$x$};
		\draw[line width=1pt,-{Stealth[length=1.5mm, width=1.5mm]}] (0,0) -- (0,5) node[pos=1,above]{$t$};
	\end{tikzpicture}
	\caption{The region distributions of asymptotic solution on the upper $(x,t)$-half-plane.}
	\label{region}
\end{figure}
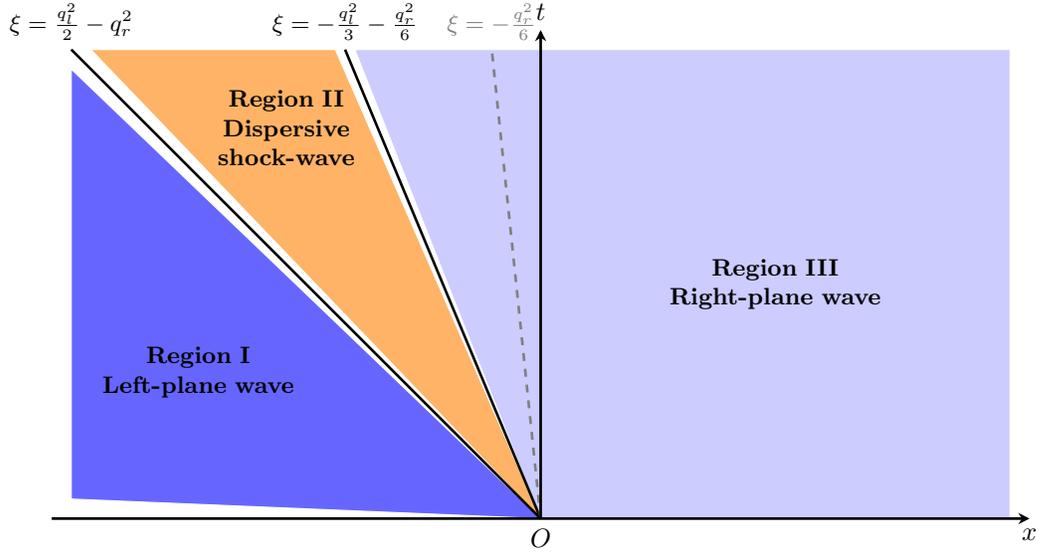
\par
\begin{figure}[htbp]
	\centering
	\subfigure[$q_l=0$, $q_r=0.5$]{
		\begin{minipage}[b]{0.45\textwidth}  % Adjusted width to fit both images
			\centering
			\includegraphics[width=\linewidth]{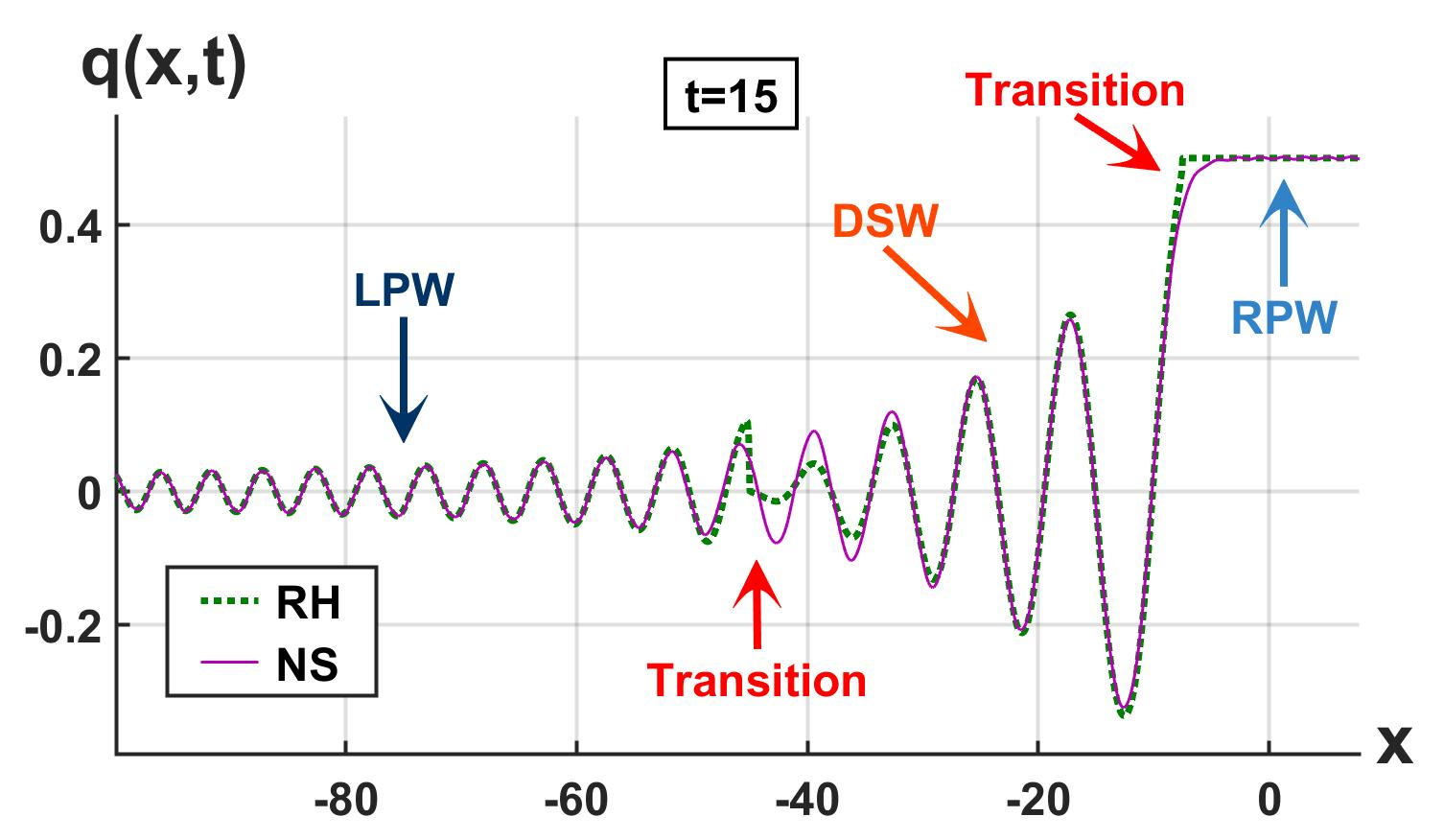}  % Use \linewidth to fit the minipage
		\end{minipage}
	}
	\hfill  % Adds horizontal space between subfigures
	\subfigure[$q_l=0.2$, $q_r=0.8$]{
		\begin{minipage}[b]{0.45\textwidth}  % Adjusted width to fit both images
			\centering
			\includegraphics[width=\linewidth]{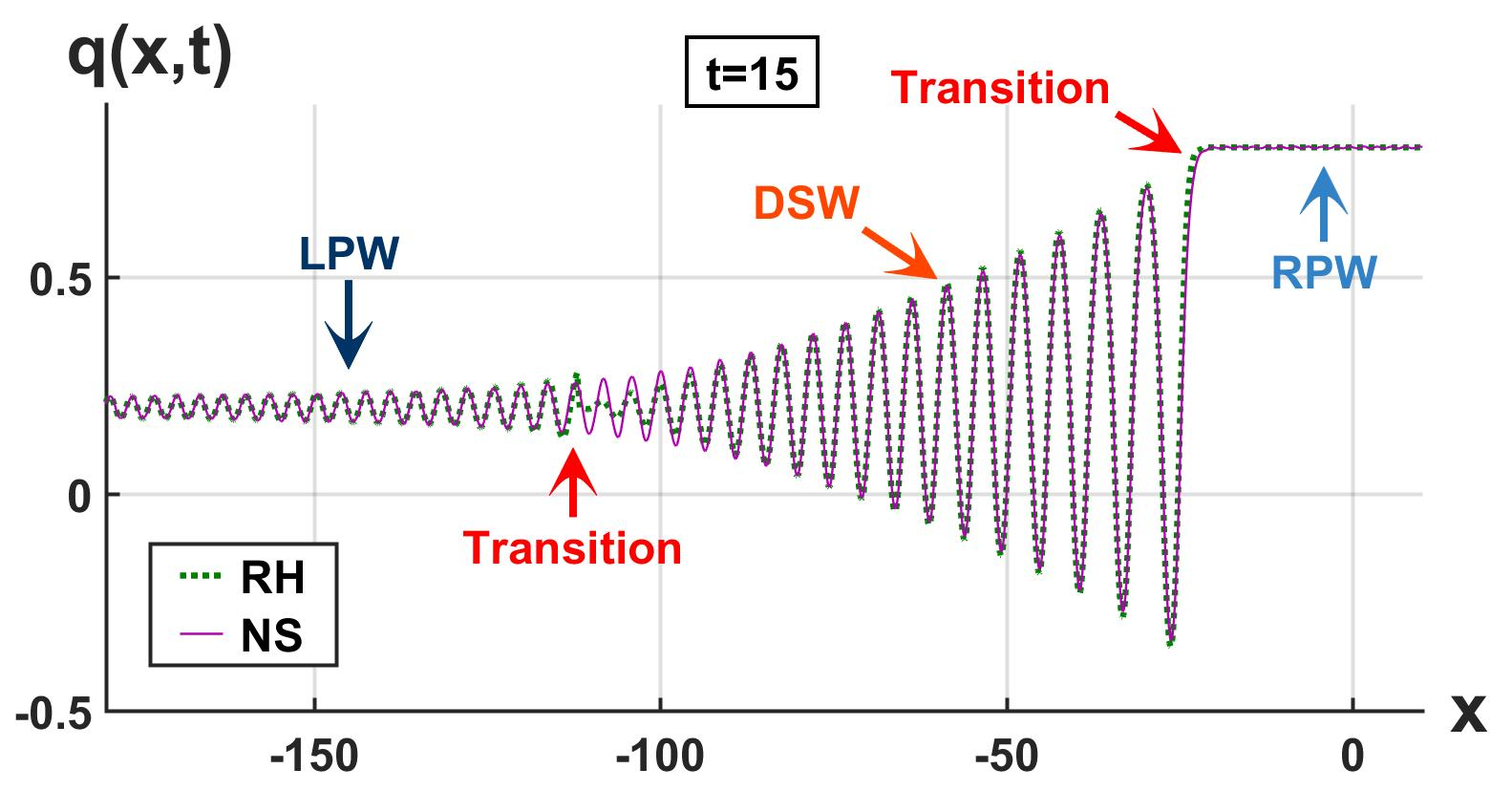}
		\end{minipage}
	}
	\caption{Comparisons of solution to the initial-value problem (\ref{mkdv}) with  (\ref{initial}) obtained from Riemann-Hilbert method (RH) and numerical simulations (NS) for different parameters.}  % Main caption
	\label{compare}  % Label for referencing
\end{figure}
\par
\begin{rmk}
	When the parameter $q_l=0$, the interval $[-q_l,q_l]$ collapses to a single point and no longer contributes to the jump condition in Riemann-Hilbert Problem \ref{rhp} below. In this case, the asymptotic solutions (\ref{left-plane}) and (\ref{shock}) in Theorem \ref{main} are simplified into: 
	\begin{equation}\label{Ann-1993}
		q(x,t)=\sqrt{\frac{\nu(\xi)}{3t(-\xi)^{1/2}}}\mathrm{cos}\left[16t\left(-\xi\right)^{\frac{3}{2}}-\nu(\xi)\ln \left(192t\left(-\xi\right)^{3/2}\right)+\phi(\xi)\right]+\mathcal{O}(t^{-1}\ln t),
	\end{equation}
	and 
	\begin{equation}\label{littledsw}
		q(x,t)=q_r-z_d+\frac{2q_r(q_r-z_d)\mathrm{cn}^2\left(q_r(x-x_0+(q_r^2+z_d^2)t)|m\right)}{q_r-(q_r-z_d)\mathrm{sn}^2\left(q_r(x-x_0+(q_r^2+z_d^2)t)|m\right)}+\mathcal{O}(t^{-1}), 
	\end{equation}
	respectively, where 
	\begin{equation}
		z_l=(-\xi)^{1/2},\quad x_0=\int_{0}^{z_d}\frac{\ln\left(i\frac{\beta_r(z)-\beta_r(z)^{-1}}{4}\right)}{\sqrt{(\zeta^2-z_d^2)(\zeta^2-q_r^2)}},\quad m^2=\frac{q_r^2-z_d^2}{q_r^2}.
	\end{equation}
	In fact, the asymptotic expression (\ref{Ann-1993}) has been given by Deift and Zhou early in 1993 \cite{DZ1993} for the defocusing mKdV equation (\ref{mkdv}) with rapidly decaying initial value. The leading-order term in (\ref{littledsw}) can also be derived directly from the auxiliary Riemann-Hilbert Problem \ref{rhp_w} by taking parameters $\tilde{q}=q_r$, $\tilde{d}=z_d$ and $q_l=0$. The profile of $q(x,t)$ for $t=15$ in this case is provided in Figure \ref{compare}(a).
\end{rmk}
\par
\begin{rmk}\label{transition}
	The white areas in Figure \ref{region}, corresponding to $6q_l^2t-12q_r^2t-\varepsilon<x<6q_l^2t-12q_r^2t+\varepsilon$ and $-4q_l^2t-2q_r^2t-\varepsilon<x<-4q_l^2t-2q_r^2t+\varepsilon$, represent the transition regions of the asymptotic solution \cite{Lenells2025}. Geometrically, taking the limit $\varepsilon\to 0$ would collapse these regions into rays, but this compression would adversely affect the boundary behaviors in adjacent regions, as clearly illustrated in Figure \ref{compare}. In fact, the complete asymptotic analysis for these transition regions proves to be prohibitively tedious (see, e.g., Ref. \cite{Minakov2019} for the case of focusing mKdV equation with step-like initial data), and thus we place it beyond the scope of the present study. 
\end{rmk}
\par
\subsection{Paper organization and notations}\
\par
\hspace{1em} In Section \ref{section2}, the inverse scattering transform is employed to reconstruct the solution to the defocusing mKdV equation (\ref{mkdv}) with initial data (\ref{initial}) by formulating an appropriate Riemann-Hilbert problem (RHP). In Section \ref{section3}, we construct the ``g"-functions required for the analysis of the RHP \ref{rhp} and demonstrate that their dynamics is governed by the Whitham equations \cite{El-SIAM-2017,AMK-2004} of the defocusing mKdV equation (\ref{mkdv}). Finally, in Sections \ref{section4}-\ref{section6}, the Deift-Zhou nonlinear steepest descent method is applied to derive the long-time asymptotics of the solution of the Riemann-Hilbert Problem \ref{rhp}, which establishes the results in Theorem \ref{main}. 
\par
We begin by establishing our notational conventions. The Pauli matrices, which play a fundamental role in asymptotic analysis, are defined as 
\begin{equation}\label{Pauli-matrices}
	\sigma_1 = \begin{pmatrix}0&1\\1&0\end{pmatrix}, \quad
	\sigma_2 = \begin{pmatrix}0&-i\\i&0\end{pmatrix}, \quad
	\sigma_3 = \begin{pmatrix}1&0\\0&-1\end{pmatrix}.
\end{equation}
These matrices facilitate the following matrix operations: 
\begin{equation}
	f^{\sigma_3} = \begin{pmatrix}f&0\\0&f^{-1}\end{pmatrix}, \quad
	f^{\widehat{\sigma}_3}A = f^{\sigma_3}Af^{-\sigma_3}, 
\end{equation}
where $f$ denotes an arbitrary scalar function and $A$ represents a 
$2\times 2$ matrix. 
\par
For a complex variable $z$, $\Im z$ and $\bar{z}$ signify its imaginary part and complex conjugation. For a complex-valued function $f(z)$, define the Schwartz conjugation as $f^*(z):= \overline{f(\bar{z})}$. Given a piecewise smooth oriented contour $\Gamma \subset \mathbb{C}$ and a scalar or matrix function $f(z)$ holomorphic in $\mathbb{C}\backslash\Gamma$, $f_\pm(z)$ denote the non-tangential boundary values from the left (+) and right (-) sides of $\Gamma$ when they exist. Finally, for any vector $\bm{z} = (z_1, z_2, \cdots, z_G) \in (\mathbb{R}_{\geq 0})^G$, introduce the function
\begin{equation}
	\mathcal{R}(z;\bm{z}) = \left(\prod_{i=1}^G(z+z_i)(z-z_i)\right)^{1/2},
\end{equation}
defined by finite branch cuts along the real axis and satisfied the asymptotics $\mathcal{R}(z;\bm{z}) \sim z^G$ for $z\to \infty$.

\section{Preliminaries}\label{section2}

This section builds the Riemann-Hilbert problem and reconstruction formula associated with the Cauchy problem (\ref{mkdv}) and (\ref{initial}).

\subsection{The Jost solutions of Lax pair} \

As a member of AKNS hierarchy \cite{AKNS_1973}, the defocusing mKdV equation (\ref{mkdv}) admits a Lax pair (i.e., Zakharov-Shabat spectrum problem) in the form of an overdetermined system of linear ordinary equations
\begin{equation}\label{lax}
	\begin{cases}
		\Phi_x(x,t;z)+iz\sigma_3\Phi(x,t;z)=Q(x,t)\Phi(x,t;z),\\
		\Phi_t(x,t;z)+4iz^3\sigma_3\Phi(x,t;z)=\tilde{Q}(x,t;z)\Phi(x,t;z),
	\end{cases}
\end{equation}
for a $2\times2$ matrix-valued function $\Phi(x,t;z)$, where $z\in\mathbb{C}$ is the spectral variable and
\begin{equation}\label{Q}
	Q(x,t)=q(x,t)\cdot \sigma_1,\quad
	\tilde{Q}(x,t;z)=4z^2Q+2iz\sigma_3(Q_x-Q^2)+2Q^3-Q_{xx}.
\end{equation}
That is, the Lax pair (\ref{lax}) satisfies the compatibility condition $\Phi_{xt}=\Phi_{tx}$ if and only if the potential function $q(x,t)$ solves the defocusing mKdV equation (\ref{mkdv}).  
\par
To apply the inverse scattering transform to the Cauchy problem (\ref{mkdv}) and (\ref{initial}), define two matrix-valued Jost solutions of the Lax pair (\ref{lax}), which satisfy the asymptotic conditions
\begin{equation}	\Phi_l(x,t;z)=\Phi_l^p(x,t;z)+o(1), \quad x\to-\infty, \quad \Im z=0, 
\end{equation}
\begin{equation}
\Phi_r(x,t;z)=\Phi_r^p(x,t;z)+o(1),\quad x\to+\infty,\quad \Im z=0, 
\end{equation}
where $\Phi_j^p(x,t;z)$ $(j=l, r)$ represent the solutions to the coupled systems 
\begin{equation}
	\begin{cases}
		\Phi_{j,x}^{p}(x,t;z)+iz\sigma_3\Phi_j^{p}(x,t;z)=Q_j^{p}\Phi_j^{p}(x,t;z),\\
		\Phi_{j,t}^{p}(x,t;z)+4iz^3\sigma_3\Phi_j^{p}(x,t;z)=\tilde{Q}_j^{p}(z)\Phi_j^{p}(x,t;z),
	\end{cases}
\end{equation}
where $j=l, r$ and
\begin{equation}
	Q_j^{p}=q_j \cdot \sigma_1,\quad \tilde{Q}_j^{p}(z)=4z^2Q^{p}-2iz\sigma_3(Q^{p})^2+2(Q^{p})^3. 
\end{equation} 
Explicitly, for $j=l,r$, it follows that
\begin{equation}
\Phi_j^p(x,t;z)=\mathcal{E}_j(z)e^{-i(\mathcal{R}_j(z)x+\Omega_j(z)t)\sigma_3},
\end{equation}
where 
\begin{equation}
	\mathcal{R}_j(z):=\mathcal{R}(z;q_j)=(z^2-q_j^2)^{1/2}, \quad \Omega_j(z)=2(2z^2+q_j^2)\mathcal{R}_j(z),
\end{equation}
and 
\begin{equation}\label{epsilon}
	\mathcal{E}_j(z)=\frac{1}{2}\begin{pmatrix}
		\beta_j(z)+\frac{1}{ \beta_j(z)}&i( \beta_j(z)-\frac{1}{ \beta_j(z)})\\-i( \beta_j(z)-\frac{1}{ \beta_j(z)})& \beta_j(z)+\frac{1}{ \beta_j(z)}
	\end{pmatrix}, 
\end{equation}
with 
\begin{equation}
	\beta_j(z):\mathbb{C}\backslash[-q_j,q_j]\rightarrow\mathbb{C}, \quad \beta_j(z)=(\frac{z-q_j}{z+q_j})^{1/4}. 
\end{equation}
Furthermore, it is easy to show that as $z\to\infty$, one has 
\begin{equation}
	\mathcal{R}_j(z)=z+\mathcal{O}(z^{-1}), \quad \beta_j(z)=1+\mathcal{O}(z^{-1}). 
\end{equation}
\par
\subsection{Direct scattering of sharp step initial data}\
\par
For the general step-like initial data, one can prove the existence and analytic extension (in $z$) theorems for the Jost functions. However, for the sharp step initial data given by (\ref{initial}), the Jost functions can be written explicitly as 
\begin{equation}\label{jl}
	\Phi_{l}(x;z)=\begin{cases}
		\mathcal{E}_l(z) e^{-i\mathcal{R}_l(z)x\sigma_3}, \quad x<0, \\
		\mathcal{E}_r(z)e^{-i\mathcal{R}_r(z)x\sigma_3}\mathcal{E}_r^{-1}(z)\mathcal{E}_l(z), \quad x>0,
	\end{cases}
\end{equation}

\begin{equation}\label{jr}
	\Phi_r(x;z)=\begin{cases}
		\mathcal{E}_l(z)e^{-i\mathcal{R}_l(z)x\sigma_3}\mathcal{E}_l^{-1}(z)\mathcal{E}_r(z), \quad x<0, \\
		\mathcal{E}_r(z)e^{-i\mathcal{R}_r(z)x\sigma_3}, \quad x>0.
	\end{cases}
\end{equation}
In addition, it is easy to show that: 

\begin{lem}
	For $j\in \{l,r\}$, the Jost functions defined by (\ref{jl}) and (\ref{jr}), with their columns expressed by $\Phi_j^{(k)}~(k=1,2)$, have the following properties: \\
	(a) $\det\text{ }\Phi_j(x;z)=1$. \\
	(b) $\Phi_j(x;z)$ is analytic for $z\in \mathbb{C} \backslash [-q_j,q_j]$. \\
	(c) $\Phi_j(x;z)e^{i\mathcal{R}_j(z)x\sigma_3}=I+\mathcal{O}(z^{-1})$ as $z\to \infty$. \\ 
	(d) For $z\in (-q_j,q_j)$, $\Phi_j(x;z)$ takes continuous boundary values satisfying
	\begin{equation}
		\Phi_{j-}^{-1}(x;z)\Phi_j(x;z)=\mathcal{E}_{j-}^{-1}(z)\mathcal{E}_{j+}(z)=-i\sigma_2.
	\end{equation} 
	(e) $(z\mp q_j)^{1/4}\Phi_j(x;z)$ are bounded as $z\to\pm q_j$. \\
	(f) Symmetries: \begin{equation}
		\begin{aligned}
			\Phi_j^{(1)}(x;-z)=\overline{\Phi_j^{(1)}(x;\bar{z})}=\sigma_1\Phi_j^{(2)}(x;z), \quad \overline{\Phi_j(x;-\bar{z})}=\Phi_j(x;z). 
		\end{aligned}
	\end{equation}
\end{lem}
\par
Since the matrix-value functions $\Phi_l(x,t;z)$ and $\Phi_r(x,t;z)$ are two solutions of the Lax pair (\ref{lax}), they are linearly dependent, that is, there exists a scattering matrix $T(z)$ independent of $x$ and $t$, such that 
\begin{equation}
	T(z):=\Phi_r^{-1}(x,t;z)\Phi_l(x,t;z)=\Phi_r^{-1}(x;z)\Phi_l(x;z)=\mathcal{E}_r^{-1}(z)\mathcal{E}_l(z)=\begin{pmatrix}
		a(z)&b^{*}(z)\\
		b(z)&a^{*}(z)
	\end{pmatrix},
\end{equation}
where
\begin{equation}\label{a}
	\begin{aligned}
		a(z)&=a^*(z)=\frac{\beta_l(z)\beta_r^{-1}(z)+\beta_l^{-1}(z)\beta_r(z)}{2}, \\
		b(z)&=-b^*(z)=\frac{\beta_l(z)\beta_r^{-1}(z)-\beta_l^{-1}(z)\beta_r(z)}{2i},\\
		r(z)&:=\frac{b(z)}{a(z)}=-i\frac{\beta_l(z)^2-\beta_r(z)^2}{\beta_l(z)^2+\beta_r(z)^2}=-\frac{b^*(z)}{a^*(z)}=:-r^*(z). 
	\end{aligned}
\end{equation}
By calculation, it is immediate that: 
\par
\begin{lem}\label{abr}
	The spectral functions $a(z), b(z)$ and $r(z)$ have the following properties: \\
	(a) Functions $a(z)$, $b(z)$ and $r(z)$ are holomorphic for $z\in \mathbb{C}\backslash([-q_r,-q_l]\cup[q_l,q_r])$, where $r(z)$ satisfies $|r(z)|<1$, and they can be continuously extended to the boundary except at the points $\pm q_l,\pm q_r$. \\
	(b) Symmetries:
	\begin{equation}
		a(z)=a(-z)=\overline{a(-\bar{z})},\quad r(z)=-r(-z)=\overline{r(-\bar{z})}. 
	\end{equation}
	(c) On $(-q_r,-q_l)\cup(q_l,q_r)$, we have
	\begin{equation}
		a_+(z)=b_-^*(z), \quad b_+(z)=a_-^*(z), \quad r_+(z)=\frac{1}{r_-^*(z)}, \quad |r_{\pm}(z)|=1. 
	\end{equation}
	(d) If $q_r\ge q_l\ge0$, the function $a(z)$ defined by (\ref{a}) has no zeros in the complex plane, which indicates that no soliton appears in this case. 
\end{lem}
\par
\subsection{Riemann-Hilbert problem in the generic case}\
\par
Construct the piecewise analytic function
\begin{equation}
	\mathcal{M}(x,t;z)=\begin{cases}
		(\frac{\Phi^{(1)}_l(x,t;z)}{a(z)}, \Phi^{(2)}_r(x,t;z))e^{it\theta(\xi,z)\sigma_3}, z\in\mathbb{C}^+, \\ 
		(\Phi^{(1)}_r(x,t;z),\frac{\Phi^{(2)}_l(x,t;z)}{a^*(z)})e^{it\theta(\xi,z)\sigma_3}, z\in\mathbb{C}^-,
	\end{cases}
\end{equation}
where 
\begin{equation}
	\theta(\xi;z)=4z^3+12\xi z, \quad \xi=\frac{x}{12t}. 
\end{equation}
Then the function $\mathcal{M}(x,t;z)$ satisfies the following RHP:
\par 
\begin{rhp}\label{rhp}
	Find a $2\times2$ complex matrix function $\mathcal{M}(x,t;z)$ that satisfies: \\
	(1) $\mathcal{M}(x,t;z)$ is holomorphic for $z\in\mathbb{C}\backslash\mathbb{R}$. \\
	(2) For $z\in\mathbb{R}$, the function $\mathcal{M}(x,t;z)$ takes continuous boundary values $\mathcal{M}_{\pm}(x,t;z)$ which satisfy the jump relation
	\begin{equation}
		\mathcal{M}_+(x,t;z)=\mathcal{M}_-(x,t;z)\mathcal{J}(x,t;z),
	\end{equation}
	where
	\begin{equation}\label{v}
		\mathcal{J}(z)=\begin{cases}
			\begin{aligned}
				&\begin{pmatrix}
					1-r(z)r^*(z)&-r^*(z)e^{-2it \theta(\xi;z)}\\r(z)e^{2it \theta(\xi;z)}&1
				\end{pmatrix}, &z\in(-\infty,-q_r)\cup(q_r,\infty), \\
				&\begin{pmatrix}
					(a_+(z)a_-^*(z))^{-1}&-e^{-2it \theta(\xi;z)}\\e^{2it \theta(\xi;z)}&0
				\end{pmatrix}, &z\in (-q_r,-q_l)\cup(q_l,q_r), \\ 
				&\begin{pmatrix}
					0&-e^{-2it \theta(\xi;z)}\\e^{2it \theta(\xi;z)}&0
				\end{pmatrix}, &z\in (-q_l,q_l).
			\end{aligned}
		\end{cases}
	\end{equation}
	(3) $\mathcal{M}(z)=I+\mathcal{O}(z^{-1})$, as $z\to\infty$ in $\mathbb{C}\backslash\mathbb{R}$. \\
	(4) Symmetries: 
	\begin{equation}
		\mathcal{M}(z)=\overline{\mathcal{M}(-\bar{z})}=\sigma_1\mathcal{M}(-z)\sigma_1=\sigma_1\overline{\mathcal{M}(\bar{z})}\sigma_1.
	\end{equation}
\end{rhp}
\par
The existence and uniqueness of the solution to RHP \ref{rhp} follows from the vanishing lemma for Schwartz-symmetric RHPs \cite{Zhou1989}. Moreover, let $\mathcal{M}_{12}(x,t;z)$ denote the $(1,2)$-entry of $\mathcal{M}(x,t;z)$. By combining the Lax pair (\ref{lax}) with the asymptotic expansion of $\Phi_j(x,t;z)e^{it\theta(\xi;z)\sigma_3}$ for $j\in\{r,l\}$ as $z\to\infty$, the reconstruction formula of potential function $q(x,t)$ is obtained 
\begin{equation}\label{construct}
	q(x,t)=2i\lim_{z\to\infty}z\mathcal{M}_{12}(x,t;z), \quad (x,t)\in \mathbb{R}\times\mathbb{R}_{\ge 0}. 
\end{equation}

\begin{rmk}\label{a(0)=0}
	For the cases of $-q_r \ge q_l \ge 0$ and $-q_r \ge q_l \ge 0$ (excluding the trivial case $q_r=q_l=0$), a corresponding direct scattering analysis shows that $z=0$ is a simple zero of the scattering coefficient $a(z)=[\beta_l(z)\beta_r(z)+\beta_l^{-1}(z)\beta_r^{-1}(z)]/2$. Consequently, the matrix-valued function $\mathcal{M}(x,t;z)$ has a simple pole at $z=0$. As established in previous work \cite{Jenkins2016}, this singularity induces the formation of a kink soliton in the solution $q(x,t)$ of the defocusing mKdV equation (\ref{mkdv}), see Figure \ref{kink} for details. 
\end{rmk}
\par

\subsection{Possible deformations of jump matrices}\

In what follows, we present several key matrix factorizations that will be employed in the subsequent contour deformations onto the steepest descent paths. These factorizations are organized according to their respective intervals of application. For brevity, the off-diagonal exponential factors have been omitted, which may be reintroduced through left and right multiplication by the appropriate diagonal matrices. 
\par
For $z\in(-\infty,-q_r)\cup(q_r,+\infty)$, the matrix factorizations below hold
\begin{equation}
	\begin{aligned}
		\begin{pmatrix}
			1-r(z)r^*(z)&-r^*(z)\\r(z)&1
		\end{pmatrix}=&
		\begin{pmatrix}
			1&-r^*(z)\\0&1
		\end{pmatrix}
		\begin{pmatrix}
			1&0\\r(z)&1
		\end{pmatrix}\\
		=&\begin{pmatrix}
			1&0\\ \frac{r(z)}{1-r(z)r^*(z)}&1
		\end{pmatrix}
		(1-r(z)r^*(z))^{\sigma_3}
		\begin{pmatrix}
			1&\frac{-r^*(z)}{1-r(z)r^*(z)}\\0&1
		\end{pmatrix}. 
	\end{aligned}
\end{equation}
\par
For $z\in(-q_r,-q_l)\cup(q_l,q_r)$, we have
\begin{equation}
	\begin{aligned}
		\begin{pmatrix}
			\frac{1}{a_+a_-^*}&-1\\1&0
		\end{pmatrix}=&\begin{pmatrix}
			1&-r_-^*\\0&1
		\end{pmatrix}
		\begin{pmatrix}
			0&-1\\1&0    
		\end{pmatrix}\
		\begin{pmatrix}
			1&0\\r_+&1
		\end{pmatrix}\\
		=& \begin{pmatrix}
			1&0\\ \frac{r_-(z)}{1-r_-(z)r^*_-(z)}&1
		\end{pmatrix}
		(a_+(z)a_-^*(z))^{-\sigma_3}
		\begin{pmatrix}
			1&\frac{-r_+^*(z)}{1-r_+(z)r_+^*(z)}\\0&1
		\end{pmatrix}. 
	\end{aligned}
\end{equation}
\par
For $z\in(-q_l,q_l)$, the matrix factorizations below hold
\begin{equation}
	\begin{aligned}
		\begin{pmatrix}
			0&-1\\1&0
		\end{pmatrix}=&
		\begin{pmatrix}
			1&-r^*(z)\\0&1
		\end{pmatrix}
		\begin{pmatrix}
			0&-1\\1&0
		\end{pmatrix}
		\begin{pmatrix}
			1&0\\r(z)&1
		\end{pmatrix}\\
		=&\begin{pmatrix}
			1&0\\ \frac{r(z)}{1-r(z)r^*(z)}&1
		\end{pmatrix}
		\begin{pmatrix}
			0&-1\\1&0
		\end{pmatrix}
		\begin{pmatrix}
			1&\frac{-r^*(z)}{1-r(z)r^*(z)}\\0&1
		\end{pmatrix}. 
	\end{aligned}
\end{equation}
\par

\section{Constructing the $g$-functions of self-similar wave motion}\label{section3}

To conduct the long-time asymptotic analysis for RHP \ref{rhp}, it is reasonable to employ the Deift-Zhou nonlinear steepest descent method \cite{DZ1993}, which serves as the fundamental tool for investigating oscillatory RHPs. The initial step involves constructing a ``g"-function, designed to renormalize the oscillatory or exponentially growing factors $e^{\pm it\theta(\xi;z)}$ in the jump matrices while maintaining asymptotic consistency with $\theta(\xi;z)$ at infinity, that is
\begin{itemize}\label{g-theta}
	\item $g(\xi;z)=\theta(\xi;z)+\mathcal{O}(z^{-1}), \quad {\rm as}~|z|\to\infty. $
\end{itemize}
\par
Following the approach proposed in \cite{Jenkins2015}, a general method to construct the $g$-function of genus $2G$ $(G\ge 0)$ for the defocusing mKdV equation is proposed. Based on this framework, we explicitly derive the expressions for the $g$-function of genus zero and genus two, which are admissible under self-similar motion via the method in Ref. \cite{Grava2002}. The detailed process of constructing $g$-function are presented below. 
\par
Suppose that a set of points $\{z_1,z_2,\cdots,z_{2G+1}\}$ is given on the real axis such that $z_{2G+1}>z_{2G}>\cdots>z_1>0$. Denote the set $\mathcal{I}= \bigcup_{k=-G}^{G}\mathcal{I}_k$, where 
\begin{equation}
	\mathcal{I}_0=(-z_1,z_1), \quad  \mathcal{I}_k=(z_{2k},z_{2k+1}),\quad \mathcal{I}_{-k}=(-z_{2k+1},-z_{2k}), \quad k=1,2,\cdots,G, 
\end{equation} 
which are known as ``bands". At the same time, define the set $\mathcal{I}'=\bigcup_{k=1}^{G}\left(\mathcal{I}'_k\cup\mathcal{I}'_{-k}\right)$, where 
\begin{equation}
	\mathcal{I}'_k=(z_{2k-1},z_{2k}),\quad\mathcal{I}'_{-k}=(-z_{2k},-z_{2k-1}), \quad k=1,2,\cdots,G, 
\end{equation}
which are known as ``gaps". Assume that there exists a set of constants $\{d_0\}\cup \{d_{k},d_{-k}\}_{k=1}^{G}$ such that the $g$-function satisfies:
\begin{itemize}
	\item   $g(\xi;z)$ is holomorphic for $z\in \mathbb{C}\backslash\bar{\mathcal{I}}$. 
	\item $g_+(\xi;z)+g_-(\xi;z)=d_i(\xi) \text{ for } z\in\mathcal{I}_i,\text{ } i=-G,-G+1, \cdots,0, \cdots, G-1, G$. 
	\item For $z\in \mathbb{C}\backslash\bar{\mathcal{I}}$, the Schwartz symmetry $g(\xi;z)=g^*(\xi;z)$ holds. 
\end{itemize}
Furthermore, to preserve the symmetry of the RHP \ref{rhp} under the gauge transformation $\widetilde{\mathcal{M}}(x,t;z)=\mathcal{M}(x,t;z) e^{it(g(\xi;z)-\theta(\xi;z))}$, the following symmetry conditions on $g(\xi;z)$ are imposed: 
\begin{itemize}
    \item For $z\in\mathbb{C}$, the symmetry $g(\xi;z)=g(\xi;-\bar{z})=-g(\xi;-z)=-g(\xi;\bar{z})$ holds.
\end{itemize}
\par
In order to achieve this conditions, introduce a compact Riemann surface of genus $2G$ for $G\geq 0$ as 
\begin{equation}
	\mathcal{S}_G:=\{P=(z,\mathcal{R}_G):~ (\mathcal{R}_G)^2=\prod_{i=1}^{2G+1}(z^2-z_j^2) \}, 
\end{equation}
equipped with the projective map $\pi:\mathcal{S}_G\to\mathbb{C}P^1$ given by $\pi(P)=z$, which makes $\mathcal{S}_G$ to be a two-sheeted covering of $\mathbb{C}P^1$. Select a basis $\{a_j,b_j,a_{-j},b_{-j}\}_{j=1}^{G}$ for the first homology group $H_1(\mathcal{S}_G)$, such that (refer to Figure \ref{cycle}): $a_j$ ($a_{-j}$, respectively) lies entirely on the upper sheet and encircles $\bar{\mathcal{I}}_{-G}$ ($\bar{\mathcal{I}}_G$, respectively) counterclockwise (clockwise, respectively), while $b_j$ ($b_{-j}$, respectively) starts from $\mathcal{I}_{0}$, travels counterclockwise (clockwise, respectively) through the upper sheet to $\mathcal{I}_{-G}$ ($\mathcal{I}_G$, respectively), and passes through $\mathcal{I}_{-G}$ ($\mathcal{I}_G$, respectively) to the lower sheet, then  returns to the starting point via the lower sheet. 
\par
\begin{figure}[ht]
	\centering
	\begin{tikzpicture}[scale=0.9]
		% cut
		\draw[line width=5pt,gray] (1,0) -- (-1,0) node[pos=0,below=-0.5mm,black]{$z_1$} node[pos=1,below=-0.5mm,black]{$-z_1$};
		\draw[line width=5pt,gray] (-2,0) -- (-4,0) node[pos=0,below=-0.5mm,black]{$-z_2$} node[pos=1,below=-0.5mm,black]{$-z_3$};
		\fill[black] (-4.8,0) circle (2pt);   % 直径4pt
		\fill[black] (-5,0) circle (2pt);
		\fill[black] (-5.2,0) circle (2pt);
		\draw[line width=5pt,gray] (-6,0) -- (-8,0) node[pos=0,below=-0.5mm,black]{$-z_{2G}$} node[pos=1,below=-0.5mm,black]{$-z_{2G+1}$};
		\draw[line width=5pt,gray] (2,0) -- (4,0) node[pos=0,below=-0.5mm,black]{$z_2$} node[pos=1,below=-0.5mm,black]{$z_3$};
		\fill[black] (4.8,0) circle (2pt);   % 直径4pt
		\fill[black] (5,0) circle (2pt);
		\fill[black] (5.2,0) circle (2pt);
		\draw[line width=5pt,gray] (6,0) -- (8,0) node[pos=0,below=-0.5mm,black]{$z_{2G}$} node[pos=1,below=-0.5mm,black]{$z_{2G+1}$};
        % a cycle 1
		\draw[line width=2pt,red!40!white] (-1.5,0) to [out=90, in=90, looseness=0.9] (-4.5,0);
		\draw[line width=2pt,red!40!white] (-4.5,0) to [out=-90, in=-90, looseness=0.9] (-1.5,0);
		\node[red!40!white] at (-3,1){$a_1$};
		\draw[red!40!white,line width=1pt,-{Stealth[length=3mm, width=2mm]}] (-3.11,0.8) -- (-3.12,0.8);
            \draw[red!40!white,line width=1pt,-{Stealth[length=3mm, width=2mm]}] (3.11,0.8) -- (3.12,0.8);
		\draw[line width=2pt,red!40!white] (1.5,0) to [out=90, in=90, looseness=0.9] (4.5,0);
		\draw[line width=2pt,red!40!white] (4.5,0) to [out=-90, in=-90, looseness=0.9] (1.5,0);
		\node[red!40!white] at (3,1){$a_{-1}$};
		% a cycle 2
		\draw[line width=2pt,red!40!white] (-5.5,0) to [out=90, in=90, looseness=0.9] (-8.5,0);
		\draw[line width=2pt,red!40!white] (-8.5,0) to [out=-90, in=-90, looseness=0.9] (-5.5,0);
		\node[red!40!white] at (-7,1){$a_G$};
		\draw[red!40!white,line width=1pt,-{Stealth[length=3mm, width=2mm]}] (-7.11,0.8) -- (-7.12,0.8);
		\draw[line width=2pt,red!40!white] (5.5,0) to [out=90, in=90, looseness=0.9] (8.5,0);
		\draw[line width=2pt,red!40!white] (8.5,0) to [out=-90, in=-90, looseness=0.9] (5.5,0);
		\node[red!40!white] at (7,1){$a_{-G}$};
		\draw[red!40!white,line width=1pt,-{Stealth[length=3mm, width=2mm]}] (7.11,0.8) -- (7.12,0.8);
        % cycle 2
		\draw[line width=2pt,blue!40!white] (-0.2,0) to [out=90, in=90, looseness=0.6] (-7,0);
		\draw[dashed,line width=2pt,blue!40!white] (-7,0) to [out=-90, in=-90, looseness=0.6] (-0.2,0);
		\draw[blue!40!white,line width=1pt,-{Stealth[length=3mm, width=2mm]}] (-3.91,1.2) -- (-3.92,1.2);
		\node[blue!40!white] at (-4,1.5){$b_G$};
		\draw[line width=2pt,blue!40!white] (0.2,0) to [out=90, in=90, looseness=0.6] (7,0);
		\draw[dashed,line width=2pt,blue!40!white] (7,0) to [out=-90, in=-90, looseness=0.6] (0.2,0);
		\draw[blue!40!white,line width=1pt,-{Stealth[length=3mm, width=2mm]}] (3.91,1.2) -- (3.92,1.2);
		\node[blue!40!white] at (4,1.5){$b_{-G}$};
        % cycle 1
		\draw[line width=2pt,blue!40!white] (-0.4,0) to [out=90, in=90, looseness=0.6] (-3,0);
		\draw[dashed,line width=2pt,blue!40!white] (-3,0) to [out=-90, in=-90, looseness=0.6] (-0.4,0);
		\draw[blue!40!white,line width=1pt,-{Stealth[length=3mm, width=2mm]}] (-1.91,0.46) -- (-1.92,0.46);
		\node[blue!40!white] at (-1.7,0.7){$b_1$};
		\draw[line width=2pt,blue!40!white] (0.4,0) to [out=90, in=90, looseness=0.6] (3,0);
		\draw[dashed,line width=2pt,blue!40!white] (3,0) to [out=-90, in=-90, looseness=0.6] (0.4,0);
		\draw[blue!40!white,line width=1pt,-{Stealth[length=3mm, width=2mm]}] (1.91,0.46) -- (1.92,0.46);
		\node[blue!40!white] at (1.7,0.7){$b_{-1}$};
	\end{tikzpicture}
	\caption{The basis $\{a_j,b_j,a_{-j},b_{-j}\}_{j=1}^{G}$ for homology group $H_1(\mathcal{S}_G)$, where $\mathcal{S}_G$ is the genus
$2G$ hyperelliptic Riemann surface.}
	\label{cycle}
\end{figure}
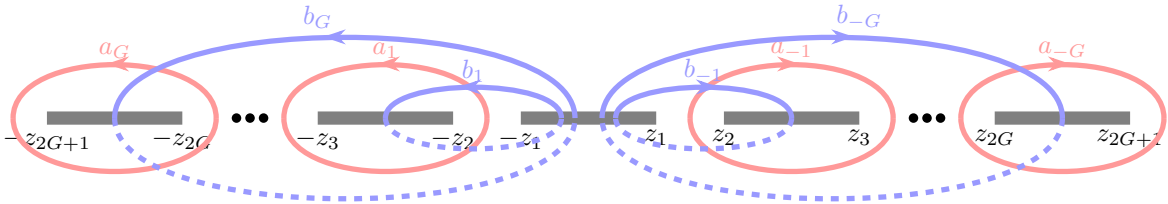
\par
Denote $\omega^{(k)}~(k=0,1)$ as Abelian differentials of the second kind on Riemann surface $\mathcal{S}_G$,  given by
\begin{equation}\label{poly}
	\begin{aligned}
		&\omega^{(k)}=\frac{\mathcal{P}_{k,G}(z;\bm{z})}{\mathcal{R}_G(z;\bm{z})}dz, \\
		&\mathcal{P}_{k,G}(z;\bm{z})= z^{2k+2G+1}+p_1 z^{2k+2G}+\cdots+p_{2k} z^{2G}+p_{k,1}z^{2G-1}+\cdots+p_{k,2G}, 
	\end{aligned}
\end{equation}
where $p_j=p_j(\bm{z})~(j=1,2,\cdots, 2k)$ are the coefficients of the expansion 
\begin{equation}\label{coe}
	\mathcal{R}_G(z;\bm{z})=\left(\prod_{i=1}^{2G+1}(z^2-z_j^2)\right)^{\frac{1}{2}}=z^{2G+1}\left(1+\frac{p_1}{z}+\cdots+\frac{p_m}{z^m}+\cdots\right), 
\end{equation}
and the $p_{k,j}=p_{k,j}(\bm{z})~(j=1,2, \cdots, 2G)$ are determined by the normalization conditions
\begin{equation}\label{normal_a}
	\oint_{a_j}\omega^{(k)}=0,\quad j=1,2, \cdots,  2G. 
\end{equation}
Additionally, the definition of Abel differential of the second kind tells us that 
\begin{equation}\label{w_infty}
	\omega^{(k)}=\pm[z^{k}+\mathcal{O}(z^{-2})]dz, \qquad {\rm as}~~P\to (\infty,\pm\infty). 
\end{equation}
\par
\begin{lem} It can be proved that $p_j~(j=1,2,\cdots, 2k)$ and $p_{k,j}~(j=1,2, \cdots, 2G)$ satisfy the following properties: 
	(a) $p_j=0$ for $j\in2\mathbb{N}+1$. 
	(b) $p_{k,j}=0$ for  $j\in 2\mathbb{N}+1$. 
\end{lem}
\par
\begin{proof}
	(a) Upon selecting a single-valued branch, the relation $\mathcal{R}_G(-z;\bm{z})+\mathcal{R}_G(z;\bm{z})=0$ holds for all $z\in\mathbb{C}\backslash\bar{\mathcal{I}}$. Substituting this into the expansion (\ref{coe}) leads to the conclusion that $p_j=0$ for $j\in2\mathbb{N}+1$. (b) Observe that the relation $ \mathcal{R}_G(-z;\bm{z})=\mathcal{R}_G(z;\bm{z})$ is satisfied for all $z\in\mathcal{I}$. By substituting this identity into (\ref{normal_a}) and employing Cramer's rule, the desired result of $p_{k,j}=0$ for  $j\in 2\mathbb{N}+1$ is obtained.
\end{proof}
\par
Now we are ready to construct a differential $dg$ which satisfies the following properties: 
\begin{itemize}
	\item $dg$ is meromorphic on $\mathcal{S}_G$ and has only poles at $(\infty, \pm\infty)$. 
	\item $dg\mp d\theta=\mathcal{O}(z^{-2})dz$, as $P\to(\infty,\pm\infty)$. 
	\item $\oint_{a_j}dg=0$ for $j=1,2, \cdots, 2G$. 
\end{itemize}
In fact, for any selection of $\bm{z}$, the three conditions above determine a meromorphic differential of the second kind, which is given by
\begin{equation}
	dg=12\omega^{(1)}+12\xi\omega^{(0)}. 
\end{equation}
Thus, define a $g$-function $g(\xi;z)$ of the form
\begin{equation}\label{gfunction}
	g(\xi;z):=\int_{z_1}^{z}dg, 
\end{equation}
which is holomorphic for $z\in \mathbb{C}\backslash\mathcal{I}$ and satisfies the jump relations
\begin{equation}
	g_+(\xi;z)+g_-(\xi;z)=\begin{cases}
		\begin{aligned}
			&\quad0,&z\in\mathcal{I}_0,\qquad\qquad\qquad\\
			&\oint_{b_{j}}dg,&z\in\mathcal{I}_{-j},j=1,2, \cdots, G, \\
			&\oint_{b_{-j}}dg=-\oint_{b_{j}}dg,&z\in\mathcal{I}_{j},j=1,2, \cdots, G, \hspace{0.5em}
		\end{aligned}
	\end{cases}
\end{equation}
by applying the symmetry of $\mathcal{R}_G(z;\bm{z})$ for $z\in\mathcal{I}$. 
\par
The symmetries of function $g(\xi;z)$ for $z\in\mathbb{C}\backslash\bar{\mathcal{I}}$ follow directly from its definition. As for its symmetry on $\mathcal{I}$, we defer the discussion to Remark \ref{symmetry_I} below. Moreover, by the asymptotic property \eqref{w_infty}, the functions $g(\xi;z)$ and $\theta(\xi;z)$ share identical behavior at infinity. Thus, $g(\xi;z)$ defined in (\ref{gfunction}) is indeed the required $g$-function. 
\par
\begin{rmk}
	In particular, if a given $z_i$ is permitted to vary within the domain $(x,t)\in\mathcal{D}\subset\mathbb{R}\times\mathbb{R}_{\ge 0}$, which is called ``soft-edge" as named in \cite{Jenkins2015} and its dynamics are governed by the condition that the differential $dg$ vanishes quadratically at $z_i$ (when treated as a branch point). This requires that $12\mathcal{P}_{1,G}(z;\bm{z})+12\xi \mathcal{P}_{0,G}(z;\bm{z})$ possesses a simple zero at $z=z_i$, or equivalently,  one has
	\begin{equation}\label{motion}
		x-v_G(\bm{z})t=0,\qquad v_G(\bm{z})=-12\cdot\frac{\mathcal{P}_{1,G}(z_i;\bm{z})}{\mathcal{P}_{0,G}(z_i;\bm{z})}. 
	\end{equation}
    Equation (\ref{motion}) establishes that the evolution of the branch point $z_i$ is characterized by the self-similar solutions of the genus-$G$ Whitham equations for the defocusing mKdV equation (\ref{mkdv}). 
\end{rmk}
\par
\begin{rmk}\label{symmetry_I}
    Regarding the symmetry of the function $g(\xi;z)$ for $z\in\mathcal{I}$, it is necessary to account for its jump conditions. If we consider the upper sheet $g_+(\xi;z)$, then $g(\xi;-z)$, $g(\xi;\bar{z})$ and $g(\xi;-\bar{z})$ correspond to $g_-(\xi;-z)$, $g_-(\xi;z)$ and $g_+(\xi;-z)$, respectively. To preserve this symmetry, choose $z_1$, one of the 
    branch points closest to the origin, as the integration origin for $dg$. (Alternatively, $-z_1$, may be selected without loss of generality, as the integral constraint $\int_{z_1}^{-z_1}dg_+=0$ ensures consistency.) This approach differs from the conventional methodology adopted in Ref. \cite{Jenkins2015} where integration begins at the rightmost branch point $z_{2G+1}$. In fact, the reverse arrangement of variables in relation (\ref{relationK}) in Section \ref{section5} further supports this choice, presenting a direct contrast to the positive ordering used in the focusing mKdV equation \cite{Minakov2020}.
\end{rmk}
\par
\subsection{Self-similar $g$-function of genus zero for $G=0$} \

For Riemann surface of genus zero, we have $\bm{z}=(-z_1,z_1)$, and the first homology group is trivial since any closed loop is contractible, then it follows from (\ref{coe}) and simple calculation that
\begin{equation}
	\mathcal{P}_{0,0}(z;\bm{z})=z,\quad
	\mathcal{P}_{1,0}(z;\bm{z})=z^3-\frac{z_1^2}{2}z. 
\end{equation}
Thus the differential $dg_0$ is given by 
\begin{equation}
	dg_0=\frac{12(z^3-\frac{z_1^2}{2}z)+12\xi z}{\mathcal{R}_0(z;\bm{z})}dz=\frac{12z(z+z^{(0)})(z-z^{(0)})}{\mathcal{R}_0(z;\bm{z})}dz, \quad z^{(0)}=\sqrt{\frac{z_1^2}{2}-\xi},
\end{equation}
where $\mathcal{R}_0(z;\bm{z})=(z^2-z_1^2)^{\frac{1}{2}}$ is cut on $(-z_1,z_1)$ and $\mathcal{R}_0(z;\bm{z})=z+\mathcal{O}(z^{-1})$ as $z\to\infty$. Hence, the $g$-function is given by
\begin{equation}\label{ggg0}
	g_0(\xi;z):=\int_{z_1}^zdg_0=(4z^2+12\xi+2z_1^2)(z^2-z_1^2)^{\frac{1}{2}}, 
\end{equation}
which satisfies all properties for $g$-function outlined in this section, where $\mathcal{I}=(-z_1,z_1)$ and $d_0(\xi)=0$. 
\par
Additionally, if $z_1$ is a soft-edge, the genus zero Whitham velocity $v_0$ in (\ref{motion}) is given by
\begin{equation}
	v_0(\bm{z})=-6z_1^2. 
\end{equation}

\subsection{Self-similar $g$-function of genus two for $G=1$}\label{g=1}\

For Riemann surface of genus two, we have $\bm{z}=(-z_3,-z_2, -z_1, z_1, z_2, z_3)$ with $ z_3>z_2>z_1>0$, and the first homology group is $H_1(\mathcal{S}_G)=\{a_1,b_2,a_{-1},b_{-1} \}$. The polynomials in (\ref{poly}) are given by 
\begin{equation}
	\mathcal{P}_{0,1}(z;\bm{z})=z^3+p_{0,1}z,\quad
	\mathcal{P}_{1,1}(z;\bm{z})=z^5-\frac{z_1^2+z_2^2+z_3^2}{2}z^3+p_{1,1}z,  
\end{equation}
where
\begin{equation}
	\begin{aligned}
		p_{0,1}&=-z_1^2-(z_3^2-z_1^2)\frac{E(m)}{K(m)}, \quad m=\frac{z_3^2-z_2^2}{z_3^2-z_1^2}, \\
		p_{1,1}&=\frac{z_1^2z_2^2+z_2^2z_3^2+z_1^2z_3^2}{3}-\frac{z_1^2+z_2^2+z_3^2}{6}\left[z_1^2+(z_3^2-z_1^2)\frac{E(m)}{K(m)}\right]. 
	\end{aligned}
\end{equation}
Here $K(m)=\int_0^{\frac{\pi}{2}}\frac{ds}{\sqrt{1-m^2\sin^2s}}$ and $E(m)=\int_0^{\frac{\pi}{2}}\sqrt{1-m^2\sin^2s} \text{ } ds$ are the complete elliptic integrals of the first and second kind, respectively. Then the differential $dg_1$ is given by
\begin{equation}
	dg_1=12 \frac{z^5+(\xi-\frac{z_1^2+z_2^2+z_3^2}{2})z^3+b z}{\mathcal{R}_1(z;\bm{z})}dz, 
\end{equation}
where  
\begin{equation}
	\begin{aligned}
		b=\frac{z_1^2z_2^2+z_2^2z_3^2+z_1^2z_3^2}{3}-\left(\xi+\frac{z_1^2+z_2^2+z_3^2}{6}\right)\left[z_1^2+(z_3^2-z_1^2)\frac{E(m)}{K(m)}\right], 
	\end{aligned}
\end{equation}
and $\mathcal{R}_1(z;\bm{z})=\left[(z^2-z_1^2)(z^2-z_2^2)(z^2-z_3^2)\right]^{\frac{1}{2}}$ is cut on $(-z_3,-z_2)\cup(-z_1,z_1)\cup(z_2,z_3)$ with $\mathcal{R}_1(z;\bm{z})=z^3+\mathcal{O}(z^{-1})$ as $z\to\infty$. 
Hence, the $g$-function is given by
\begin{equation}\label{ggg1}
	g_1(\xi;z):=\int_{z_1}^zdg_1=\int_{z_1}^z\frac{12s(s^2-z_2^2)(s^2-\mu^2)}{\mathcal{R}_1(s;\bm{z})}ds, 
\end{equation}
where $\mu$ is determined by $\mu^2=\frac{1}{2}(z_1^2-z_2^2+z_3^2)-\xi$. Consequently, the function $g_1(\xi;z)$ satisfies all properties for $g$-function outlined in this section, where $\mathcal{I}=(-z_3,-z_2)\cup(-z_1,z_1)\cup(z_2,z_3)$ and 
\begin{equation}
    d_1(\xi)=\oint_{b_1}dg=:\delta_0(\xi), \quad d_{-1}(\xi)=-\oint_{b_1}dg=-\delta_0(\xi), \quad d_{0}(\xi)=0.
\end{equation}
\par
Additionally, if $z_2$ is a soft-edge, the genus one Whitham velocity $v_2$ in (\ref{motion}) is given by
\begin{equation}\label{v2}
	v_1(\bm{z})=-2(z_1^2+z_2^2+z_3^2)+\frac{4(z_2^2-z_1^2)(z_2^2-z_3^2)}{z_1^2-z_2^2+(z_3^2-z_1^2)\frac{E(m)}{K(m)}}, 
\end{equation}
which will be appeared in Section \ref{section5}. 
\par

\section{Region \RNum{1} for $6q_l^2t-12q_r^2t-M<x<6q_l^2t-12q_r^2t-\varepsilon$}\label{section4}

\hspace{1em} In Region \RNum{1} for $6q_l^2t-12q_r^2t-M<x<6q_l^2t-12q_r^2t-\varepsilon$, the 
leading-order term of large-time asymptotic solution should be similar to the plane wave specified by the left half of the initial data (\ref{initial}). From the perspective of RHP, it means that the $g$-function should be cut on $I_l=(-q_l,q_l)$, that is, let $G=0$ and $z_1=q_l$, then one has
\begin{equation}\label{gl}
	g_l(\xi;z)=(4z^2+12\xi+2q_l^2)\mathcal{R}_l(z)=\int_{q_l}^z\frac{12s(s-z_l)(s+z_l)}{\mathcal{R}_l(s)}ds, 
\end{equation}
where $z_l=\sqrt{\frac{q_l^2}{2}-\xi}\in(q_r,+\infty)$ with $z_l=q_r$ only for $\xi=\frac{q_l^2}{2}-q_r^2$. Furthermore, the signature table of $\Im g_l(\xi;z)$ is illustrated in Figure \ref{imgl}. 
\par

\subsection{The opening of lense}\

Introduce the following transformation 
\begin{equation}\label{M-widetilde}
	\widetilde{\mathcal{M}}_l(x,t;z)=\mathcal{M}(x,t;z)e^{it(g_l(\xi;z)-\theta(x;\xi))\sigma_3}\mathcal{G}_l(x,t;z)\mathcal{D}_l(\xi;z)^{-\sigma_3}, 
\end{equation}
where the matrix-valued function $\mathcal{G}_l(x,t;z)$ is explicitly defined by
\begin{equation}
	\mathcal{G}_{l}(x,t;z)=\begin{cases}
		\begin{aligned}
			&\begin{pmatrix}
				1&0\\r(z)e^{2it g_l(\xi;z)}&1
			\end{pmatrix}, &z\in \Omega_l^{(1)}\cup\Omega_l^{(2)}\\
			&\begin{pmatrix}
				1&-r^*(z)e^{-2it g_l(\xi;z)}\\0&1
			\end{pmatrix}, &z\in \overline{\Omega}_l^{(1)}\cup \overline{\Omega}_l^{(2)}, \\
			&\begin{pmatrix}
				1&0\\ \frac{r(z)}{1-r(z)r^{*}(z)}
				e^{2it g_l(\xi;z)}&1
			\end{pmatrix}, &z\in \Omega_l^{(3)}, \\ 
			&\begin{pmatrix}
				1&-\frac{r^*(z)}{1-r(z)r^{*}(z)}
				e^{-2it g_l(\xi;z)}\\0&1
			\end{pmatrix}, &z\in\overline{\Omega}_l^{(3)},\\
			&\quad I, &elsewhere, 
		\end{aligned}
	\end{cases}
\end{equation}
with the domains $\Omega_l^{(j)}$ and $\overline{\Omega}_l^{(j)}$ ($j=1,2,3$) depicted in Figure \ref{o1}. The scalar function $\mathcal{D}_l(\xi;z)$ is constructed by 
\begin{equation}\label{Dl}
	\mathcal{D}_l(\xi;z):=\exp\left\{\frac{\mathcal{R}_{l+}(z)}{2\pi i}\left[(\int_{q_r}^{z_l}+\int_{-z_l}^{-q_r})\frac{\ln(1-r(s)r^*(s))}{(s-z)\mathcal{R}_{l+}(s)}ds-(\int_{q_l}^{q_r}+\int_{-q_r}^{-q_l})\frac{\ln(a_+(s)a_-^*(s))}{(s-z)\mathcal{R}_{l_+}(s)}ds\right]\right\}, 
\end{equation}
which satisfies the following scalar RHP.  
\par
\begin{prop}
	The function $\mathcal{D}_l(z)=\mathcal{D}_l(\xi;z)$ defined in (\ref{Dl}) exhibits the following characteristics: \\
	(a) $\mathcal{D}_l(z)$ is holomorphic for $z\in\mathbb{C}\backslash[-z_l, z_l]$. \\
	(b) $\mathcal{D}_l(z)$ satisfies the jump conditions:
	\begin{equation}
		\begin{cases}
			\begin{aligned}
				&\mathcal{D}_{l+}(z)=\mathcal{D}_{l-}(z)(1-r(z)r^*(z)),  &z\in(-z_l,-q_r)\cup(q_r,z_l) , \\
				&\mathcal{D}_{l+}(z)=\mathcal{D}_{l-}(z)(a_+(z)a_-^*(z)),&z\in(-q_r,-q_l)\cup(q_l,q_r) ,\\
				&\mathcal{D}_{l+}(z)\mathcal{D}_{l-}(z)=1,&z\in(-q_l,q_l). \\
			\end{aligned}
		\end{cases}
	\end{equation}
	(c) $\mathcal{D}_l(z)$ admits the symmetry $\mathcal{D}_l(-z)=1/\mathcal{D}_l(z)$ for $z\in\mathbb{C}$. \\
	(d) $\mathcal{D}_l(z)=1+\mathcal{O}(z^{-1})$ as $z\to\infty$. 
\end{prop}
\begin{proof}
	(a) follows from the general properties of Cauchy type integrals. (b) is a consequence of the Plemelj's formula and the fact that $\mathcal{R}_{l+}(z)+\mathcal{R}_{l-}(z)=0$ for $z\in(-q_l,q_l)$. (c) follows from the change of variable $s\mapsto-s$ in integration, the symmetry of $r(z)$ and $a(z)$ shown in Lemma \ref{abr}, and the fact that $\mathcal{R}_{l+}(-z)=-\mathcal{R}_{l+}(z)$  for $z\in\mathbb{C}\backslash[-q_l,q_l]$. Regarding property (d), the limit of  $\mathcal{D}_l(\xi;z)$ at $z=\infty$ is given by
	\begin{equation}
		\exp\left\{\frac{-1}{2\pi i}\left[(\int_{q_r}^{z_l}+\int_{-z_l}^{-q_r})\frac{\ln(1-r(s)r^*(s))}{\mathcal{R}_{l_+}(s)}ds-(\int_{q_l}^{q_r}+\int_{-q_r}^{-q_l})\frac{\ln(a_+(s)a_-^*(s))}{\mathcal{R}_{l_+}(s)}ds\right]\right\}, 
	\end{equation}
	which reduces to 1 due to the symmetry of $r(z)$, $a(z)$ and $\mathcal{R}_{l+}(z)$.   
\end{proof}
\begin{figure}[ht]
	\begin{minipage}[b]{0.42\textwidth}  % 确保有单位（如 \textwidth）
		\centering
		\includegraphics[width=\linewidth]{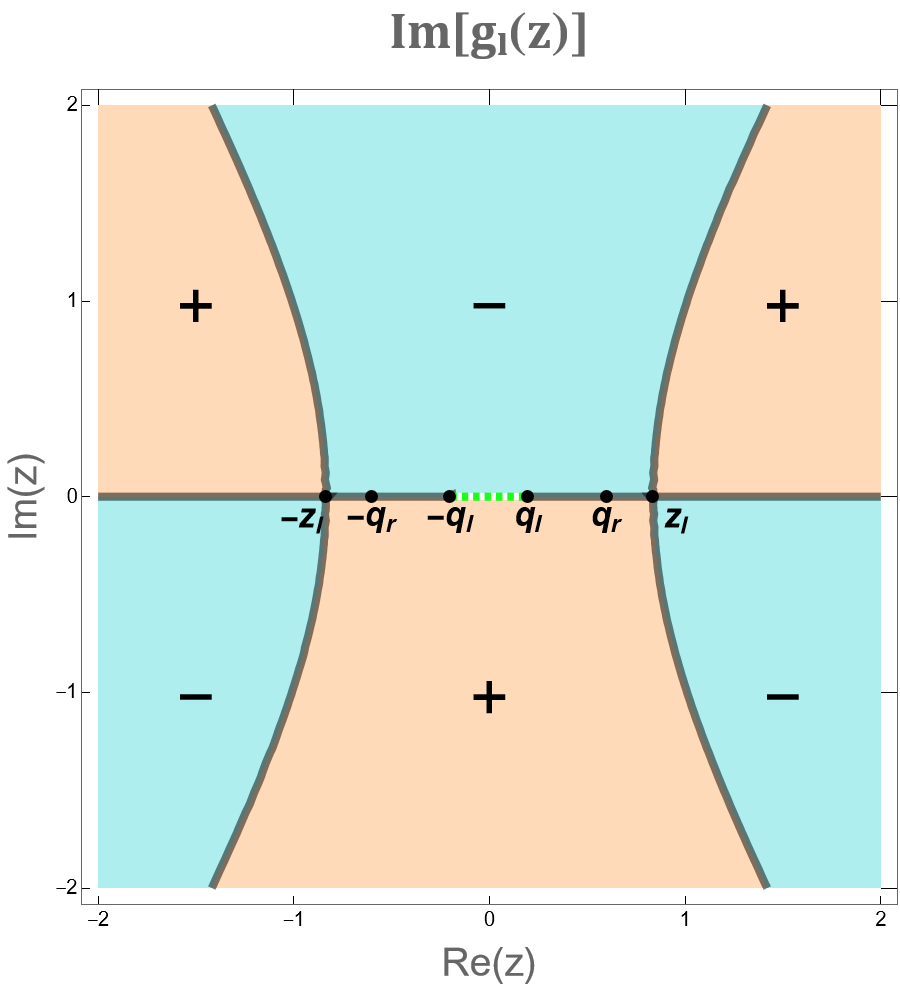}  
		\caption{The signs of $\Im g_l(\xi;z)$}
		\label{imgl}
	\end{minipage}
	\begin{minipage}[b]{0.56\textwidth}  % 确保宽度总和不超过 1.0\textwidth
		\centering
		\begin{tikzpicture}[scale=1]
			\draw[line width=1pt,dashed, draw=gray] (-3.5,0) -- (3.5,0) node[pos=1, below, font=\small]{$\mathbb{R}$};
			\draw[line width=1pt,-{Stealth[length=2mm, width=1.5mm]}] (-0.45,0) -- (0.1,0) node[pos=-2,above=1mm, font=\small] {$-z_l$}  node[pos=0, below=-0.5mm, font=\small] {$-q_l$};
			\draw[line width=1pt] (-0.1,0) -- (0.45,0) node[pos=1, below, font=\small] {$q_l$} node[pos=3,above=1mm, font=\small] {$z_l$};
			% 第一条线
			\draw[line width=1pt,-{Stealth[length=2mm, width=1.5mm]}] (0,-1.5) -- (0.8,-0.7) node[pos=1, below=2mm, font=\small, red!80!white]{$\overline{\Gamma}_l^{(2)}$};
			\draw[line width=1pt,-{Stealth[length=2mm, width=1.5mm]}] (0.7,-0.8) -- (2.5,1); % 前半段带箭头
			\draw[line width=1pt] (2.4,0.9) -- (3.5,2) node[pos=0.5, below=2mm, font=\small, red!80!white]{$\Gamma_l^{(1)}$}; % 后半段无线
			% 第二条线
			\draw[line width=1pt,-{Stealth[length=2mm, width=1.5mm]}] (0,1.5) -- (0.8,0.7) node[pos=1, above=2mm, font=\small, red!80!white]{$\Gamma_l^{(2)}$};
			\draw[line width=1pt,-{Stealth[length=2mm, width=1.5mm]}] (0.7,0.8) -- (2.5,-1); % 前半段带箭头
			\draw[line width=1pt] (2.4,-0.9) -- (3.5,-2) node[pos=0.5, above=2mm, font=\small, red!80!white]{$\overline{\Gamma}_l^{(1)}$}; % 后半段无线
			% 第三条线
			\draw[line width=1pt,-{Stealth[length=2mm, width=1.5mm]}] (-3.5,-2) -- (-2.4,-0.9) node[pos=0.7, above=0.5mm, font=\small, red!80!white]{$\overline{\Gamma}_l^{(4)}$};
			\draw[line width=1pt,-{Stealth[length=2mm, width=1.5mm]}] (-2.5,-1) -- (-0.7,0.8); % 前半段带箭头
			\draw[line width=1pt] (-0.8,0.7) -- (0,1.5) node[pos=0.1, above, font=\small, red!80!white]{$\Gamma_l^{(3)}$}; % 后半段无线
			% 第四条线
			\draw[line width=1pt,-{Stealth[length=2mm, width=1.5mm]}] (-3.5,2) -- (-2.4,0.9) node[pos=0.5, below=1mm, font=\small, red!80!white]{$\overline{\Gamma}_l^{(4)}$};
			\draw[line width=1pt,-{Stealth[length=2mm, width=1.5mm]}] (-2.5,1) -- (-0.7,-0.8); % 前半段带箭头
			\draw[line width=1pt] (-0.8,-0.7) -- (0,-1.5) node[pos=0.1, below=1.5mm, font=\small, red!80!white]{$\overline{\Gamma}_l^{(3)}$}; % 后半段无线
			% 标注点
			\draw[fill=black, draw=black, line width=0.05pt] (0.45,0) circle (0.03cm);
			\draw[fill=black, draw=black, line width=0.05pt] (0.95,0) circle (0.03cm) node[below, font=\small] {$q_r$};
			\draw[fill=black, draw=black, line width=0.05pt] (-0.45,0) circle (0.03cm);
			\draw[fill=black, draw=black, line width=0.05pt] (-0.95,0) circle (0.03cm) node[below=-0.5mm, font=\small] {$-q_r$};
			\path (-4.5,-3.5) rectangle (4,3.5);
			\node[font=\small,blue!40!black] at (0,0.6) {$\Omega_l^{(2)}$};
			\node[font=\small,orange!40!black] at (2.5,0.3) {$\Omega_l^{(1)}$};
			\node[font=\small,orange!40!black] at (-2.5,0.3) {$\Omega_l^{(3)}$};
			\node[font=\small,orange!40!black] at (0,-0.6) {$\overline{\Omega}_l^{(2)}$};
			\node[font=\small,blue!40!black] at (2.5,-0.3) {$\overline{\Omega}_l^{(1)}$};
			\node[font=\small,blue!40!black] at (-2.5,-0.3) {$\overline{\Omega}_l^{(3)}$};
			\draw (1.8,0) arc (0:45:0.3) node[pos=0.8,right] {$\frac{\pi}{4}$};
			\draw (1.3,0.2) arc (135:180:0.3) node[pos=0.1,left] {$\frac{\pi}{4}$};
		\end{tikzpicture}
		\caption{The opening of lenses in Region \RNum{1}}
		\label{o1}
	\end{minipage}
\end{figure}
\par
Consequently, the function $\widetilde{\mathcal{M}}_l(x,t;z)$ in (\ref{M-widetilde}) satisfies the following RHP. 
\begin{rhp}
	Find a $2\times2$ matrix-valued function $\widetilde{\mathcal{M}}_l(z)=\widetilde{\mathcal{M}}_l(x,t;z)$ satisfying: \\
	(1) $\widetilde{\mathcal{M}}_l(z)$ is holomorphic for $z\in\mathbb{C}\backslash\left(\bigcup_{j=1}^4(\Gamma_l^{(j)}\cup\overline{\Gamma}_l^{(j)})\cup [-q_l,q_l]\right)$. \\
	(2) For $z\in\bigcup_{j=1}^4(\Gamma_l^{(j)}\cup\overline{\Gamma}_l^{(j)})\cup (-q_l,q_l)$ in Figure \ref{o1}, the function $\widetilde{\mathcal{M}}_l(x,t;z)$ takes continuous boundary values $\widetilde{\mathcal{M}}_{l\pm}(x,t;z)$ which satisfy the jump relation 
	\begin{equation}
		\widetilde{\mathcal{M}}_{l+}(x,t;z)=\widetilde{\mathcal{M}}_{l-}(x,t;z)\widetilde{\mathcal{J}}_l(x,t;z),
	\end{equation}
	where
	\begin{equation}\label{j1}
		\widetilde{\mathcal{J}}_l(x,t;z)=\begin{cases}
			\begin{aligned}
				&\begin{pmatrix}
					1&0\\\mathcal{D}_l^{-2}re^{2itg_l}&1
				\end{pmatrix}, &z\in\Gamma_l^{(1)}\cup\Gamma_l^{(4)},\\
				&\begin{pmatrix}
					1&-\mathcal{D}_l^2 r^*e^{-2itg_l}\\0&1
				\end{pmatrix},&z\in\overline{\Gamma}_l^{(1)}\cup\overline{\Gamma}_l^{(4)}, \\
				&\begin{pmatrix}
					1&-\mathcal{D}_l^2\frac{r^*}{1-rr^*}e^{-2itg_l}\\0&1
				\end{pmatrix},&z\in\Gamma_l^{(2)}\cup\Gamma_l^{(3)},\\
				&\begin{pmatrix}
					1&0\\\mathcal{D}_l^{-2}\frac{r}{1-rr^*}e^{2itg_l}&1
				\end{pmatrix},&z\in\overline{\Gamma}_l^{(2)}\cup\overline{\Gamma}_l^{(3)},\\
				& -i\sigma_2, &z\in(-q_l,q_l). 
			\end{aligned}
		\end{cases}
	\end{equation}
	(3) $\widetilde{\mathcal{M}}_l(z)=I+\mathcal{O}(z^{-1})$, as $z\to\infty$ in $\mathbb{C}\backslash\bigcup_{j=1}^4(\Gamma_l^{(j)}\cup\overline{\Gamma}_l^{(j)})$ . \\
	(4) Symmetries: 
	\begin{equation}
		\widetilde{\mathcal{M}}_l(z)=\overline{\widetilde{\mathcal{M}}_l(-\bar{z})}=\sigma_1\widetilde{\mathcal{M}}_l(-z)\sigma_1=\sigma_1\overline{\widetilde{\mathcal{M}}_l(\bar{z})}\sigma_1. 
	\end{equation}
\end{rhp}
\par

\subsection{The outer parametrix for $\mathcal{M}_l^{(\infty)}(z)$}\

As $t$ large enough, the jump matrix $\widetilde{\mathcal{J}}_{l}(x,t;z)$ approaches to
\begin{equation}
	\mathcal{J}_{l}^{(\infty)}=\begin{cases}
		\begin{aligned}
			&-i\sigma_2, &z \in (-q_l,q_l), \\
			&\quad I, &elsewhere, 
		\end{aligned}
	\end{cases}
\end{equation}
so it is natural to establish the following parametrix for $\mathcal{M}_l^{(\infty)}(z)=\mathcal{M}_{l}^{(\infty)}(x,t;z)$. 
\par
\begin{rhp}\label{rhp_l}
	Find a $2\times2$ complex matrix-valued function $\mathcal{M}_{l}^{(\infty)}(z)=\mathcal{M}_{l}^{(\infty)}(x,t;z)$ that satisfies: \\
	(1) $\mathcal{M}_{l}^{(\infty)}(z)$ is holomorphic for $z\in\mathbb{C}\backslash[-q_l,q_l]$.  \\
	(2) For $z\in(-q_l,q_l)$, the jump condition below holds
    \begin{equation}
		\mathcal{M}_{l+}^{(\infty)}(z)=\mathcal{M}_{l-}^{(\infty)}(z)\cdot (-i\sigma_2). 
	\end{equation}
	(3) $\mathcal{M}_{l}^{(\infty)}(z)=I+\mathcal{O}(z^{-1})$, as $z\to\infty$ in $\mathbb{C}\backslash[-q_l,q_l]$. \\
\end{rhp}
It is easy to show that the unique solution of RHP \ref{rhp_l} is $\mathcal{M}_{l}^{(\infty)}(z)=\mathcal{E}_l(z)$, where $\mathcal{E}_l(z)$ is defined by (\ref{epsilon}). 
\par

\subsection{Small norm RHP for $\mathcal{M}_l^{(err)}(x,t;z)$ }\

Let $\varepsilon'>0$ be sufficiently small such that the pair of closed disks
\begin{equation}
	\mathcal{U}_l^{\pm} := \{ z \in \mathbb{C} \ \Big| \ |z \mp z_l| \le \varepsilon' \},
\end{equation}
centered at $\pm z_l$ with radius $\varepsilon$, are disjoint from $\mathcal{I}_l$, i.e., $\mathcal{U}_l^{\pm}\cap \mathcal{I}_l=\varnothing$. Denote their positively oriented boundaries by $\partial \mathcal{U}_l^{\pm}$, where the positive orientation is counterclockwise. Define a small-norm RHP as 
\begin{equation}\label{err_l}
	\mathcal{M}_l^{(err)}(x,t;z):=\widetilde{\mathcal{M}}_l(x,t;z) \left[\mathcal{P}^{(0)}(x,t;z)\right]^{-1},
\end{equation}
where $\mathcal{P}^{(0)}(x,t;z)$ coincides with $\widetilde{\mathcal{M}}_l(x,t;z)$ on $\mathcal{U}_l^{+}\cup \mathcal{U}_l^{-}$ and degenerates to $\mathcal{E}_l(z)$ elsewhere. Consequently, the jump matrix $\mathcal{J}_l^{(err)}(x,t;z)$ on the contour shown in Figure \ref{e1} takes the form 
\begin{equation}
	\mathcal{J}_l^{(err)}(x,t;z)=\begin{cases}
		\begin{aligned}
			& \mathcal{E}_l(z)\widetilde{\mathcal{J}}_l(x,t;z)\mathcal{E}_l(z)^{-1}, &z\in \left(\bigcup_{j=1}^4\Gamma_l^{(j)}\cup\overline{\Gamma}_l^{(j)}\right)\backslash \left(\mathcal{U}_l^{+}\cup \mathcal{U}_l^{-}\right), \\
			& \mathcal{P}^{(0)}(x,t;z)\mathcal{E}_l(z)^{-1}, &z\in\partial \mathcal{U}_l^{+}\cup\partial\mathcal{U}_l^{-}. 
		\end{aligned}
	\end{cases}
\end{equation}
\par
Let $\mathcal{M}_{l,1}^{(err)}(x,t)$ denote the coefficient of the $z^{-1}$-term in the Laurent expansion of $\mathcal{M}_l^{(err)}(x,t;z)$ at infinity. The reconstruction formula (\ref{construct}) yields the decomposition
\begin{equation}
	q(x,t)=2i \lim_{z\to\infty}(z\widetilde{\mathcal{M}}_l(x,t;z))_{12}=q_l+2i\left(\mathcal{M}_{l,1}^{(err)}(x,t)\right)_{12}. 
\end{equation}
The remaining task reduces to asymptotic analysis of the coefficient $\mathcal{M}_{l,1}^{(err)}(x,t)$, which can be achieved via the Deift-Zhou nonlinear steepest descent method \cite{DZ1993}.
\par
The procedure initiates with the construction of suitable local conformal mappings for $\mathcal{P}^{(1)}(x,t;z)=\mathcal{E}_l(z)^{-1}\mathcal{P}^{(0)}(x,t;z)$, which inherits identical jump conditions to $\widetilde{\mathcal{M}}_l(x,t;z)$ within $\partial \mathcal{U}_l^{+}\cup\partial\mathcal{U}_l^{-}$, while it reduces to the identity matrix elsewhere. 
\par
Building upon the definition of $ z_l $ and the analyticity of $ g_l(z) $, combined with the symmetry relation $ g_l(-z_l) = -g_l(z_l) $, the function $ g_l(z) $ admits a Taylor series expansion about $ \pm z_l $ (specifically within the neighborhoods $ \mathcal{\mathcal{U}}_l^{\pm} $):  
\begin{equation}\label{eq:taylor_expansion}  
	g_l(z) = \pm g_l(z_l) \pm \frac{g_l''(z_l)}{2} (z \mp z_l)^2 + \sum_{n=3}^{\infty} \frac{g_l^{(n)}(\pm z_l)}{n!} (z \mp z_l)^n,  
\end{equation}  
where $ g_l^{(n)}(z) $ denotes the $n$-th derivative of $ g_l(z) $. Crucially, the leading-order coefficients exhibit the following properties:
\begin{equation}  
	g_l(z_l) = -8 \left( -\xi - \frac{q_l^2}{2} \right)^{3/2} < 0, \quad g_l''(z_l) = \frac{24 \left( -\xi + \frac{q_l^2}{2} \right)}{\sqrt{ -\xi - \frac{q_l^2}{2} }} > 0.  
\end{equation}
\par
This analytical structure motivates the construction of local conformal mappings $z\mapsto\zeta_k$ for $k\in\{r,l\}$ (see Figure \ref{cl}): 
\begin{equation}
	g_l(z)=g_l(z_l)+\frac{1}{4t}\zeta_r^2,\quad z\in \mathcal{U}_l^{+}, \qquad \quad g_l(z)=-g_l(z_l)-\frac{1}{4t}\zeta_l^2,\quad z\in \mathcal{U}_l^{-}. 
\end{equation}
Leveraging the expansion (\ref{eq:taylor_expansion}), it is systematically derived that
\begin{equation}\label{sqrt_t}
	\begin{aligned}
		\zeta_r&=\sqrt{2tg''(z_l)}(z-z_l)\left(1+\sum_{n=1}^{\infty}\frac{g_l^{(n+2)}(z_l)}{n!} (z-z_l)^n\right)^{1/2}, & z\in \mathcal{U}_l^{\pm}, \\
		\zeta_l&=-\sqrt{2tg''(z_l)}(z+z_l)\left(1-\sum_{n=1}^{\infty}\frac{g_l^{(n+2)}(- z_l)}{n!} (z + z_l)^n\right)^{1/2},& z\in \mathcal{U}_l^{-}. \\
	\end{aligned}
\end{equation}
Here, the variable $\zeta_l$ undergoes a $\pi$-radian phase rotation relative to the $(z+z_l)$. Furthermore, the inverse series representations can also be obtained 
\begin{equation}
	\begin{aligned}
		&z-z_l=\frac{\zeta_r}{\sqrt{2tg''(z_l)}}+\sum_{n=2}^{\infty} f_{r,n}\left(\frac{\zeta_r}{\sqrt{2tg''(z_l)}}\right)^n, &z\in\mathcal{U}_l^+, \\
		&z+z_l=-\frac{\zeta_l}{\sqrt{2tg''(z_l)}}+\sum_{n=2}^{\infty} f_{l,n}\left(-\frac{\zeta_l}{\sqrt{2tg''(z_l)}}\right)^n, &z\in\mathcal{U}_l^-, 
	\end{aligned}
\end{equation}
where the coefficients $f_{l,n}$ and $f_{r,n}$ ($n\ge 2$) are recursively determined from the known derivatives $g_l^{(n)}(\pm z_l)$. 
\begin{figure}[ht]
	\begin{minipage}[b]{0.49\textwidth}  % 确保有单位（如 \textwidth）
		\centering
		\begin{tikzpicture}[scale=1]
			% 圈1
			\draw[line width=1pt] (2,1.5) circle (1cm) node[above=5mm, font=\small] {$\mathcal{U}^{+}$}; 
			\draw[line width=0.01pt,-{Stealth[length=2mm, width=1.5mm]}](1.92,2.5) -- (1.91,2.5);
			\node[font=\small] at (2,1.3) {$z_l$};
			% 线1
			\draw[line width=1pt,-{Stealth[length=2mm, width=1.5mm]}] (1.3,0.8) -- (1.8,1.3);
			\draw[line width=1pt,-{Stealth[length=2mm, width=1.5mm]}] (1.7,1.2) -- (2.5,2); % 前半段带箭头
			\draw[line width=1pt] (2,1.5) -- (2.7,2.2);
			\draw[line width=1pt,-{Stealth[length=2mm, width=1.5mm]}] (1.3,2.2) -- (1.8,1.7);
			\draw[line width=1pt,-{Stealth[length=2mm, width=1.5mm]}] (1.7,1.8) -- (2.5,1); % 前半段带箭头
			\draw[line width=1pt] (2,1.5) -- (2.7,0.8);
			% 圈11
			\draw[line width=1pt] (-2,1.5) circle (1cm) node[above=5mm, font=\small] {$\zeta_r(\mathcal{U}^{+})$}; 
			\draw[line width=0.01pt,-{Stealth[length=2mm, width=1.5mm]}](-2.08,2.5) -- (-2.09,2.5);
			\node[font=\small] at (-2,1.2) {$O$};
			% 线11
			\draw[line width=1pt,-{Stealth[length=2mm, width=1.5mm]}] (-2.7,2.2) -- (-2.2,1.7);
			\draw[line width=1pt,-{Stealth[length=2mm, width=1.5mm]}] (-2.3,1.8) -- (-1.5,1); % 前半段带箭头
			\draw[line width=1pt] (-1.6,1.1) -- (-1.3,0.8);
			\draw[line width=1pt,-{Stealth[length=2mm, width=1.5mm]}] (-2.7,0.8) -- (-2.2,1.3);
			\draw[line width=1pt,-{Stealth[length=2mm, width=1.5mm]}] (-2.3,1.2) -- (-1.5,2); % 前半段带箭头
			\draw[line width=1pt] (-1.6,1.9) -- (-1.3,2.2);
			% 圈2
			\draw[line width=1pt] (-2,-1.5) circle (1cm) node[above=5mm, font=\small] {$\mathcal{U}^{-}$}; 
			\draw[line width=0.01pt,-{Stealth[length=2mm, width=1.5mm]}](-2.08,-0.5) -- (-2.09,-0.5);
			\node[font=\small] at (-2,-1.8) {$-z_l$};
			% 线2
			\draw[line width=1pt,-{Stealth[length=2mm, width=1.5mm]}] (-2.7,-2.2) -- (-2.2,-1.7);
			\draw[line width=1pt,-{Stealth[length=2mm, width=1.5mm]}] (-2.3,-1.8) -- (-1.5,-1); % 前半段带箭头
			\draw[line width=1pt] (-1.6,-1.1) -- (-1.3,-0.8);
			\draw[line width=1pt,-{Stealth[length=2mm, width=1.5mm]}] (-2.7,-0.8) -- (-2.2,-1.3);
			\draw[line width=1pt,-{Stealth[length=2mm, width=1.5mm]}] (-2.3,-1.2) -- (-1.5,-2); % 前半段带箭头
			\draw[line width=1pt] (-1.6,-1.9) -- (-1.3,-2.2);
			% 圈22
			\draw[line width=1pt] (2,-1.5) circle (1cm) node[above=5mm, font=\small] {$\zeta_l(\mathcal{U}^{-})$}; 
			\draw[line width=0.01pt,-{Stealth[length=2mm, width=1.5mm]}](2.08,-2.5) -- (2.09,-2.5);
			\node[font=\small] at (2,-1.8) {$O$};
			% 线22
			\draw[line width=1pt,-{Stealth[length=2mm, width=1.5mm]}] (2.7,-2.2) -- (2.2,-1.7);
			\draw[line width=1pt,-{Stealth[length=2mm, width=1.5mm]}] (2.3,-1.8) -- (1.5,-1); % 前半段带箭头
			\draw[line width=1pt] (1.6,-1.1) -- (1.3,-0.8);
			\draw[line width=1pt,-{Stealth[length=2mm, width=1.5mm]}] (2.7,-0.8) -- (2.2,-1.3);
			\draw[line width=1pt,-{Stealth[length=2mm, width=1.5mm]}] (2.3,-1.2) -- (1.5,-2); % 前半段带箭头
			\draw[line width=1pt] (1.6,-1.9) -- (1.3,-2.2);
			% mapping
			\draw[line width=1pt,-{Stealth[length=3mm, width=2mm]}] (0.8,1.5) -- (-0.8,1.5) node[pos=0.5,above] {$\zeta_r$};
			\draw[line width=1pt,-{Stealth[length=3mm, width=2mm]}] (-0.8,-1.5) -- (0.8,-1.5) node[pos=0.5,above] {$\zeta_l$};
		\end{tikzpicture}
		\caption{Local conformal mappings}
		\label{cl}
	\end{minipage}
	\begin{minipage}[b]{0.49\textwidth}  % 确保宽度总和不超过 1.0\textwidth
		\centering
		\begin{tikzpicture}[scale=1]
			% 第一条线
			\draw[line width=1pt,-{Stealth[length=2mm, width=1.5mm]}] (0,-1.5) -- (0.8,-0.7) node[pos=1.2, below=5mm, font=\small, red!80!white]{$\overline{\Gamma}_l^{(2)}\backslash\mathcal{U}^{+}$};
			\draw[line width=1pt,-{Stealth[length=2mm, width=1.5mm]}] (0.7,-0.8) -- (2.5,1); % 前半段带箭头
			\draw[line width=1pt] (2.4,0.9) -- (3.5,2) node[pos=1, below=5mm, font=\small, red!80!white]{$\Gamma_l^{(1)}\backslash\mathcal{U}^{+}$}; % 后半段无线
			% 第二条线
			\draw[line width=1pt,-{Stealth[length=2mm, width=1.5mm]}] (0,1.5) -- (0.8,0.7) node[pos=1.2, above=5mm, font=\small, red!80!white]{$\Gamma_l^{(2)}\backslash\mathcal{U}^{+}$};
			\draw[line width=1pt,-{Stealth[length=2mm, width=1.5mm]}] (0.7,0.8) -- (2.5,-1); % 前半段带箭头
			\draw[line width=1pt] (2.4,-0.9) -- (3.5,-2) node[pos=1, above=5mm, font=\small, red!80!white]{$\overline{\Gamma}_l^{(1)}\backslash\mathcal{U}^{+}$}; % 后半段无线
			% 第三条线
			\draw[line width=1pt,-{Stealth[length=2mm, width=1.5mm]}] (-3.5,-2) -- (-2.4,-0.9) node[pos=0.5, above=3mm, font=\small, red!80!white]{$\overline{\Gamma}_l^{(4)}\backslash\mathcal{U}^{-}$};
			\draw[line width=1pt,-{Stealth[length=2mm, width=1.5mm]}] (-2.5,-1) -- (-0.7,0.8); % 前半段带箭头
			\draw[line width=1pt] (-0.8,0.7) -- (0,1.5) node[pos=0, above=3mm, font=\small, red!80!white]{$\Gamma_l^{(3)}\backslash\mathcal{U}^{-}$}; % 后半段无线
			% 第四条线
			\draw[line width=1pt,-{Stealth[length=2mm, width=1.5mm]}] (-3.5,2) -- (-2.4,0.9) node[pos=0.5, below=3.5mm, font=\small, red!80!white]{$\Gamma_l^{(4)}\backslash\mathcal{U}^{-}$};
			\draw[line width=1pt,-{Stealth[length=2mm, width=1.5mm]}] (-2.5,1) -- (-0.7,-0.8); % 前半段带箭头
			\draw[line width=1pt] (-0.8,-0.7) -- (0,-1.5) node[pos=0, below=3.5mm, font=\small, red!80!white]{$\overline{\Gamma}_l^{(3)}\backslash\mathcal{U}^{-}$}; % 后半段无线
			% 掏空
			\draw[fill=gray!40!white, draw=white, line width=0.05pt] (1.5,0) circle (0.4cm);
			\draw[fill=gray!40!white, draw=white, line width=0.05pt] (-1.5,0) circle (0.4cm);
			\draw[line width=1pt,dashed, draw=gray] (-3.5,0) -- (3.5,0) node[pos=1, below, font=\small]{$\mathbb{R}$};
			\draw[line width=1pt,-{Stealth[length=2mm, width=1.5mm]}] (-0.5,0) -- (0.1,0);
			\draw[line width=1pt] (-0.1,0) -- (0.5,0);
			% 标注点
			\draw[fill=black, draw=black, line width=0.05pt] (0.5,0) circle (0.03cm);
			\draw[fill=black, draw=black, line width=0.05pt] (0.9,0) circle (0.03cm);
			\draw[fill=black, draw=black, line width=0.05pt] (-0.5,0) circle (0.03cm);
			\draw[fill=black, draw=black, line width=0.05pt] (-0.9,0) circle (0.03cm);
			\node[font=\small] at (-0.9,-0.15) {$-q_r$};
			\node[font=\small] at (-0.5,-0.15) {$-q_l$}; 
			\node[font=\small] at (0.9,-0.2) {$q_r$};
			\node[font=\small] at (0.5,-0.2) {$q_l$};
			\draw[line width=1pt] (1.5,0) circle (0.4cm) node[above=5mm, font=\small] {$\partial \mathcal{U}^{+}$}; 
			\draw[line width=0.01pt,-{Stealth[length=2mm, width=1.5mm]}](1.42,0.4) -- (1.41,0.4);
			\draw[line width=1pt] (-1.5,0) circle (0.4cm) node[above=5mm, font=\small] {$\partial \mathcal{U}^{-}$};
			\draw[line width=0.01pt,-{Stealth[length=2mm, width=1.5mm]}](-1.61,0.4) -- (-1.62,0.4);
			\draw[fill=black, draw=black, line width=0.05pt] (-1.5,0) circle (0.03cm);
			\draw[fill=black, draw=black, line width=0.05pt] (1.5,0) circle (0.03cm);
			\node[font=\small] at (1.5,-0.2) {$z_l$}; 
			\node[font=\small] at (-1.5,-0.15) {$-z_l$}; 
			\path (-4.5,-3) rectangle (4,3);
		\end{tikzpicture}
		\caption{The jump contour of $\mathcal{M}_l^{(err)}(z)$}
		\label{e1}
	\end{minipage}
\end{figure}
\par
Additionally, we investigate the asymptotic behavior of $\mathcal{D}_l(z)$ in the vicinity of $\pm z_l$. The singular structure of $\mathcal{D}_l(z)$ is characterized by the essential factor $\left[(z-z_l)/(z+z_l)\right]^{i\nu}$, with the critical exponent $\nu=-(2\pi)^{-1}\ln \left(1-|r(z_l)|^2\right)$. This motivates the introduction of a meromorphic function $\chi(z)$ through the factorization:
\begin{equation}\label{chi}
	\mathcal{D}_l(z)=:\left(\frac{z-z_l}{z+z_l}\right)^{i\nu}{\chi(z)}
\end{equation}
where $\chi(z)$ constitutes a meromorphic function in $\mathbb{C}\backslash[-z_l,z_l]$ admitting well-defined, nonzero limits at $\pm z_l$. The complete derivation of these properties is detailed in Ref. \cite{Muskhelishvili1958}. 
\par
We now establish the asymptotic factorizations below
\begin{equation}
	\mathcal{D}_l(z(\zeta_r))e^{-itg(z(\zeta_r))}=\mathcal{D}_0\cdot D_1(\zeta_r),\qquad \mathcal{D}_l
	(z(\zeta_l))e^{-itg(z(\zeta_l))}=\mathcal{D}_0^{-1}\cdot\mathcal{D}_2(\zeta_l),  
\end{equation}
where the components are explicitly given by
\begin{equation}
	\begin{aligned}
		&\mathcal{D}_0=\left[2tg''(z_l)\right]^{-i\nu/2}\cdot \left(2z_l\right)^{-i\nu}\cdot\chi(z_l)\cdot\exp\left[-itg(z_l)\right],\qquad \qquad |\mathcal{D}_0|=1,  \\
		&\mathcal{D}_1(\zeta_r)=\zeta_r^{i\nu}\cdot\chi(z(\zeta_r))\cdot\chi(z_l)^{-1}\exp\left[-\frac{i\zeta_r}{4}+i\nu\ln\left(1+\sum_{n=2}^{\infty} f_{r,n}\left(\frac{\zeta_r}{\sqrt{2tg''(z_l)}}\right)^{n-1}\right)\right. \\&\qquad\qquad\qquad\qquad\qquad\qquad\qquad\qquad\left. -i\nu\ln\left(1+\frac{\zeta_r}{2z_l\sqrt{2tg''(z_l)}}+\sum_{n=2}^{\infty} \frac{f_{r,n}}{2z_l}\left(\frac{\zeta_r}{\sqrt{2tg''(z_l)}}\right)^n\right) \right],\\
		&\mathcal{D}_2(\zeta_l)=\zeta_l^{-i\nu}\cdot\chi(z(\zeta_l))\cdot\chi(z_l)\cdot\exp\left[\frac{i\zeta_l}{4}-i\nu\ln\left(1-\sum_{n=2}^{\infty} f_{l,n}\left(-\frac{\zeta_l}{\sqrt{2tg''(z_l)}}\right)^{n-1}\right)\right. \\&\qquad\qquad\qquad\qquad\qquad\qquad\qquad\qquad\left. +i\nu\ln\left(1+\frac{\zeta_l}{2z_l\sqrt{2tg''(z_l)}}-\sum_{n=2}^{\infty} \frac{f_{l,n}}{2z_l}\left(-\frac{\zeta_l}{\sqrt{2tg''(z_l)}}\right)^n\right) \right]. 
	\end{aligned}
\end{equation}
In the asymptotic regime $t\to\infty$, the following uniform convergence holds
\begin{equation}
	r(z(\zeta_r))\to r(z_r), \quad \mathcal{D}_1(\zeta_r)\to\zeta_r^{i\nu}e^{-i\zeta_r/4}, \quad
	r(z(\zeta_l))\to r(z_l), \quad \mathcal{D}_2(\zeta_r)\to\zeta_l^{-i\nu}e^{i\zeta_r/4}. 
\end{equation} 
\par
Let $\mathcal{P}_{1}^{(r)}$ and $\mathcal{P}_{1}^{(l)}$ denote the coefficients of the $\zeta_r^{-1}$ and $\zeta_l^{-1}$ terms in the Laurent series expansion of $\mathcal{P}^{(r)}(z)$ and $\mathcal{P}^{(l)}(z)$, respectively, defined by
\begin{equation}
	\begin{aligned}
		\mathcal{P}^{(r)}(x,t;z)&=\mathcal{D}_0^{\hat{\sigma}_3}\left[\mathcal{M}^{(PC)}(\zeta_r(z),r(z_l))\right], &z\in \mathcal{U}_l^{+},\\
		\mathcal{P}^{(l)}(x,t;z)&=\mathcal{D}_0^{-\hat{\sigma}_3}\left[\sigma_1\mathcal{M}^{(PC)}(\zeta_l(z),r(z_l))\sigma_1\right], &z\in \mathcal{U}_l^{-}, 
	\end{aligned}
\end{equation}
where $\mathcal{M}^{(PC)}$ and $\beta=\beta(r(z_l))$ are defined in Appendix \ref{Appendix-B}. Following the analysis in \cite{DZ1993}, the asymptotic expansion below is obtained
\begin{equation}
	\mathcal{P}^{(1)}(z)=\begin{cases}
		\begin{aligned}
			&I+\left[\frac{1}{\zeta_r}\mathcal{P}_1^{(r)}+\mathcal{O}(\frac{\ln t}{t})\right]+\mathcal{O}(\zeta_r^{-2}),&z\in\partial\mathcal{U}_l^{+}\\
			&I+\left[\frac{1}{\zeta_l}\mathcal{P}_1^{(l)}+\mathcal{O}(\frac{\ln t}{t})\right]+\mathcal{O}(\zeta_l^{-2}),&z\in\partial\mathcal{U}_l^{-}. 
		\end{aligned}
	\end{cases}
\end{equation}
Thus the asymptotic expansion for the error jump matrix is derived  
\begin{equation}\label{Jerr}
	\mathcal{J}_l^{(err)}(x,t;z)=\begin{cases}
		\begin{aligned}
			&I+\mathcal{E}(z)\left[\frac{1}{\zeta_r}\begin{pmatrix}0&-i\beta\cdot\mathcal{D}_0^{2}\\i\overline{\beta\cdot\mathcal{D}_0^2}&0\end{pmatrix}\right]\mathcal{E}(z)^{-1}+\mathcal{O}(\frac{\ln t}{t}), &z\in\partial\mathcal{U}_l^{+},\\
			&I+\mathcal{E}(z)\left[\frac{1}{\zeta_l}\begin{pmatrix}0&i\overline{\beta\cdot\mathcal{D}_0^2}\\-i\beta\cdot\mathcal{D}_0^{2}&0\end{pmatrix}\right]\mathcal{E}(z)^{-1}+\mathcal{O}(\frac{\ln t}{t}),&z\in\partial\mathcal{U}_l^{-}.
		\end{aligned}
	\end{cases}
\end{equation}
In conjunction with (\ref{sqrt_t}) above, we establish that $\mathcal{J}_l^{(err)}(x,t;z)=I+\mathcal{O}(t^{-1/2})$ for $z\in \partial \mathcal{U}_l^{+}\cup\partial\mathcal{U}_l^{-}$. Moreover, a straightforward verification shows that $\mathcal{J}_l^{(err)}$ exhibits exponential decay to the identity matrix as $t\to\infty$ for $z\in\left(\Gamma_r\cup\overline{\Gamma}_r\right)\backslash \left(\mathcal{U}_d^{(l)}\cup \mathcal{U}_d^{(r)}\right)$. Consequently, for $z\in\Gamma_l^{(err)}$, it follows that
\begin{equation}\label{F1}
	\mathcal{J}_l^{(err)}(x,t;z)=I+\mathcal{O}(t^{-1/2}). 
\end{equation}
\par
In accordance with the Beals-Coifman theory \cite{BC1984}, the unique solution to the RHP associated with $\mathcal{M}_l^{(err)}$ admits the following integral representation
\begin{equation}
	\mathcal{M}_l^{(err)}(x,t;z)=I+\frac{1}{2\pi i}\int_{\Gamma_l^{(err)}}\frac{\mu(s)(\mathcal{J}_l^{(err)}(s)-I)}{s-z}ds, 
\end{equation}
where $\Gamma_l^{(err)}$ denotes the jump contour for $\mathcal{M}_l^{(err)}(x,t;z)$, and $\mu(z)$ is the unique solution to the singular integral equation: 
\begin{equation}
	\begin{aligned}
		\mu&=I+\mathcal{C}(\mu), \qquad\mu\in I+L^2(\Gamma_l^{(err)}), \qquad \mathcal{C}:L^2(\Gamma_l^{(err)})\to L^2(\Gamma_l^{(err)}), \\
		\mathcal{C}&f:=\mathcal{C}_-(f(\mathcal{J}_l^{(err)}-I)),\qquad (\mathcal{C}_-f)(z):=\lim_{\varepsilon\to0^+}\frac{1}{2\pi i}\int_{\Gamma_l^{(err)}}\frac{f(s)}{s-(z-i\varepsilon)}ds. 
	\end{aligned}
\end{equation}
This yields the coefficient $\mathcal{M}_{l,1}^{(err)}(x,t)$ of the asymptotic expansion $\mathcal{M}_{l}^{(err)}(x,t; z)=I+\mathcal{M}_{l,1}^{(err)}(x,t)/z+O(1/z^2)$ as
\begin{equation}
	\mathcal{M}_{l,1}^{(err)}(x,t)=-\frac{1}{2\pi i}\int_{\Gamma_l^{(err)}}\mu(s)(J_l^{(err)}(x,t;s)-I)ds, 
\end{equation}
which admits the following decomposition
\begin{equation}\label{decom}
	\begin{aligned}
		&\mathcal{M}_{l,1}^{(err)}(x,t)=\mathcal{F}_1+\mathcal{F}_2+\mathcal{F}_3, \quad \mathcal{F}_1:=-\frac{1}{2\pi i}\int_{\Gamma_l^{(err)}}(\mu(s)-I)\left(\mathcal{J}_{l}^{(err)}(s)-I\right)ds, \\
		&\mathcal{F}_2:=-\frac{1}{2\pi i}\int_{\Gamma_l^{(err)}\backslash(\mathcal{U}_l^{(l)}\cup \mathcal{U}_l^{(r)})}\left(\mathcal{J}_l^{(err)}(s)-I\right)ds,\quad\mathcal{F}_3:=-\frac{1}{2\pi i}\oint_{\partial \mathcal{U}_l^{(l)}\cup\partial \mathcal{U}_l^{(r)}}\left(\mathcal{J}_l^{(err)}(s)-I\right)ds. 
	\end{aligned}
\end{equation}
\par
From the estimation in (\ref{F1}), the boundedness of operator norm is obtained as
\begin{equation}
	||\mathcal{C}||_{L^2\to L^2}\le||\mathcal{C}_-||_{L^2\to L^2}||\mathcal{J}_l^{(err)}-I||_{L^2}=\mathcal{O}(t^{-\frac{1}{2}}), 
\end{equation}
which implies the solution estimation  
\begin{equation}
	||\mu-I||_{L^2}\le ||\mathcal{C}||_{L^2\to L^2}||\mu||_{L^2}\le \frac{||\mathcal{C}||_{L^2\to L^2}}{1-||\mathcal{C}||_{L^2\to L^2}}=\mathcal{O}(t^{-\frac{1}{2}}), 
\end{equation}
and consequently $\mathcal{F}_1=\mathcal{O}(t^{-1})$. Furthermore, we establish the exponential decay estimate $\mathcal{F}_2=\mathcal{O}(e^{-ct})$ for some positive constant $c>0$. The asymptotic analysis therefore reduces to the estimation of $\tilde{\mathcal{F}}$, the (1,2)-entry of $\mathcal{F}_3$. 
\par
By neglecting the higher-order terms $\mathcal{O}(\zeta_r^{-2})$ or $\mathcal{O}(\zeta_l^{-2})$ which contribute an error of order $\mathcal{O}(t^{-1})$, we can evaluate $\tilde{\mathcal{F}}$ through residue calculus at the simple poles located at $\pm z_l$ as 
\begin{equation}\label{F3}
	\begin{aligned}
		\tilde{\mathcal{F}}=&-\operatorname{Res}_{z=z_l} \left(\frac{1}{\zeta_r(z)}\left[\mathcal{E}(z)\mathcal{P}_1^{(r)}\mathcal{E}(z)^{-1}\right]_{12}\right)-\operatorname{Res}_{z=-z_l} \left(\frac{1}{\zeta_l(z)}\left[\mathcal{E}(z)\mathcal{P}_1^{(l)}\mathcal{E}(z)^{-1}\right]_{12}\right)+\mathcal{O}(\frac{\ln t}{t})\\
		=&-\frac{1}{\sqrt{2tg''(z_l)}}\left(a_l^2(z_l)\cdot i\beta\mathcal{D}_0^2+b_l^2(z_l)\cdot i\overline{\beta\mathcal{D}_0^2}+a_l^2(z_l)\cdot i\overline{\beta\mathcal{D}_0^2}+b_l^2(z_l)\cdot i\beta\mathcal{D}_0^2\right)+\mathcal{O}(\frac{\ln t}{t})\\
		=&-\frac{2i}{\sqrt{2tg''(z_l)}}\Re\left(\beta\mathcal{D}_0^2\right)+\mathcal{O}(t^{-1}\ln t)=-i\sqrt{\frac{2\nu}{tg''(z_l)}}\cos\left(\beta\mathcal{D}_0^2\right)+\mathcal{O}(\frac{\ln t}{t}),
	\end{aligned}
\end{equation}
where $a_l(z)=\left(\beta_l(z)+\beta(z)^{-1}\right)/2$ and $b_l(z)=i\left(\beta_l(z)-\beta(z)^{-1}\right)/2$ correspond to the (1,1) and (1,2) entries of $\mathcal{E}_l(z)$, respectively. 
\par
The synthesis of formulae (\ref{decom}), (\ref{F3}), (\ref{err_l}) and (\ref{construct}) yields the asymptotic solution $q(x,t)$ in Region \RNum{1} for $t\to \infty$ as follows
\begin{equation}
	q(x,t)=q_l+2\sqrt{\frac{2\nu}{tg''(z_l)}}\cos\left(\beta\mathcal{D}_0^2\right)+\mathcal{O}\left(\frac{\ln t}{t}\right), 
\end{equation}
which completes the proof of Region \RNum{1} in Theorem \ref{main}.
\par

\section{Region \RNum{2} for $6q_l^2t-12q_r^2t+\varepsilon<x<-4q_l^2t-2q_r^2t-\varepsilon$}\label{section5}

In Region \RNum{2} for $6q_l^2t-12q_r^2t+\varepsilon<x<-4q_l^2t-2q_r^2t-\varepsilon$, the stationary phase point $z_l$ lies to the left of $q_r$, which implies that the $(1,1)$-entry of $\widetilde{\mathcal{J}}_l(x,t;z)$ defined by (\ref{j1}) is exponentially large in $(-q_r,-z_l)\cup(z_l,q_r)$. To establish a stable limit problem, it is necessary to modify the $g$-function to incorporate two origin-symmetric gaps contained within $[-q_r,-q_l]$ and $[q_l,q_r]$, respectively. According to Section \ref{g=1}, it suffices to set the Riemann invariants $(z_1,z_2, z_3)=(q_r, z_d, q_l)$, whereupon the $g$-function takes the explicit form
\begin{equation}\label{gd}
	g_d(\xi;z) = \int_{q_l}^z \frac{12\zeta(\zeta^2 - z_d^2)(\zeta^2 - \mu^2)}{\mathcal{R}_d(z)} d\zeta, \quad \mathcal{R}_d(z)=\mathcal{R}(z;q_l,z_d,q_r),
\end{equation}
where $\mu=\mu(z_d)$ satisfies $\mu^2=-\xi+\frac{q_l^2+q_r^2-z_d^2}{2}$ with $z_d=z_d(\xi)$ being determined by the equation
\begin{equation}\label{zd}
	\int_{z_d}^{q_r}\frac{12\zeta(\zeta^2-z_d^2)(\zeta^2-\mu^2)}{\mathcal{R}_d(z)}d\zeta=0. 
\end{equation}
The signature table of $\Im g_d(\xi;z)$ is shown in Figure \ref{imgd}. Additionally, it holds that $\Im g_{d+}<0$ on $(-q_l,q_l)$, while $\Im g_{d+}>0$ on $(-q_r,-z_d)\cup(z_d,q_r)$. 
\par
We still need to check the existence of $z_d$, which satisfies the equation (\ref{zd}), or the equation
\begin{equation}\label{zdd}
	12\xi=v_1(q_l,z_d,q_r),
\end{equation}
where $v_1(z_1,z_2,z_3)$, $z_1<z_2<z_3$ has been shown in (\ref{v2}). Note that $v_1(z_1,z_2,z_3)$ is the velocity of the Whitham modulation equations for the defocusing mKdV equation derived in \cite{Whitham1976}. The relation with the velocity $w(r_1,r_2,r_3)$ of the Whitham modulation equations for KdV equation is given by
\begin{equation}\label{relationK}
	v_1(z_1,z_2,z_3)=w(-z_3^2,-z_2^2,-z_1^2). 
\end{equation}
It was shown in \cite{Levermore1988} that the Whitham modulation equations for KdV equation are strictly hyperbolic and satisfy the relation $\partial_{r_2}w_2(r_1,r_2,r_3)>0$ for $r_1<r_2<r_3$, which implies that
\begin{equation}
	\frac{\partial}{\partial z_2}v_1(z_1,z_2,z_3)=-2z_2\frac{\partial}{\partial r_2}w(-z_3^2,r_2,-z_1^2)<0,\qquad z_2>0. 
\end{equation}
According to the Implicit Function Theorem, the above relation shows that (\ref{zdd}) is invertible for $z_d$ as a function of $\xi$ only when $z_d>0$ or equivalently when $q_l>0$. 
\par
Additionally, the complete elliptic integrals admit expansions as $m\to 0$:
\begin{equation}
	K(m)=\frac{\pi}{2}\left(1+\frac{m^2}{4}+\frac{9}{64}m^4+\mathcal{O}(m^6)\right), \quad E(m)=\frac{\pi}{2}\left(1-\frac{m^2}{4}-\frac{3}{64}m^4+\mathcal{O}(m^6)\right),
\end{equation}
and as $m\to 1$: 
\begin{equation}
	K(m)=\frac{1}{2}\ln\frac{16}{1-m^2}(1+o(1)),\quad E(m)=1+\frac{1}{2}(1-m)\left[\ln\frac{16}{1-m^2}-1\right](1+o(1)).  
\end{equation}
Thus, as $m\to 0$ (or $z_d\to q_l$) and $m\to 1$ (or $z_d\to q_r$), it is derived that the right and left boundaries of $x/t$ in Region \RNum{2} are 
\begin{equation}
	v_1(q_l,q_l,q_r)=-4q_l^2-2q_r^2, \qquad v_1(q_l,q_r,q_r)=6q_l^2-12q_r^2, 
\end{equation}
respectively.
\par

\subsection{The opening of lenses}\

Introduce the transformation 
\begin{equation}
	\widetilde{\mathcal{M}}_d(x,t;z)=\mathcal{M}(x,t;z)e^{it(g_d(\xi;z)-\theta(\xi;z))\sigma_3}\mathcal{G}_d(x,t;z)\mathcal{D}_d(\xi;z)^{-\sigma_3}, 
\end{equation}
where the matrix function $\mathcal{G}_d(x,t;z)$ is precisely defined by
\begin{equation}
	\mathcal{G}_d(x,t;z)=\begin{cases}
		\begin{aligned}
			&\begin{pmatrix}
				1&0\\-re^{2itg_d}&1
			\end{pmatrix}, &z\in\Omega_d^{(1)}\cup\Omega_d^{(4)},\\
			&\begin{pmatrix}
				1&-r^*e^{-2itg_d}\\0&1
			\end{pmatrix},&z\in\overline{\Omega}_d^{(1)}\cup\overline{\Omega_d}^{(4)}, \\
			&\begin{pmatrix}
				1&\frac{r^*}{1-rr^*}e^{-2itg_d}\\0&1
			\end{pmatrix},&z\in\Omega_d^{(3)},\\
			&\begin{pmatrix}
				1&0\\\frac{r}{1-rr^*}e^{2itg_d}&1
			\end{pmatrix},&z\in\overline{\Omega}_d^{(3)},\\
			&\quad I, &else,
		\end{aligned}
	\end{cases}
\end{equation}
with the domains $\Omega_d^{(j)}$ and $\overline{\Omega}_d^{(j)}$ ($j=1,2,3$) illustrated in Figure \ref{o2}. The scalar function $\mathcal{D}_d(\xi;z)$ is given by: 
\begin{equation}\label{DD}
	\mathcal{D}_d(\xi;z)=\exp\left\{\frac{\mathcal{R}_d(z)}{2\pi i}\left[\left(\int_{-z_d}^{-q_l}+\int_{q_l}^{z_d} \right)\frac{\ln(a_+(\zeta)a^*_-(\zeta))}{(\zeta-z)\mathcal{R}_{d+}(\zeta)}d\zeta+\left(\int_{-q_r}^{-z_d}+\int_{z_d}^{q_r} \right)\frac{i \delta_1}{(\zeta-z)\mathcal{R}_d(\zeta)}d\zeta\right]\right\}, 
\end{equation}
which satisfies the following RHP.  
\par
\begin{prop}
	The function $\mathcal{D}_d(z)=\mathcal{D}_d(\xi;z)$ defined in (\ref{DD}) exhibits the following characteristics:\\
	(a) $\mathcal{D}_d(z)$ is holomorphic for $z\in\mathbb{C}\backslash[-q_r,q_r]$. \\
	(b) For $z\in (-q_r, q_r)$, $\mathcal{D}_d(z)$ satisfy the following jump conditions: \begin{equation}
		\begin{cases}
			\begin{aligned}
				&\mathcal{D}_{d+}(z)=\mathcal{D}_{d-}(z)\left(a_+(z)a_-^*(z)\right)^{-1}, &z\in(-z_d,-q_l)\cup(q_l,z_d), \\
				&\mathcal{D}_{d+}(z)\mathcal{D}_{d-}(z)=e^{-i\delta_1},&z\in(-q_r,-z_d), \\
				&\mathcal{D}_{d+}(z)\mathcal{D}_{d-}(z)=e^{i\delta_1}, &z\in(z_d,q_r), \\
				&\mathcal{D}_{d+}(z)\mathcal{D}_{d-}(z)=1, &z\in(-q_l,q_l). 
			\end{aligned}
		\end{cases}
	\end{equation}
    (c) $\mathcal{D}_d(z)$ admits the symmetry $\mathcal{D}_d(-z)=1/\mathcal{D}_d(z)$ for $z\in\mathbb{C}$. \\
    (d) $\mathcal{D}_d(z)=1+\mathcal{O}(z^{-1})$ as $z\to\infty$. 
\end{prop}
\par
\begin{proof}
	(a) and (b) are the direct conclusions of Plemelj's formula. (c) follows from the change of variable $\zeta\mapsto-\zeta$ in integration, the symmetry of $a(z)$ shown in Lemma \ref{abr}, and the fact that $\mathcal{R}_{d+}(-z)=-\mathcal{R}_{d+}(z)$  for $z\in\mathbb{C}\backslash[-q_r,-z_d]\cup[-q_l,q_l]\cup[z_d,q_r]$. In fact, if we set  
	\begin{equation}\label{delta-1-1}
		\delta_1=-i\frac{\int_{q_l}^{z_d}\frac{\ln(|a(\zeta)|^2\zeta d\zeta}{\mathcal{R}_{d+}(\zeta)}}{\int_{z_d}^{q_r}\frac{\zeta d\zeta}{\mathcal{R}_{d+}(\zeta)}}=-\frac{\sqrt{q_r^2-q_l^2}}{K(m)}\int_{q_l}^{z_d}\frac{\ln(|a(\zeta)|^2)\zeta d\zeta}{\mathcal{R}_{d+}(\zeta)}, 
	\end{equation}
	then (d) is automatically satisfied via the symmetry of the problem. 
\end{proof}
\begin{figure}[ht]
	\begin{minipage}[b]{0.42\textwidth}  % 确保有单位（如 \textwidth）
		\centering
		\includegraphics[width=\linewidth]{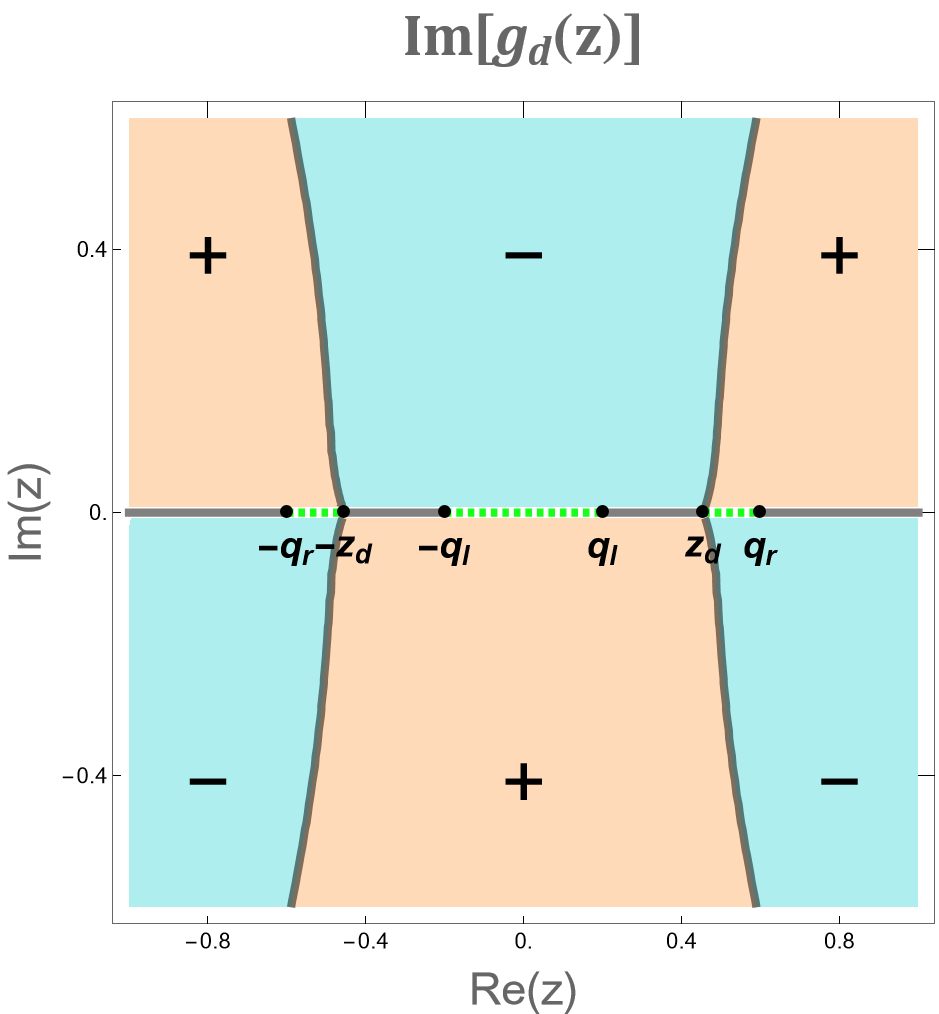}  % \linewidth 正确
		\caption{The signs of $\Im g_d(\xi;z)$}
		\label{imgd}
	\end{minipage}  
    \hfill
	\begin{minipage}[b]{0.56\textwidth}  % 确保宽度总和不超过 1.0\textwidth
		\centering
		\begin{tikzpicture}[scale=1]
			\draw[line width=1pt,dashed, draw=gray] (-4.5,0) -- (4.5,0) node[pos=1, below, font=\small]{$\mathbb{R}$};
			\draw[line width=1pt,-{Stealth[length=2mm, width=1.5mm]}] (-0.8,0) -- (0.15,0) node[pos=-0.95,above=1mm, font=\small] {$-z_d$} node[pos=0, below=-0.5mm,, font=\small] {$-q_l$};
			\draw[line width=1pt] (-0.15,0) -- (0.8,0) node[pos=1, below, font=\small] {$q_l$} node[pos=1.85,above=1mm, font=\small] {$z_d$};
			% 第一条线
			\draw[line width=1pt,-{Stealth[length=2mm, width=1.5mm]}] (0,-2.595) -- (0.9,-1.038) node[pos=0.9, below=3mm, font=\small, red!80!white]{$\overline{\Gamma}_d^{(2)}$};
			\draw[line width=1pt] (0.6,-1.557) -- (1.5,0) ;
			\draw[line width=1pt,-{Stealth[length=2mm, width=1.5mm]}] (1.5,0) -- (2.5,0.5); 
			\draw[line width=1pt] (2.4,0.45) -- (4.5,1.5) node[pos=0.5, above, font=\small, red!80!white]{$\Gamma_d^{(1)}$}; 
			% 第二条线
			\draw[line width=1pt,-{Stealth[length=2mm, width=1.5mm]}] (0,2.595) -- (0.9,1.038) node[pos=1, above=3mm, font=\small, red!80!white]{$\Gamma_d^{(2)}$};
			\draw[line width=1pt] (0.6,1.557) -- (1.5,0); 
			\draw[line width=1pt,-{Stealth[length=2mm, width=1.5mm]}] (1.5,0) -- (2.5,-0.5); 
			\draw[line width=1pt] (2.4,-0.45) -- (4.5,-1.5) node[pos=0.5, below, font=\small, red!80!white]{$\overline{\Gamma}_d^{(1)}$}; 
			% 第三条线
			\draw[line width=1pt,-{Stealth[length=2mm, width=1.5mm]}] (-4.5,-1.5) -- (-2.4,-0.45) node[pos=0.7, below, font=\small, red!80!white]{$\overline{\Gamma}_d^{(4)}$};
			\draw[line width=1pt] (-2.5,-0.5) -- (-1.5,0); 
			\draw[line width=1pt,-{Stealth[length=2mm, width=1.5mm]}] (-1.5,0) -- (-0.6,1.557) ; 
			\draw[line width=1pt] (-0.9,1.038) -- (0,2.595) node[pos=0.1, above=1mm, font=\small, red!80!white]{$\Gamma_d^{(3)}$}; 
			% 第四条线
			\draw[line width=1pt,-{Stealth[length=2mm, width=1.5mm]}] (-4.5,1.5) -- (-2.4,0.45) node[pos=0.5, above, font=\small, red!80!white]{$\Gamma_d^{(4)}$};
			\draw[line width=1pt] (-2.5,0.5) -- (-1.5,0); 
			\draw[line width=1pt,-{Stealth[length=2mm, width=1.5mm]}] (-1.5,0) -- (-0.6,-1.557); 
			\draw[line width=1pt] (-0.9,-1.038) -- (0,-2.595) node[pos=0.15, below=2.5mm, font=\small, red!80!white]{$\overline{\Gamma}_d^{(3)}$}; 
			% RRR
			\draw[line width=1pt,-{Stealth[length=2mm, width=1.5mm]}] (1.5,0) -- (2.4,0);
			\draw[line width=1pt] (2.1,0) -- (3,0) node[pos=1,below, font=\small] {$q_r$};
			\draw[line width=1pt,-{Stealth[length=2mm, width=1.5mm]}] (-3,0) -- (-2.1,0) node[pos=0,below=-1mm, font=\small] {$-q_r$};
			\draw[line width=1pt] (-2.4,0) -- (-1.5,0);
			% 标注点
			\draw[fill=black, draw=black, line width=0.05pt] (0.8,0) circle (0.03cm);
			\draw[fill=black, draw=black, line width=0.05pt] (1.5,0) circle (0.03cm);
			\draw[fill=black, draw=black, line width=0.05pt] (-0.8,0) circle (0.03cm);
			\draw[fill=black, draw=black, line width=0.05pt] (-1.5,0) circle (0.03cm);
			\path (-4.5,-3.5) rectangle (4,3.5);
			% 区域
			\node[font=\small,blue!40!black] at (0,0.9) {$\Omega_d^{(2)}$};
			\node[font=\small,orange!40!black] at (3.5,0.5) {$\Omega_d^{(1)}$};
			\node[font=\small,orange!40!black] at (-3.5,0.5) {$\Omega_d^{(3)}$};
			\node[font=\small,orange!40!black] at (0,-0.9) {$\overline{\Omega}_d^{(2)}$};
			\node[font=\small,blue!40!black] at (3.5,-0.5) {$\overline{\Omega}_d^{(1)}$};
			\node[font=\small,blue!40!black] at (-3.5,-0.5) {$\overline{\Omega}_d^{(3)}$};
			\path (-4.5,-3) rectangle (4,3);
		\end{tikzpicture}
		\caption{The opening of lenses in Region \RNum{2}}
		\label{o2}
	\end{minipage}
\end{figure}
\par
Thus, the function $\widetilde{\mathcal{M}}_d(x,t;z)$ satisfies the following RHP: 
\begin{rhp}
	Find a $2\times2$ matrix-valued function $\widetilde{\mathcal{M}}_d(z)=\widetilde{\mathcal{M}}_d(x,t;z)$ satisfies the properties: \\
	(1) $\widetilde{\mathcal{M}}_d(z)$ is holomorphic for $z\in\mathbb{C}\backslash\left(\bigcup_{j=1}^4(\Gamma_d^{(j)}\cup\overline{\Gamma}_d^{(j)}) \cup [-q_r,-z_d]\cup[-q_l,q_l]\cup[z_d,q_r]\right)$. \\
	(2) For $z\in\bigcup_{j=1}^4(\Gamma_d^{(j)}\cup\overline{\Gamma}_d^{(j)}) \cup (-q_r,-z_d)\cup(-q_l,q_l)\cup(z_d,q_r)$, we have 
	\begin{equation}
		\widetilde{\mathcal{M}}_{d+}(x,t;z)=\widetilde{\mathcal{M}}_{d-}(x,t;z)\widetilde{\mathcal{J}}_d(x,t;z),
	\end{equation}
	where
	\begin{equation}
		\widetilde{\mathcal{J}}_d(x,t;z)=\begin{cases}
			\begin{aligned}
				&\begin{pmatrix}
					1&0\\\mathcal{D}_d^{-2}re^{2itg_d}&1
				\end{pmatrix}, &z\in\Gamma_d^{(1)}\cup\Gamma_d^{(4)},\\
				&\begin{pmatrix}
					1&-\mathcal{D}_d^2 r^*e^{-2itg_d}\\0&1
				\end{pmatrix},&z\in\overline{\Gamma}_d^{(1)}\cup\overline{\Gamma}_d^{(4)}, \\
				&\begin{pmatrix}
					1&-\mathcal{D}_d^2\frac{r^*}{1-rr^*}e^{-2itg_d}\\0&1
				\end{pmatrix},&z\in\Gamma_d^{(2)}\cup\Gamma_d^{(3)},\\
				&\begin{pmatrix}
					1&0\\\mathcal{D}_d^{-2}\frac{r}{1-rr^*}e^{2itg_d}&1
				\end{pmatrix},&z\in\overline{\Gamma}_d^{(2)}\cup\overline{\Gamma}_d^{(3)},\\
				&\begin{pmatrix}
					\frac{\mathcal{D}_{d-}e^{2itg_{d+}}}{\mathcal{D}_{d+}a_+a^*_-}&-e^{-i(\delta_0+\delta_1)}\\e^{i(\delta_0+\delta_1)}&0
				\end{pmatrix}, &z\in(-q_r,-z_d),\\
				&\begin{pmatrix}
					\frac{\mathcal{D}_{d-}e^{2itg_{d+}}}{\mathcal{D}_{d+}a_+a^*_-}&-e^{i(\delta_0+\delta_1)}\\e^{-i(\delta_0+\delta_1)}&0
				\end{pmatrix}, &z\in(z_d,q_r),\\
				&-i\sigma_2, &z\in(-q_l,q_l), 
			\end{aligned}
		\end{cases}
	\end{equation}
    where $\delta_1$ is given by (\ref{delta-1-1}) and $\delta_0=\delta_0(\xi)$ is formulated as
\begin{equation}
	\delta_0(\xi)=2\int_{-q_l}^{-z_d}g'_{d+}(\zeta)d\zeta=\pi\frac{\sqrt{q_r^2-q_l^2}}{K(m)}\left(x+2(q_r^2+z_d^2+q_l^2)t\right), \quad m^2=\frac{q_r^2-z_d^2}{q_r^2-q_l^2}. 
\end{equation}
	(3) $\mathcal{M}_{d}^{(2)}(z)=I+\mathcal{O}(z^{-1})$, as $z\to\infty$ in $\mathbb{C}\backslash\left(\bigcup_{j=1}^4\Gamma_d^{(j)}\cup\overline{\Gamma}_d^{(j)} \right)$. \\
	(4) Symmetries: 
	\begin{equation}
\widetilde{\mathcal{M}}_d(z)=\overline{\widetilde{\mathcal{M}}_d(-\bar{z})}=\sigma_1\widetilde{\mathcal{M}}_d(-z)\sigma_1=\sigma_1\overline{\widetilde{\mathcal{M}}_d(\bar{z})}\sigma_1. 
	\end{equation}
\end{rhp}
\par

\par
\subsection{The outer parametrix for $\mathcal{M}^{(\infty)}_{dsw}$}\

According to the signature of $\Im g_d$ in Figure \ref{imgd}, as the time variable $t$  is sufficiently large, the jump matrix $\widetilde{\mathcal{J}}_{d}(x,t;z)$  converges pointwise to identity away from $\mathbb{R}$, and to constant matrices on $\mathbb{R}$. In fact, the convergence is uniform and exponential away from $\pm z_d$, while both of the non-uniform convergences near $z_d$ and $-z_d$ correspond to the Airy parametrix in Appendix \ref{Appendix-B}. Thus, it is natural for us to consider the limiting model problem for $\mathcal{M}_{dsw}^{(\infty)}(x,t;z)$ as follow. 
\par
\begin{rhp}\label{rhp_dsw}
	Find a $2\times2$ matrix-valued function $\mathcal{M}^{(\infty)}_{dsw}(z)=\mathcal{M}^{(\infty)}_{dsw}(x,t;z)$ satisfying: \\
	(1) $\mathcal{M}^{(\infty)}_{dsw}(z)$ is holomorphic for $z\in\mathbb{C}\backslash([-q_r,-z_d]\cup[-q_l,q_l]\cup [z_d,q_r])$.  \\
	(2) For $z\in(-q_r,-z_d)\cup(-q_l,q_l)\cup (z_d,q_r)$, the jump condition is
    \begin{equation}
		\mathcal{M}^{(\infty)}_{dsw+}(z)=\mathcal{M}^{(\infty)}_{dsw-}(z)\mathcal{\mathcal{J}}^{(\infty)}_{dsw}(x,t;z), 
	\end{equation}where 
	\begin{equation}
		\mathcal{J}^{(\infty)}_{dsw}(x,t;z)=\begin{cases}
			\begin{aligned}
				&\begin{pmatrix}
					0&-e^{i\delta}\\e^{-i\delta}&0
				\end{pmatrix}, &z\in(z_d,q_r),\\
				&-i\sigma_2, &z\in (-q_l,q_l), \\
				&\begin{pmatrix}
					0&-e^{-i\delta}\\e^{i\delta}&0
				\end{pmatrix}, &z\in(-q_r,-z_d),
			\end{aligned}
		\end{cases}
	\end{equation}
	with $\delta:=\delta_0+\delta_1$. \\
	(3) $\mathcal{M}^{(\infty)}_{dsw}(z)=I+\mathcal{O}(z^{-1})$ as $z\to\infty$. 
\end{rhp}
\par
In fact, the model problem above can be solved via a method of conformal transformation and provides the periodic traveling wave solution (\ref{qdsw}) of the defocusing mKdV equation. 
\par

\subsubsection{Auxiliary RHP for a new function $\mathcal{N}(x,t;w)$}\

To solve the model RHP \ref{rhp_dsw}, consider two real parameters $\tilde{q}>\tilde{d}>0$ satisfying $\tilde{q}^2=q_r^2-q_l^2$ and $\tilde{d}^2=z_d^2-q_l^2$, then it is ready to construct an auxiliary RHP for a new function $\mathcal{N}(x,t;w)$ which can be solved via the Riemann-theta functions. 
\par
\begin{rhp}\label{rhp_w}
	Find a $2\times2$ matrix-valued function $\mathcal{N}(w)=\mathcal{N}(x,t;w)$ satisfying: \\
	(1) $\mathcal{N}(w)$ is holomorphic for $w\in\mathbb{C}\backslash([-\tilde{q},-\tilde{d}]\cup[\tilde{d},\tilde{q}])$.  \\
	(2) For $w\in(-\tilde{q},-\tilde{d})\cup(\tilde{d},\tilde{q})$, the jump condition is
	\begin{equation}
		\mathcal{N}_{+}(w)=\mathcal{N}_{-}(w)J(w), 
	\end{equation}
	where 
	\begin{equation}
		J(w)=\begin{cases}
			\begin{aligned}
				&\begin{pmatrix}
					0&-e^{-i\tilde{\delta}}\\e^{i\tilde{\delta}}&0
				\end{pmatrix}, &w\in(-\tilde{q},-\tilde{d}),\\
				&\begin{pmatrix}
					0&-e^{i\tilde{\delta}}\\e^{-i\tilde{\delta}}&0
				\end{pmatrix}, &w\in(\tilde{d},\tilde{q}),    
			\end{aligned}
		\end{cases}
	\end{equation}
	where $\tilde{\delta}:= \delta_0 + \delta_1 + \delta_2$ and $\delta_2 = \delta_2(x,t)$ is determined by equation (\ref{delta2}) below. \\
	(3) $\mathcal{N}(w)=I+\mathcal{O}(w^{-1})$ as $w\to\infty$.     
\end{rhp}
\par
The explicit formula for function $\mathcal{N}(w)$ can be constructed by Riemann-theta function defined on the Jacobi variety ($\cong \mathbb{C}/\mathbb{Z}^2$) of the algebraic curve corresponding to a compact Riemann surface of the genus one, i.e., $\mathcal{S}_1=\{(w,y)\in\mathbb{C}^2|y^2=(w^2-\tilde{d}^2)(w^2-\tilde{q}^2)\}$, whose canonical homology basis $\{a,b\}$ is shown in Figure \ref{abcycle}. 
\par
\begin{figure}[ht]
	\centering
	\begin{tikzpicture}
		\draw[line width=5pt,gray] (1,0) -- (3,0) node[pos=0,below=-0.5mm,black]{$\tilde{d}$} node[pos=1,below=-0.5mm,black]{$\tilde{q}$};
		\draw[line width=5pt,gray] (-3,0) -- (-1,0) node[pos=0,below=-0.5mm,black]{$-\tilde{q}$} node[pos=1,below=-1mm,black]{$-\tilde{d}$};
		\draw[line width=2pt,blue!40!white] (1.5,0) to [out=90, in=90, looseness=0.9] (-1.5,0);
		\draw[dashed,line width=2pt,blue!40!white] (-1.5,0) to [out=-90, in=-90, looseness=0.9] (1.5,0);
		\draw[line width=2pt,red!40!white] (-0.5,0) to [out=90, in=90, looseness=0.9] (-3.5,0);
		\draw[line width=2pt,red!40!white] (-3.5,0) to [out=-90, in=-90, looseness=0.9] (-0.5,0);
		\node[red!40!white] at (-2,1){$a$};
		\node[blue!40!white] at (0,1){$b$};
		\draw[blue!40!white,line width=1pt,-{Stealth[length=3mm, width=2mm]}] (-0.2,0.8) -- (-0.21,0.8);
		\draw[red!40!white,line width=1pt,-{Stealth[length=3mm, width=2mm]}] (-2.11,0.8) -- (-2.12,0.8);
	\end{tikzpicture}
	\caption{The basis $\{a,b\}$ of homology group $H_1(\mathcal{S}_1)$ associated with the Riemann surface $\mathcal{S}_1$.}
	\label{abcycle}
\end{figure}
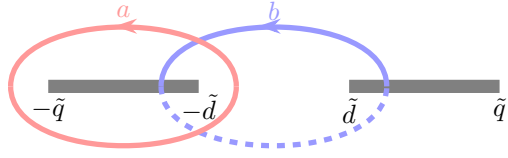
\par
Firstly, introduce the normalized holomorphic differential 
\begin{equation}
	\eta=\frac{-\tilde{q}}{2i K(m)}\cdot \frac{dw}{\sqrt{(w^2-\tilde{d}^2)(w^2-\tilde{q}^2)}}, \quad m=\frac{\sqrt{\tilde{q}^2-\tilde{p}^2}}{\tilde{q}}
\end{equation}
on $\mathcal{S}_1$, which satisfies
\begin{equation}
	\oint_a \eta=2\int_{-\tilde{d}}^{-\tilde{q}} \eta_+=1,\qquad \tau:=\oint_b\eta=2\int_{\tilde{d}}^{-\tilde{d}} \eta=2i\frac{K'(m)}{K(m)}\in i\mathbb{R}_+,
\end{equation}
where $K'(m):=K(m')=K(\sqrt{1-m^2})$.  
\par
Next, define the Riemann-theta function 
\begin{equation}
	\Theta(w)=\Theta(w;\tau)=\sum_{n\in\mathbb{Z}}e^{2\pi inw+\pi in^2\tau}, 
\end{equation}
with the elliptic half-period ratio $\tau$ written as 
\begin{equation}\label{tau}
	\tau=i\frac{K'(\tilde{m})}{K(\tilde{m})},\quad  K(\tilde{m})=\frac{1+m'}{2}K(m),\quad \tilde{m}=\frac{1-m'}{1+m'}. 
\end{equation}
The Riemann-theta function is an entire function of $w$ and satisfies the quasi-doubly periodic properties:
\begin{equation}
	\Theta(w+1)=\Theta(w), \quad\Theta(w+\tau)=e^{-2\pi iw-\pi i\tau}\Theta(w),
\end{equation}
and vanishes at the lattice of half periods:
\begin{equation}
	\Theta(w)=0\iff w\in\frac{1}{2}+\frac{\tau}{2}+\mathbb{Z}+\tau\mathbb{Z}. 
\end{equation}
\par
Then, introduce the Abel-Jacobi embedding with base point $q_r$ as
\begin{equation}
	\mathcal{A}(w):=\int_{q_r}^w \eta, 
\end{equation}
which satisfies the following jump conditions 
\begin{equation}
	\begin{cases}
		\begin{aligned}
			& \mathcal{A}_+(w)+\mathcal{A}_-(w)=\tau, &w\in(-\tilde{q},-\tilde{d}),\\
			&\mathcal{A}_+(w)+\mathcal{A}_-(w)=0, &w\in(\tilde{d},\tilde{q}),\\
			&\mathcal{A}_+(w)-\mathcal{A}_-(w)=-1, &w\in(-\tilde{d}, \tilde{d}),\\
			&\mathcal{A}_+(w)=\mathcal{A}_-(w), & elsewhere,
		\end{aligned}
	\end{cases}
\end{equation}
and it is derived from the symmetry of $\mathcal{A}(w)$ that $\mathcal{A}(0)=-\frac{1}{2}+\frac{\tau}{4}$ and $\mathcal{A(\infty)}=\frac{\tau}{4}$.  
\par
Finally, let us also introduce the function $\alpha:=\alpha(w)=\left( \frac{(w-\tilde{q})(w+\tilde{d})}{(w-\tilde{d})(w
	+\tilde{q})}\right)^{\frac{1}{4}}$, characterized by the following properties: 
\begin{equation}
	\begin{cases}
		\begin{aligned}
			&\alpha(w) \text{ is analytic off } [-\tilde{q},-\tilde{d}]\cup[\tilde{d},\tilde{q}]& \\
			&\alpha_+(w)=i\alpha_-(w),\quad z\in(-\tilde{q},-\tilde{d})\cup(\tilde{d}, \tilde{q}),\\
			&\alpha(w)\to 1, \qquad \qquad  z\to\infty,
		\end{aligned}
	\end{cases}
\end{equation}
then the solution of RHP \ref{rhp_w} is given by 
\begin{equation}\label{mathcal-N-Sol}
	\begin{aligned}
		\mathcal{N}(w)&=\frac{\Theta(0)}{2\Theta\left(\frac{\tilde{\delta}}{\pi}\right)}
		\begin{pmatrix} 
			\left(\alpha(w)+\alpha(w)^{-1}\right)\frac{\Theta\left(\mathcal{A}(w)-\frac{\tilde{\delta}}{\pi}-\frac{\tau}{4}\right)}{\Theta\left(\mathcal{A}(w)-\frac{\tau}{4}\right)} & 
			i\left(\alpha(w)-\alpha(w)^{-1}\right)\frac{\Theta\left(-\mathcal{A}(w)-\frac{\tilde{\delta}}{\pi}-\frac{\tau}{4}\right)}{\Theta\left(-\mathcal{A}(w)-\frac{\tau}{4}\right)}e^{i\tilde{\delta}} \\
			-i\left(\alpha(w)-\alpha(w)^{-1}\right)\frac{\Theta\left(\mathcal{A}(w)-\frac{\tilde{\delta}}{\pi}+\frac{\tau}{4}\right)}{\Theta\left(\mathcal{A}(w)+\frac{\tau}{4}\right)}e^{-i\tilde{\delta}} & 
			\left(\alpha(w)+\alpha(w)^{-1}\right)\frac{\Theta\left(-\mathcal{A}(w)-\frac{\tilde{\delta}}{\pi}+\frac{\tau}{4}\right)}{\Theta\left(-\mathcal{A}(w)+\frac{\tau}{4}\right)}
		\end{pmatrix}.
	\end{aligned}
\end{equation}
\par
\subsubsection{Variable substitution: $w\mapsto z$}\

To proceed further, take a transformation of the complex plane to reduce RHP \ref{rhp_w} to RHP \ref{rhp_dsw} by introducing a change of variable $w\mapsto z$ defined by
\begin{equation}
	w=\sqrt{z^2-q_l^2}. 
\end{equation}
In addition, denote $q_r=\sqrt{\tilde{q}^2-q_l^2}$ and $z_d=\sqrt{\tilde{p}^2-q_l^2}.$ It is clear that the function $w=w(z)$ is analytic off $[-q_l,q_l]$ and satisfies $w_+(z) + w_-(z)=0$ on $(-q_l,q_l)$. 
\par
Next, introduce a matrix 
\begin{equation}
	\mathcal{N}^{(1)}(z)=\frac{1}{2}\begin{pmatrix}
		\gamma(z)+\gamma(z)^{-1}&\gamma(z)-\gamma(z)^{-1}\\
		\gamma(z)-\gamma(z)^{-1}&\gamma(z)+\gamma(z)^{-1}
	\end{pmatrix} \mathcal{N}(w(z)),
\end{equation}
where 
\begin{equation}
	\gamma(z)=\left(\frac{z^2}{z^2-q_l^2}\right)^{\frac{1}{4}}. 
\end{equation}
Then the matrix $\mathcal{N}^{(1)}(z)$ solves the following RHP. 
\par
\begin{rhp}
	Find a $2\times2$ matrix-valued function $\mathcal{N}^{(1)}(z)=\mathcal{N}^{(1)}(x,t;z)$ satisfying: \\
	(1) $\mathcal{N}^{(1)}(z)$ is holomorphic for $z\in\mathbb{C}\backslash([-q_r,-z_d]\cup[-q_l,q_l]\cup [z_d,q_r])$. \\
	(2) For $z\in(-q_r,-z_d)\cup(-q_l,q_l)\cup (z_d,q_r)$, the jump condition is
    \begin{equation}
		\mathcal{N}^{(1)}_{+}(z)=\mathcal{N}^{(1)}_{-}(z)J^{(1)}(z), 
	\end{equation}
	where
	\begin{equation}
		J^{(1)}(z)=\begin{cases}
			\begin{aligned}
				&\begin{pmatrix}
					0&-e^{-i\tilde{\delta}}\\e^{i\tilde{\delta}}&0
				\end{pmatrix}, &z\in(-q_r,-z_d),\\
				&\quad i\sigma_1, &z\in(-q_l,0),\\
				&-i\sigma_1, &z\in(0,q_l),\\
				&\begin{pmatrix}
					0&-e^{i\tilde{\delta}}\\e^{-i\tilde{\delta}}&0
				\end{pmatrix}, &z\in(z_d,q_r).     
			\end{aligned}
		\end{cases}
	\end{equation}
	(3) $\mathcal{N}^{(1)}(z)=I+\mathcal{O}(z^{-1})$ as $z\to \infty.$
\end{rhp}
\par
\begin{rmk}
	To ensure that $\gamma(z)+\gamma(z)^{-1}$ and $\gamma(z)-\gamma(z)^{-1}$ have no poles (except the branch points $\pm q_l$), define $\gamma(z)$ as $\left(\frac{z^2}{z^2-q_l^2}\right)^{\frac{1}{4}}$ instead of $\left(\frac{z-q_l}{z+q_l}\right)^{\frac{1}{4}}$. 
\end{rmk}
\par
In fact, a clear discrepancy persists between the jump conditions for $\mathcal{N}^{(1)}(z)$ and $\mathcal{M}_{dsw}(z)$, which require additional transformation. To do so, introduce a scalar function $\mathcal{D}_u=\mathcal{D}_u(z)$ that is holomorphic for $z\in\mathbb{C}\backslash([-q_r,-z_d]\cup[-q_l,q_l]\cup [z_d,q_r])$ and satisfies the jump conditions: 
\begin{equation}
	\begin{cases}
		\begin{aligned}
			&\mathcal{D}_{u+}\mathcal{D}_{u-}=i, &z\in (-q_l, 0), \\
			&\mathcal{D}_{u+}\mathcal{D}_{u-}=-i, &z\in (0, q_l), \\
			&\mathcal{D}_{u+}\mathcal{D}_{u-}=e^{i\delta_2}, &z\in (-q_r,-z_d), \\
			&\mathcal{D}_{u+}\mathcal{D}_{u-}=e^{-i\delta_2}, &z\in (z_d, q_r),  \\
		\end{aligned}
	\end{cases}
\end{equation}
where $\delta_2$ is determined by the normalization condition $\mathcal{D}_u(z)=1+\mathcal{O}(z^{-1})$ as $z\to\infty$. Explicitly, according to Plemelj's formula, the function $\mathcal{D}_u$ can be written as 
\begin{equation}
	\begin{aligned}
		\mathcal{D}_u(z) = \exp \bigg\{ 
		\frac{\mathcal{R}_d(z)}{2\pi i} \times &\bigg[ 
		\int_{-q_l}^0 \frac{\frac{\pi i}{2}}{(s-z)\mathcal{R}_d(s)} ds 
		+ \int_0^{q_l} \frac{-\frac{\pi i}{2}}{(s-z)\mathcal{R}_d(s)} ds \\
		&+\int_{-q_r}^{-z_d} \frac{i\delta_2}{(s-z)\mathcal{R}_d(s)} ds 
		+ \int_{z_d}^{q_r} \frac{-i\delta_2}{(s-z)\mathcal{R}_d(s)} ds 
		\bigg] \bigg\},
	\end{aligned}
\end{equation}
and the normalization above implies that
\begin{equation}\label{delta2}
	\delta_2=\frac{\pi}{2}\frac{\int_0^{q_l}\frac{sds}{\mathcal{R}_d(s)}}{\int_{z_d}^{q_r}\frac{sds}{\mathcal{R}_d(s)}}=\pi \int_{0}^{iq_l}\eta=\frac{\pi}{2}\cdot \frac{F(\frac{q_l}{z_d},m)}{K(m)}, 
\end{equation}
where $F(\frac{q_l}{z_d},m)=\int_{0}^{q_l/z_d}\frac{d\theta}{\sqrt{(1-t^2)(1-m^2t^2)}}$ is the incomplete elliptic integral of the first kind. 
\par
Further, define $\mathcal{N}^{(2)}(z):=\mathcal{N}^{(1)}(z)\mathcal{D}_u^{-\sigma_3}(z)$, then the jump condition for $\mathcal{N}^{(2)}(z)$ is given by
\begin{equation}
	\mathcal{N}^{(2)}_{+}(z)=\mathcal{N}^{(2)}_{-}(z)J^{(2)}(z), 
\end{equation}
where 
\begin{equation}
	J^{(2)}(z)=\begin{cases}
		\begin{aligned}
			&\begin{pmatrix}
				0&-e^{-i\delta}\\e^{i\delta}&0
			\end{pmatrix}, &z\in(-q_r,-z_d),\\
			&-i\sigma_2, &z\in (-q_l,q_l), \\
			&\begin{pmatrix}
				0&-e^{i\delta}\\e^{-i\delta}&0
			\end{pmatrix}, &z\in(z_d,q_r).
		\end{aligned}
	\end{cases}
\end{equation}
\par
However, $\mathcal{N}^{(2)}(z)$ still does not coincide with the desired matrix function $\mathcal{M}_{dsw}(z)$, since it has poles at the point $z=0$, which occurs due to the fact that $D_u(z)$ is decaying as $\sqrt{z}$ as $z\to0+$ (that is, from $\Im z>0$) and is increasing as $1/\sqrt{z}$ as $z\to 0-$ (that is, from $\Im z<0$). Thus, the first column of the $\mathcal{N}^{(2)}(z)$ has a pole of the first order when $z\to 0+$ and the second column of the $\mathcal{N}^{(2)}(z)$ has a pole of the second order when $z\to 0-$. 
\par
To remove the singularity of $\mathcal{N}^{(2)}(z)$ at $z=0$, introduce a matrix-valued function of the form
\begin{equation}\label{N3}
	\mathcal{N}^{(3)}(z):=\begin{pmatrix}
		1+\frac{\sigma}{z}&\frac{\sigma}{z}\\
		-\frac{\sigma}{z}&1-\frac{\sigma}{z}
	\end{pmatrix} \mathcal{N}^{(2)}(z)
\end{equation}
with 
\begin{equation}\label{uuu}
	\sigma=-\frac{iq_l}{2}\frac{\mathcal{N}_{11}(iq_l)-\mathcal{N}_{21}(iq_l)}{\mathcal{N}_{11}(iq_l)+\mathcal{N}_{21}(iq_l)}, 
\end{equation}
then it can be shown that $\mathcal{N}^{(3)}(z)$ does not have pole at $z=0$ or elsewhere (except for the branch points) and satisﬁes the same jump condition as $\mathcal{N}^{(2)}(z)$.  Hence, we arrive at the following lemma. 
\par
\begin{lem}
	The matrix-valued function $\mathcal{N}^{(3)}(z)$ defined in (\ref{N3}) is the unique solution to the RHP \ref{rhp_dsw}. 
\end{lem}
\par

\subsection{The local parametrices for $\mathcal{M}_d^{(l)}$ and $\mathcal{M}_d^{(r)}$}\

For a sufficiently small real number $\varepsilon''>0$,
define two small disks
\begin{equation}
	\mathcal{U}_d^{-}:=\{ z\in\mathbb{C}\Big||z+z_d|<\varepsilon''\},\quad \mathcal{U}_d^{+}:=\{ z\in\mathbb{C}\Big||z-z_d|<\varepsilon''\},
\end{equation}
around $-z_d$ and $z_d$, respectively. Then the local model problem near $z_d$ is given below. 
\par
\begin{rhp}
	Find a $2\times2$ matrix-valued function $\mathcal{M}_d^{(r)}(z)=\mathcal{M}_d^{(r)}(x,t;z)$ satisfying: \\
	(1) $\mathcal{M}_d^{(r)}(z)$ is holomorphic for $z\in\mathbb{C}\backslash\left(\bigcup_{j=1}^2(\Gamma_d^{(j)}\cup\overline{\Gamma}_d^{(j)})\cap\mathcal{U}_d^{(r)}\right)$. \\
	(2) For $z\in\bigcup_{j=1}^2(\Gamma_d^{(j)}\cup\overline{\Gamma}_d^{(j)})\cap\mathcal{U}_d^{(r)}$, the jump condition is 
    \begin{equation}
		\mathcal{M}^{(r)}_{d+}(z)=\mathcal{M}^{(r)}_{d-}(z)\mathcal{J}_d^{(r)}(z), 
	\end{equation}
	where 
	\begin{equation}
		\mathcal{J}_d^{(r)}(z)=\begin{cases}
			\begin{aligned}
				&\mathcal{J}_d^{(2)}(z), &z\in(\Gamma_d^{(2)}\cup\overline{\Gamma}_d^{(2)})\cap\mathcal{U}_d^{(r)},\\
				&\begin{pmatrix}
					\frac{\mathcal{D}_{d-}e^{2itg_{d+}}}{\mathcal{D}_{d+}a_+a^*_-}&-e^{i\delta}\\e^{-i\delta}&0
				\end{pmatrix}, &z\in(z_d,z_d+\varepsilon'').\\
			\end{aligned}
		\end{cases}
	\end{equation}
	(3) $\mathcal{M}_d^{(r)}(z)=I+\mathcal{O}(z^{-1})$ as $z\to \infty$. 
\end{rhp}
In fact, due to the symmetry from $-z$ to $z$, the local model problem near $-z_d$ can be given by 
\begin{equation}
	\mathcal{M}_d^{(l)}(z)=\sigma_1\mathcal{M}_d^{(r)}(-z)\sigma_1, 
\end{equation}
so it is enough for us to consider the local model problem around $z_d$ only. 
\par
\begin{rmk}
	Here, to match the jump contour of the Airy model in Appendix \ref{Appendix-B}, we replace the jump contour on the right side of $z_d$ with the one that opens at $z_d+\varepsilon''$. Based on the signs of $\Im g_d$, the resulting error is exponentially small. 
\end{rmk}
\par
To connect the above problem with the Airy parametrix, introduce the conformal mapping from $\mathcal{U}_d^{(r)}$ to a neighborhood of the origin: 
\begin{equation}
    \frac{2}{3}\kappa^{3/2}(z)=-i sgn (\Im z)\left[g_d(z)-g_d(z_d)\right], 
\end{equation}
and the piecewise function 
\begin{equation}
	h(z):=\begin{cases}
		\begin{aligned}
			&\mathcal{D}_d(z)\left(\frac{r^*(z)}{1-r(z)r^*(z)}\right)^{\frac{1}{2}}e^{\frac{i\delta_0}{2}},&z\in \mathcal{U}_d^{(r)}\cap\mathbb{C}_+,\\
			&\mathcal{D}^{-1}_d(z)\left(\frac{r(z)}{1-r(z)r^*(z)}\right)^{\frac{1}{2}}e^{-\frac{i\delta_0}{2}},&z\in \mathcal{U}_d^{(r)}\cap\mathbb{C}_-,
		\end{aligned}
	\end{cases}
\end{equation}
which satisfies the jump conditions
\begin{equation}\label{jump_h}
	h_+(z)h_-(z)=1, \quad z\in(z_d-\varepsilon'',z_d); \qquad h_+(z)=h_-(z)e^{i\delta}, \quad z\in (z_d,z_d+\varepsilon'').  
\end{equation}
Then the jump matrix $\mathcal{J}_d^{(r)}(z)$ takes the following deformation: 
\begin{equation}
	\mathcal{J}_d^{(r)}(z)=
	\begin{cases}
		\begin{aligned}
			& h_-^{-\sigma_3}\sigma_2h_+^{\sigma_3}\sigma_2, &z\in(z_d-\varepsilon'',z_d),\\
			&h_-^{-\sigma_3}\begin{pmatrix}1&e^{-\frac{4}{3}\zeta_d^{\frac{3}{2}}}\\0&1\end{pmatrix}h_+^{\sigma_3}\cdot(-i\sigma_2),&z\in(z_d,z_d+\varepsilon''),\\
			&h^{-\sigma_3}\begin{pmatrix}1&0\\e^{\frac{4}{3}\zeta_d^{\frac{3}{2}}}&1\end{pmatrix}h^{\sigma_3}, &z\in\overline{\Gamma}_d^{(2)}\cap\mathcal{U}_d^{(r)},\\
			& \sigma_2h^{-\sigma_3}\begin{pmatrix}1&0\\e^{\frac{4}{3}\zeta_d^{\frac{3}{2}}}&1\end{pmatrix}h^{\sigma_3}\sigma_2, &z\in \Gamma_d^{(2)}\cap\mathcal{U}_d^{(r)},
		\end{aligned}
	\end{cases}
\end{equation}
where $\zeta_d(z)=t^{2/3}\kappa(z)$. So the local model problem for $\mathcal{M}_d^{(r)}(z)$ is given by 
\begin{equation}
	\mathcal{M}_d^{(r)}(z)=\begin{cases}
		\begin{aligned}
			&\mathcal{H}(z)\mathcal{M}^{(Ai)}\left(\zeta_d(z)\right)h(z)^{\sigma_3}\cdot(-i\sigma_2), &z\in \mathcal{U}_d^{(r)}\cap\mathbb{C}_+, \\
			&\mathcal{H}(z)\mathcal{M}^{(Ai)}\left(\zeta_d(z)\right)h(z)^{\sigma_3}, &z\in \mathcal{U}_d^{(r)}\cap\mathbb{C}_+, 
		\end{aligned}
	\end{cases}
\end{equation}
with 
\begin{equation}
	\mathcal{H}(z):=\begin{cases}
		\begin{aligned}
			&\mathcal{M}_{dsw}^{(\infty)}(z)\cdot(i\sigma_2)\cdot h^{-\sigma_3}e^{\frac{1}{4}i\pi\sigma_3}\frac{1}{\sqrt{2}}\begin{pmatrix}1&-1\\1&1\end{pmatrix}\zeta_d(z)^{\frac{\sigma_3}{4}},& z\in \mathcal{U}_d^{(r)}\cap\mathbb{C}_+,\\
			&\mathcal{M}_{dsw}^{(\infty)}(z)h^{-\sigma_3}e^{\frac{1}{4}i\pi\sigma_3}\frac{1}{\sqrt{2}}\begin{pmatrix}1&-1\\1&1\end{pmatrix}\zeta_d(z)^{\frac{\sigma_3}{4}}, & z\in \mathcal{U}_d^{(r)}\cap\mathbb{C}_-. 
		\end{aligned}
	\end{cases}
\end{equation}
In addition, direct calculation shows that $\mathcal{H}(z)$ is analytic in $\mathcal{U}_d^{(r)}$.
\par
\subsection{The small-norm RHP for $\mathcal{M}_{dsw}^{(err)}$}\

Finally, define the small-norm RHP as 
\begin{equation}\label{err_dsw}
	\mathcal{M}_d^{(err)}(x,t;z)=\begin{cases}
		\begin{aligned}
			& \widetilde{\mathcal{M}}_d(x,t;z)\left[\mathcal{M}_{dsw}^{(\infty)}(x,t;z)\right]^{-1}, &z\in \mathbb{C}\backslash \left(\mathcal{U}_d^{(l)}\cup \mathcal{U}_d^{(r)}\right), \\
			& \widetilde{\mathcal{M}}_d(x,t;z)\left[\mathcal{M}_{d}^{(l)}(x,t;z)\right]^{-1}, &z\in \mathcal{U}_d^{(l)}, \\
			& \widetilde{\mathcal{M}}_d(x,t;z)\left[\mathcal{M}_{d}^{(r)}(x,t;z)\right]^{-1}, &z\in \mathcal{U}_d^{(r)}, 
		\end{aligned}
	\end{cases}
\end{equation}
with the jump contour shown in Figure \ref{e2} and jump matrix given by 
\begin{equation}
	\mathcal{J}_d^{(err)}(x,t;z)=\begin{cases}
		\begin{aligned}
			& \mathcal{N}^{(3)}(z)\widetilde{\mathcal{J}}_d(x,t;z)\left[\mathcal{N}^{(3)}(z)\right]^{-1}, &z\in \bigcup_{j=1}^4(\Gamma_d^{(j)}\cup\overline{\Gamma}_d^{(j)})\backslash \left(\mathcal{U}_d^{(l)}\cup \mathcal{U}_d^{(r)}\right), \\
			& \mathcal{M}_{d}^{(l)}(x,t;z)\left[\mathcal{N}^{(3)}(z)\right]^{-1}, &z\in\partial \mathcal{U}_d^{(l)},\\
			& \mathcal{M}_{d}^{(r)}(x,t;z)\left[\mathcal{N}^{(3)}(z)\right]^{-1}, &z\in\partial \mathcal{U}_d^{(r)}. 
		\end{aligned}
	\end{cases}
\end{equation}
\par
\begin{figure}[htbp]
	\centering
	\begin{tikzpicture}[scale=1]
		% 第一条线
		\draw[line width=1pt,-{Stealth[length=2mm, width=1.5mm]}] (0,-2.595) -- (0.9,-1.038) node[pos=0.5, right, font=\small, red!80!white]{$\overline{\Gamma}_d^{(2)}\backslash\mathcal{U}_d^{+}$};
		\draw[line width=1pt] (0.6,-1.557) -- (1.5,0) ;
		\draw[line width=1pt,-{Stealth[length=2mm, width=1.5mm]}] (2,0) -- (3,0.5); 
		\draw[line width=1pt] (2.9,0.45) -- (5,1.5) node[pos=0.5, above=2mm, font=\small, red!80!white]{$\Gamma_d^{(1)}\backslash\mathcal{U}_d^{+}$}; 
		% 第二条线
		\draw[line width=1pt,-{Stealth[length=2mm, width=1.5mm]}] (0,2.595) -- (0.9,1.038) node[pos=0.5, right, font=\small, red!80!white]{$\Gamma_d^{(2)}\backslash\mathcal{U}_d^{+}$};
		\draw[line width=1pt] (0.6,1.557) -- (1.5,0); 
		\draw[line width=1pt,-{Stealth[length=2mm, width=1.5mm]}] (2,0) -- (3,-0.5); 
		\draw[line width=1pt] (2.9,-0.45) -- (5,-1.5) node[pos=0.5, below=2mm, font=\small, red!80!white]{$\overline{\Gamma}_d^{(1)}\backslash\mathcal{U}_d^{+}$}; 
		% 第三条线
		\draw[line width=1pt,-{Stealth[length=2mm, width=1.5mm]}] (-5,-1.5) -- (-2.9,-0.45) node[pos=0.7, below=2mm, font=\small, red!80!white]{$\overline{\Gamma}_d^{(4)}\backslash\mathcal{U}_d^{-}$};
		\draw[line width=1pt] (-3,-0.5) -- (-2,0); 
		\draw[line width=1pt,-{Stealth[length=2mm, width=1.5mm]}] (-1.5,0) -- (-0.6,1.557) ; 
		\draw[line width=1pt] (-0.9,1.038) -- (0,2.595) node[pos=0.5, left, font=\small, red!80!white]{$\Gamma_d^{(3)}\backslash\mathcal{U}_d^{-}$}; 
		% 第四条线
		\draw[line width=1pt,-{Stealth[length=2mm, width=1.5mm]}] (-5,1.5) -- (-2.9,0.45) node[pos=0.5, above=2mm, font=\small, red!80!white]{$\Gamma_d^{(4)}\backslash\mathcal{U}_d^{-}$};
		\draw[line width=1pt] (-3,0.5) -- (-2,0); 
		\draw[line width=1pt,-{Stealth[length=2mm, width=1.5mm]}] (-1.5,0) -- (-0.6,-1.557); 
		\draw[line width=1pt] (-0.9,-1.038) -- (0,-2.595) node[pos=0.5, left, font=\small, red!80!white]{$\overline{\Gamma}_d^{(3)}\backslash\mathcal{U}_d^{-}$}; 
		% 掏空
		\draw[fill=gray!40!white, draw=white, line width=0.05pt] (1.5,0) circle (0.5cm);
		\draw[fill=gray!40!white, draw=white, line width=0.05pt] (-1.5,0) circle (0.5cm);
		\draw[line width=1pt,dashed, draw=gray] (-4.5,0) -- (4.5,0) node[pos=1, below, font=\small]{$\mathbb{R}$};
		% RRR
		\draw[line width=1pt,-{Stealth[length=2mm, width=1.5mm]}] (2,0) -- (2.5,0);
		\draw[line width=1pt] (2.4,0) -- (3,0) node[pos=1,below, font=\small] {$q_r$};
		\draw[line width=1pt,-{Stealth[length=2mm, width=1.5mm]}] (-3,0) -- (-2.5,0) node[pos=0,below=-1mm, font=\small] {$-q_r$};
		\draw[line width=1pt] (-2.6,0) -- (-2,0);
		\draw[line width=1pt,-{Stealth[length=2mm, width=1.5mm]}] (-0.8,0) -- (0,0) node[pos=0, below=-0.5mm, font=\small] {$-q_l$};
		\draw[line width=1pt] (-0.15,0) -- (0.8,0) node[pos=1, below, font=\small] {$q_l$};
		% 标注点
		\draw[fill=black, draw=black, line width=0.05pt] (0.8,0) circle (0.03cm);
		\draw[fill=black, draw=black, line width=0.05pt] (1.5,0) circle (0.03cm);
		\draw[fill=black, draw=black, line width=0.05pt] (-0.8,0) circle (0.03cm);
		\draw[fill=black, draw=black, line width=0.05pt] (-1.5,0) circle (0.03cm);
		% 圆
		\draw[line width=1pt] (1.5,0) circle (0.5cm) node[above=5mm, font=\small] {$\partial \mathcal{U}_d^{+}$}; 
		\draw[line width=0.01pt,-{Stealth[length=2mm, width=1.5mm]}](1.42,0.5) -- (1.41,0.5);
		\draw[line width=1pt] (-1.5,0) circle (0.5cm) node[above=5mm, font=\small] {$\partial \mathcal{U}_d^{-}$};
		\draw[line width=0.01pt,-{Stealth[length=2mm, width=1.5mm]}](-1.61,0.5) -- (-1.62,0.5);
		\draw[fill=black, draw=black, line width=0.05pt] (-1.5,0) circle (0.03cm);
		\draw[fill=black, draw=black, line width=0.05pt] (1.5,0) circle (0.03cm);
		\node[font=\small] at (1.5,-0.2) {$z_d$}; 
		\node[font=\small] at (-1.5,-0.15) {$-z_d$};
	\end{tikzpicture}
	\caption{The jump contour of $\mathcal{M}_d^{(err)}(z)$}
	\label{e2}
\end{figure}
\par
According to result of the Airy parametrix in Appendix \ref{Appendix-B}, it is shown that 
\begin{equation}
	\mathcal{J}_d^{(err)}(z)=I+\mathcal{O}(\zeta_d^{-\frac{3}{2}})=I+\mathcal{O}(t^{-1}), 
\end{equation}
which implies that  
\begin{equation}
	||\langle\cdot\rangle (\mathcal{J}_d-I)||_{L^p(\Sigma)}=\begin{cases}
		\begin{aligned}
			& \mathcal{O}(t^{-1}),&\Sigma=\partial \mathcal{U}_d^{(l)}\cup\partial \mathcal{U}_d^{(r)},\\
			&\mathcal{O}(e^{-ct}), &\Sigma=\left(\Gamma_r\cup\overline{\Gamma}_r\right)\backslash \left(\mathcal{U}_d^{(l)}\cup \mathcal{U}_d^{(r)}\right), 
		\end{aligned}
	\end{cases}
\end{equation}
with $\langle z\rangle:=\sqrt{1+|z|^2}$ for $p\ge1$ and some $c>0$. Thus,  the small norm theory for Riemann–Hilbert problems tells us that 
\begin{equation}
	\mathcal{M}_d^{(err)}=I+\frac{\mathcal{M}_{d,1}^{(err)}(x,t)}{z}+\mathcal{O}(z^{-2}),\qquad \mathcal{M}_{d,1}^{(err)}(x,t)=\mathcal{O}(t^{-1}). 
\end{equation}
\par
Therefore, the formulas (\ref{construct}) and (\ref{err_dsw}) imply that the potential function $q(x,t)$ is formulated by
\begin{equation}\label{redsw}
	q(x,t)=2i \lim_{z\to\infty}(z\mathcal{M}^{(\infty)}_{dsw}(z))_{12}+\left(\mathcal{M}_{d,1}^{(err)}(x,t)\right)_{12}=2i \lim_{z\to\infty} z \mathcal{N}_{12}(w(z))+2i\sigma+\mathcal{O}(t^{-1}), 
\end{equation}
where $\mathcal{N}_{12}(w(z))$ is the (1,2) entry of matrix $\mathcal{N}(w(z))$ in (\ref{mathcal-N-Sol}). The calculation of $\sigma$ is non-trivial and will be given in detail in Appendix \ref{Appendix-C}. This yields the leading-order term $q_{dsw}(x,t)$ in (\ref{qdsw}) and the proof of Theorem \ref{main} for Region \RNum{2} is finished. 
\par

\section{Region \RNum{3} for $x>-4q_l^2t-2q_r^2t+\varepsilon$}\label{section6}

In Region \RNum{3} for $x>-4q_l^2t-2q_r^2t+\varepsilon$, the branch point $z_d$ of the $g$-function defined in (\ref{gd}), which varies with the parameter $\xi$, is located to the left of $q_l$. Then the $g$-function possesses only one cut on the interval $I_r=[-q_r,q_r]$, corresponding to the right plane wave as $t\to \infty$. Thus, by setting $z_1=q_r$, the $g$-function defined in (\ref{ggg0}) takes the form
\begin{equation}\label{g2}
	g_r(\xi;z):=(4z^2+12\xi+2q_r^2)\mathcal{R}_r(z), \quad
	\xi>-\frac{q_r^2}{3}-\frac{q_r^2}{6}. 
\end{equation}
A straightforward calculation reveals that $\pm z_r(\xi)$, where $z_r(\xi)=\sqrt{-3\xi-\frac{q_r^2}{2}}$, are the two zeros of $g_r(\xi;z)$. The location of $z_r(\xi)$ is characterized as follows: 
\begin{equation}
	\begin{cases}
		\begin{aligned}
			& z_r(\xi)\in(0,q_l), &\text{if } -\frac{q_l^2}{3}-\frac{q_r^2}{6}<\xi<-\frac{q_r^2}{6},\\
			&z_r(\xi)\in i\mathbb{R_+}, &\text{if }  \xi>-\frac{q_r^2}{6}. 
		\end{aligned}
	\end{cases}
\end{equation}
The signature tables of $\Im g_r(z)$ are illustrated in Figure \ref{g3} and \ref{g4} for $-\frac{q_l^2}{3}-\frac{q_r^2}{6}<\xi<-\frac{q_r^2}{6}$ and $\xi>-\frac{q_r^2}{6}$, respectively. Additionally, it should be noted that if $-\frac{q_l^2}{3}-\frac{q_r^2}{6}<\xi<-\frac{q_r^2}{6}$, it holds that $\Im g_{r+}>0$ on $(-q_r,-z_r)\cup(z_r,q_r)$, while $\Im g_{r+}<0$ on $(-z_r,z_r)$; if $\xi>-\frac{q_r^2}{6}$, it holds that $\Im g_{r+}>0$ on $(-q_r,q_r)$. 
\par
\subsection{Opening of lenses}\

For different $\xi$ in Region \RNum{3}, introduce two transformations below
\begin{equation}
	\widetilde{\mathcal{M}}_r(x,t;z)=\mathcal{M}(x,t;z)\mathcal{G}_r(x,t;z), \quad -\frac{q_l^2}{3}-\frac{q_r^2}{6}<\xi<-\frac{q_r^2}{6}, 
\end{equation}
\begin{equation}
	\widetilde{\mathcal{M}}_p(x,t;z)=\mathcal{M}(x,t;z)\mathcal{G}_p(x,t;z), \quad \xi>-\frac{q_r^2}{6}, 
\end{equation}
where the matrix functions $\mathcal{G}_r(x,t;z)$ and $\mathcal{G}_p(x,t;z)$ are defined by
\begin{equation}
	\mathcal{G}_r(x,t;z):=\begin{cases}
		\begin{aligned}
			&\begin{pmatrix}
				1&0\\-r_+e^{2itg_{r+}}&1
			\end{pmatrix},&z\in\Omega_r^{(1)}\cup\Omega_r^{(2)},\\
			&\begin{pmatrix}
				1&-r_-^*e^{-2itg_{r-}}\\0&1
			\end{pmatrix}, &z\in\overline{\Omega}_r^{(1)}\cup\overline{\Omega}_r^{(2)},\\
			&\quad I, &elsewhere,
		\end{aligned}
	\end{cases} 
\end{equation}
and
\begin{equation}
	\mathcal{G}_p(x,t;z):=\begin{cases}
		\begin{aligned}
			&\begin{pmatrix}
				1&0\\-re^{2itg_r}&0
			\end{pmatrix}, &z\in\Omega_p&
			\\
			&\begin{pmatrix}
				1&-r^*e^{-2itg_r}\\0&1
			\end{pmatrix}, & z\in\overline{\Omega}_p, &\\
			&\quad I,\qquad  &elsewhere, & 
		\end{aligned}
	\end{cases}
\end{equation}
with the domains $\Omega_r^{(j)}$, $\overline{\Omega}_r^{(j)}~(j=1,2)$ and $\Omega_p, \overline{\Omega}_p$ illustrated in Figure \ref{o3} and \ref{o4}, respectively. Thus, functions $\widetilde{M}_r(x,t;z)$ and $\widetilde{M}_p(x,t;z)$ satisfy the following two RHPs.  
\par
\begin{rhp}
	Find a $2\times2$ matrix-valued function $\widetilde{\mathcal{M}}_r(z)=\widetilde{\mathcal{M}}_r(x,t;z)$ satisfying: \\
	(1) $\widetilde{\mathcal{M}}_r(z)$ is holomorphic for $z\in \mathbb{C}\backslash\left(\bigcup_{j=1}^2(\Gamma_r^{(j)}\cup\overline{\Gamma}_r^{(j)})\cup[-q_r,q_r]\right)$. \\
	(2) For $z\in\bigcup_{j=1}^2(\Gamma_r^{(j)}\cup\overline{\Gamma}_r^{(j)})\cup[-q_r,q_r]$, the jump condition is
	\begin{equation}
		\widetilde{\mathcal{M}}_{r+}(x,t;z)=\widetilde{\mathcal{M}}_{r-}(x,t;z)\widetilde{\mathcal{J}}_r(x,t;z), 
	\end{equation}
	where 
	\begin{equation}
		\widetilde{\mathcal{J}}_r(x,t;z)=\begin{cases}
			\begin{aligned}
				&\begin{pmatrix}
					1&0\\r_+e^{2itg_{r+}}&1
				\end{pmatrix},&z\in\Gamma_r^{(1)}\cup\Gamma_r^{(2)},\\
				&\begin{pmatrix}
					1&-r_-^*e^{-2itg_{r+}}\\0&1
				\end{pmatrix} ,& z\in\overline{\Gamma}_r^{(1)}\cup\overline{\Gamma}_r^{(2)},\\
				&-i\sigma_2, & z\in(-q_r,q_r). 
			\end{aligned}
		\end{cases}
	\end{equation}
	(3) $\widetilde{\mathcal{M}}_r(z)=I+\mathcal{O}(z^{-1})$ as $z\to\infty$ in $\mathbb{C}\backslash\left(\bigcup_{j=1}^2(\Gamma_r^{(j)}\cup\overline{\Gamma}_r^{(j)})\right)$. \\
	(4) Symmetries: 
	\begin{equation}
		\widetilde{\mathcal{M}}_r(z)=\overline{\widetilde{\mathcal{M}}_r(-\bar{z})}=\sigma_1\widetilde{\mathcal{M}}_r(-z)\sigma_1=\sigma_1\overline{\widetilde{\mathcal{M}}_r(\bar{z})}\sigma_1. 
	\end{equation}
\end{rhp}
\par
\begin{figure}[ht]
	\begin{minipage}[b]{0.42\textwidth}  % 确保有单位（如 \textwidth）
		\centering
		\includegraphics[width=\linewidth]{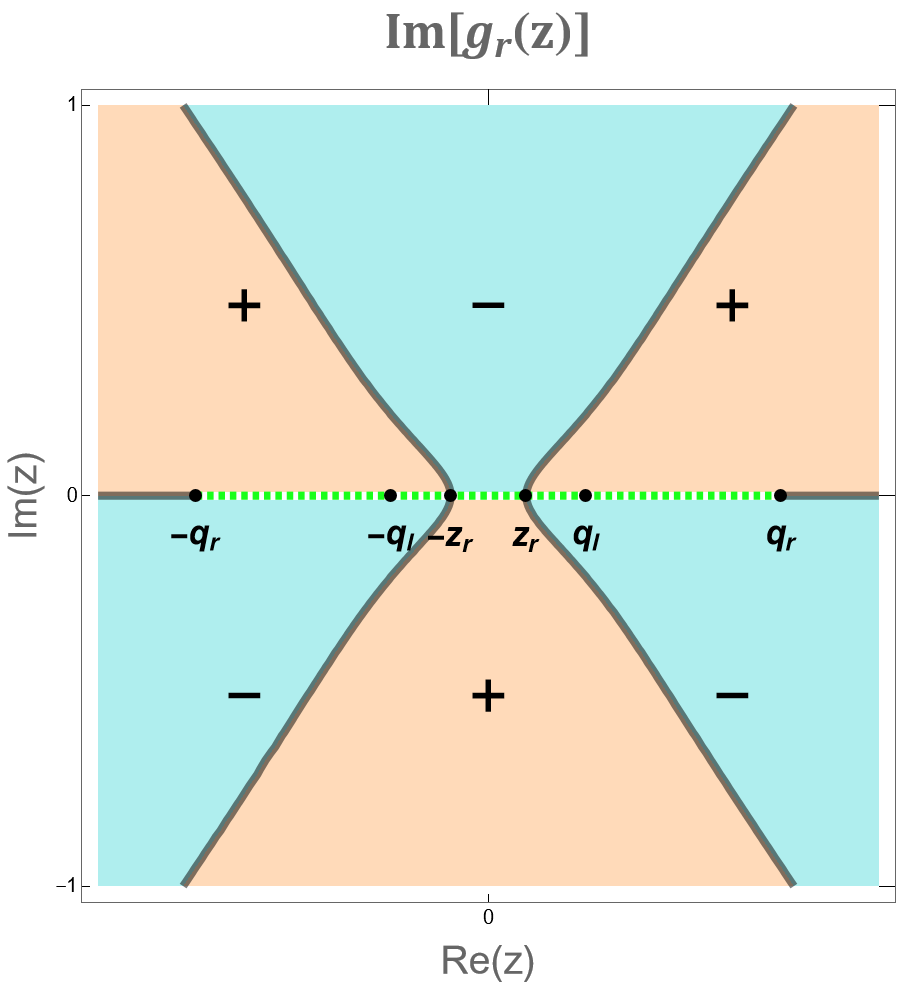}  % \linewidth 正确
		\caption{The signs of $\Im g_r(\xi;z)$ for $-\frac{q_l^2}{3}-\frac{q_r^2}{6}<\xi<-\frac{q_r^2}{6}$.}
		\label{g3}
	\end{minipage} 
	\begin{minipage}[b]{0.56\textwidth}  % 确保宽度总和不超过 1.0\textwidth
		\centering
		\begin{tikzpicture}[scale=1]
			\draw[line width=1pt,dashed, draw=gray] (-4.5,0) -- (4.5,0) node[pos=1, below, font=\small]{$\mathbb{R}$};
			\draw[line width=1pt,-{Stealth[length=2mm, width=1.5mm]}] (-1.5,0) -- (0.15,0) node[pos=-0.15,above=1mm, font=\small] {$-q_l$} node[pos=0.5, below=-0.5mm, font=\small] {$-z_r$};
			\draw[line width=1pt] (-0.15,0) -- (1.5,0) node[pos=0.5, below, font=\small] {$z_r$} node[pos=1.05, above=1mm, font=\small] {$q_l$};
			% 第一条线
			\draw[line width=1pt,-{Stealth[length=2mm, width=1.5mm]}] (1.5,0) -- (2.5,0.5); 
			\draw[line width=1pt] (2.4,0.45) -- (4.5,1.5) node[pos=0.5, above, font=\small, red!80!white]{$\Gamma_r^{(1)}$}; 
			% 第二条线
			\draw[line width=1pt,-{Stealth[length=2mm, width=1.5mm]}] (1.5,0) -- (2.5,-0.5); 
			\draw[line width=1pt] (2.4,-0.45) -- (4.5,-1.5) node[pos=0.5, below, font=\small, red!80!white]{$\overline{\Gamma}_r^{(1)}$}; 
			% 第三条线
			\draw[line width=1pt,-{Stealth[length=2mm, width=1.5mm]}] (-4.5,-1.5) -- (-2.4,-0.45) node[pos=0.7, below, font=\small, red!80!white]{$\overline{\Gamma}_r^{(2)}$};
			\draw[line width=1pt] (-2.5,-0.5) -- (-1.5,0); 
			% 第四条线
			\draw[line width=1pt,-{Stealth[length=2mm, width=1.5mm]}] (-4.5,1.5) -- (-2.4,0.45) node[pos=0.5, above, font=\small, red!80!white]{$\Gamma_r^{(2)}$};
			\draw[line width=1pt] (-2.5,0.5) -- (-1.5,0); 
			% RRR
			\draw[line width=1pt,-{Stealth[length=2mm, width=1.5mm]}] (1.5,0) -- (2.4,0);
			\draw[line width=1pt] (2.1,0) -- (3,0) node[pos=1,below, font=\small] {$q_r$};
			\draw[line width=1pt,-{Stealth[length=2mm, width=1.5mm]}] (-3,0) -- (-2.1,0) node[pos=0,below=-1mm, font=\small] {$-q_r$};
			\draw[line width=1pt] (-2.4,0) -- (-1.5,0);
			% 标注点
			\draw[fill=black, draw=black, line width=0.05pt] (0.8,0) circle (0.03cm);
			\draw[fill=black, draw=black, line width=0.05pt] (1.5,0) circle (0.03cm);
			\draw[fill=black, draw=black, line width=0.05pt] (-0.8,0) circle (0.03cm);
			\draw[fill=black, draw=black, line width=0.05pt] (-1.5,0) circle (0.03cm);
			\path (-5,-3.5) rectangle (5,3.5);
			% 区域
			\node[font=\small,orange!40!black] at (3.5,0.5) {$\Omega_r^{(1)}$};
			\node[font=\small,orange!40!black] at (-3.5,0.5) {$\Omega_r^{(2)}$};
			\node[font=\small,blue!40!black] at (3.5,-0.5) {$\overline{\Omega}_r^{(1)}$};
			\node[font=\small,blue!40!black] at (-3.5,-0.5) {$\overline{\Omega}_r^{(2}$};
		\end{tikzpicture}
		\caption{The opening of lenses in Region \RNum{3} for $-\frac{q_l^2}{3}-\frac{q_r^2}{6}<\xi<-\frac{q_r^2}{6}$.}
		\label{o3}
	\end{minipage}
\end{figure}
\par
\begin{figure}[ht]
	\begin{minipage}[b]{0.42\textwidth}  % 确保有单位（如 \textwidth）
		\centering
		\includegraphics[width=\linewidth]{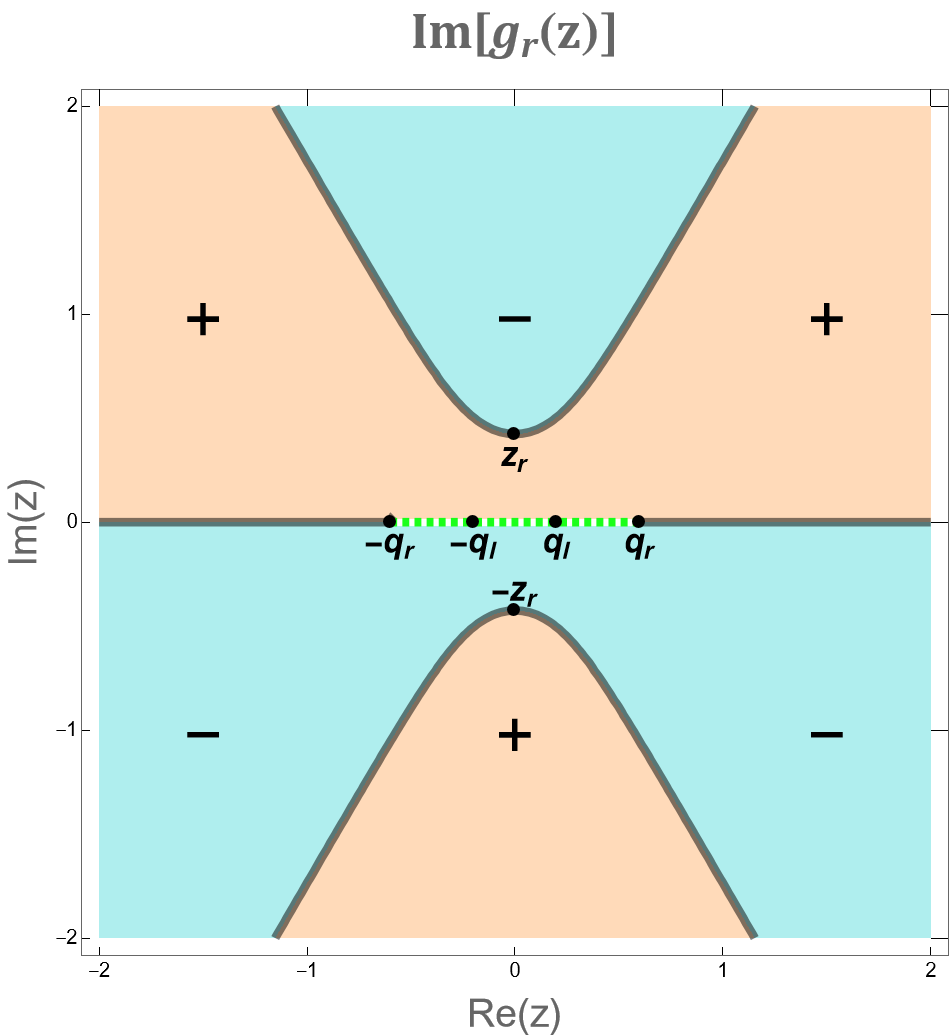}  % \linewidth 正确
		\caption{The signs of $\Im g_r(\xi;z)$ for $\xi>-\frac{q_r^2}{6}$.}
		\label{g4}
	\end{minipage}
	\begin{minipage}[b]{0.56\textwidth}  % 确保宽度总和不超过 1.0\textwidth
		\centering
		\begin{tikzpicture}[scale=1]
			\draw[line width=1pt,dashed, draw=gray] (-4,0) -- (4,0) node[pos=1, below, font=\small]{$\mathbb{R}$};
			\draw[line width=1pt,-{Stealth[length=2mm, width=1.5mm]}] (-4,1) -- (0.5,1);
			\draw[line width=1pt,-{Stealth[length=2mm, width=1.5mm]}] (-4,-1) -- (0.5,-1);
			\draw[line width=1pt] (0,1) -- (4,1) node[pos=1,below, font=\small,red!80!white] {$\Gamma_p$};
			\draw[line width=1pt] (0,-1) -- (4,-1) node[pos=1,below, font=\small, red!80!white] {$\overline{\Gamma}_p$};
			% 标注点
			\draw[fill=black, draw=black, line width=0.05pt] (0.8,0) circle (0.03cm); \node[font=\small,black] at (0.8,-0.2) {$q_l$};
			\draw[fill=black, draw=black, line width=0.05pt] (1.5,0) circle (0.03cm);
			\node[font=\small,black] at (1.5,-0.2) {$q_r$};
			\draw[fill=black, draw=black, line width=0.05pt] (-0.8,0) circle (0.03cm);
			\node[font=\small,black] at (-0.8,-0.2) {$-q_l$};
			\draw[fill=black, draw=black, line width=0.05pt] (-1.5,0) circle (0.03cm);
			\node[font=\small,black] at (-1.5,-0.2) {$-q_r$};
			\draw[fill=black, draw=black, line width=0.05pt] (0,2) circle (0.03cm); \node[font=\small,black] at (0.2,2) {$z_r$};
			\draw[fill=black, draw=black, line width=0.05pt] (0,-2) circle (0.03cm); \node[font=\small,black] at (0.2,-2) {$-z_r$};
			\path (-5,-3.5) rectangle (5,3.5);
			% 区域
			\node[font=\small,orange!40!black] at (0,0.5) {$\Omega_p$};
			\node[font=\small,blue!40!black] at (0,-0.5) {$\overline{\Omega}_p$};
		\end{tikzpicture}
		\caption{The opening of lenses in Region \RNum{3} for $\xi>-\frac{q_r^2}{6}$.}
		\label{o4}
	\end{minipage}
\end{figure}
\par
\begin{rhp}
	Find a $2\times2$ matrix-valued function $\widetilde{\mathcal{M}}_p(z)=\widetilde{\mathcal{M}}_p(x,t;z)$ satisfying:\\ 
	(1) $\widetilde{\mathcal{M}}_p(z)$ is holomorphic for $z\in\mathbb{C}\backslash([-q_r,q_r]\cup \Gamma_p\cup \overline{\Gamma}_p)$.  \\
	(2) For $z\in(-q_r,q_r)\cup \Gamma_p\cup \overline{\Gamma}_p$, the jump condition is
	\begin{equation}
		\widetilde{\mathcal{M}}_{p+}(x,t;z)=\widetilde{\mathcal{M}}_{p-}(x,t;z)\widetilde{\mathcal{J}}_p(x,t;z), 
	\end{equation}
	where 
	\begin{equation}
		\widetilde{\mathcal{M}}_p(x,t;z)=\begin{cases}
			\begin{aligned}
				&\begin{pmatrix}
					1&0\\re^{2itg_r}&0
				\end{pmatrix}, &z\in \Gamma_p&
				\\
				&\begin{pmatrix}
					1&-r^*e^{-2itg_r}\\0&1
				\end{pmatrix},& z\in\overline{\Gamma}_p,& \\
				&-i\sigma_2, &z\in (-q_r,q_r). & 
			\end{aligned}
		\end{cases}
	\end{equation}
	(3) $\widetilde{\mathcal{M}}_p(z)=I+\mathcal{O}(z^{-1})$ as $z\to\infty$ in $\mathbb{C}\backslash(\Gamma_p\cup \overline{\Gamma}_p)$. \\
	(4) Symmetries: 
	\begin{equation}
		\widetilde{\mathcal{M}}_p(z)=\overline{\widetilde{\mathcal{M}}_p(-\bar{z})}=\sigma_1\widetilde{\mathcal{M}}_p(-z)\sigma_1=\sigma_1\overline{\widetilde{\mathcal{M}}_p(\bar{z})}\sigma_1. 
	\end{equation}
\end{rhp}

\subsection{The outer parametrix for $\mathcal{M}_r^{(\infty)}(z)$}\

According to the distribution of the sign of $\Im g_r$ in Figure \ref{g3} and \ref{g4}, when $t$ is large enough, both $\widetilde{\mathcal{J}}_r(x,t;z)$ and $\widetilde{\mathcal{J}}_p(x,t;z)$ converge uniformly to $-i\sigma_2$ on the interval $(-q_r,q_r)$ and to the identity matrix elsewhere. Thus, it is natural to consider the following model RHP, whose unique solution is $\mathcal{M}_r^{(\infty)}(z)=\mathcal{E}_r(z)$, where $\mathcal{E}_r(z)$ defined in (\ref{epsilon}). 
\par
\begin{rhp}\label{model_r}
	Find a $2\times2$ matrix-valued function $\mathcal{M}_r^{(\infty)}(z)$ satisfying: \\
	(1) $\mathcal{M}_r^{(\infty)}(z)$ is holomorphic for $z\in \mathbb{C}\backslash[-q_r,q_r]$. \\
	(2) For $z\in(-q_r, q_r)$, the jump condition is
	\begin{equation}
		\mathcal{M}_{r+}^{(\infty)}(z)=\mathcal{M}_{r-}^{(\infty)}(z)\cdot(-i\sigma_2).
	\end{equation}
	(3) $\mathcal{M}_r^{(\infty)}(z)=I+\mathcal{O}(z^{-1})$ as $z\to\infty$. 
\end{rhp}
\par

\subsection{The small norm RHPs for $\mathcal{M}_r^{(err)}$ and $\mathcal{M}_p^{(err)}$}\

For $-\frac{q_l^2}{3}-\frac{q_r^2}{6}<\xi<-\frac{q_r^2}{6}$, define the small norm RHP
\begin{equation}\label{err_right}
	\mathcal{M}_r^{(err)}(x,t;z)=\widetilde{\mathcal{M}}_r(x,t;z)\left[\mathcal{M}_r^{(\infty)}(z)\right]^{-1},
\end{equation}
whose jump matrix takes the form
\begin{equation}
	\mathcal{J}_r^{(err)}(x,t;z)=\mathcal{E}_r(z)\widetilde{\mathcal{J}}_r(x,t;z)\mathcal{E}_r(z)^{-1}, \quad z\in\bigcup_{j=1}^{2}(\Gamma_r^{j}\cup\overline{\Gamma}_r^{j}). 
\end{equation}
Similarly, for $\xi>-\frac{q_r^2}{6}$, define
\begin{equation}\label{err_para}
	\mathcal{M}_p^{(err)}(x,t;z)=\widetilde{\mathcal{M}}_p(x,t;z)\left[\mathcal{M}_r^{(\infty)}(z)\right]^{-1}, 
\end{equation}
whose jump matrix is 
\begin{equation}
	\mathcal{J}_p^{(err)}(x,t;z)=\mathcal{E}_r(z)\widetilde{\mathcal{J}}_p(x,t;z)\mathcal{E}_r(z)^{-1}, \quad z\in\Gamma_p\cup\overline{\Gamma}_p. 
\end{equation}
\par
It can be shown that both $\mathcal{J}_r^{(err)}(z)$ and $\mathcal{J}_p^{(err)}(z)$ are exponentially close to identity as $t\to\infty$, that is 
\begin{equation}
	||\langle \cdot \rangle (\mathcal{J}_r^{(err)}-I)||_{L^q(\Gamma_r\cup\overline{\Gamma}_r)}=\mathcal{O}(e^{-ct}),\quad ||\langle \cdot \rangle (\mathcal{J}_p^{(err)}-I)||_{L^q(\Gamma_p\cup\overline{\Gamma}_p)}=\mathcal{O}(e^{-ct}),
\end{equation}
for $q\ge1$  and some constant $c>0$. Then the small norm theory for Riemann–Hilbert problem indicates that 
\begin{equation}
	\begin{aligned}
		&\mathcal{M}_r^{(err)}(x,t;z)=I+\frac{\mathcal{M}_{r,1}^{(err)}(x,t)}{z}+\mathcal{O}(z^{-2}), \qquad\mathcal{M}_{r,1}^{(err)}(x,t)=\mathcal{O}(e^{-ct}), \\
		&\mathcal{M}_p^{(err)}(x,t;z)=I+\frac{\mathcal{M}_{p,1}^{(err)}(x,t)}{z}+\mathcal{O}(z^{-2}), \qquad \mathcal{M}_{p,1}^{(\infty)}(x,t)=\mathcal{O}(e^{-ct}). 
	\end{aligned}
\end{equation}
\par
Finally, the formulas (\ref{construct}), (\ref{err_right}) and (\ref{err_para}) result in the long-time asymptotic solution $q(x,t)$ for $t\to \infty$ as
\begin{equation}
	q(x,t)=
	\begin{cases}
		\begin{aligned}
			&2i\lim_{z\to\infty}(z\mathcal{M}_r^{(\infty)}(z))_{12}+\left(\mathcal{M}_{r,1}^{(err)}(x,t)\right)_{12}=q_r+\mathcal{O}(e^{-ct}),&-\frac{q_l^2}{3}-\frac{q_r^2}{6}<\xi<-\frac{q_r^2}{6}, \\
			&2i\lim_{z\to\infty}(z\mathcal{M}_r^{(\infty)}(z))_{12}+\left(\mathcal{M}_{p,1}^{(err)}(x,t)\right)_{12}=q_r+\mathcal{O}(e^{-ct}), &\xi>-\frac{q_r^2}{6}. 
		\end{aligned}
	\end{cases}
\end{equation}
This completes the proof of Theorem \ref{main} for Region \RNum{3}. \\
\par
\par

\appendix

\section{Traveling wave solution of the defocusing mKdV equation}\label{Appendix-A}
\par
Substituting the traveling wave ansatz 
\begin{equation}
	q(x,t)=q(\phi), \quad \phi=x-\mathcal{V}t 
\end{equation}
into the defocusing mKdV equation (\ref{mkdv}) and integrating twice, the following ordinary differential equation is obtained
\begin{equation}\label{q-polynomial}
	q_{\phi}^2=q^4+\mathcal{V}q^2+Bq+A, 
\end{equation}
where $A$ and $B$ are constants of integration. Assuming the polynomial on the right side of (\ref{q-polynomial}) admits four real roots ordered as $q_1\le q_2\le q_3\le q_4$, the ordinary differential equation (\ref{q-polynomial}) is expressed in factorized form: 
\begin{equation}
	q_{\phi}^2=(q-q_1)(q-q_2)(q-q_3)(q-q_4), \quad \sum_{i=1}^{4}q_i=0, \quad \mathcal{V}=\sum_{i<k} q_i q_k. 
\end{equation}
The periodic solution corresponds to oscillations within the interval
\begin{equation}
	q_2\le q \le q_3, 
\end{equation}
where the polynomial on the right side of (\ref{q-polynomial}) is non-negative. The phase variable $\phi$ is then given by the elliptic integral of the form 
\begin{equation}
	\phi-\phi_0=\int_{q}^{q_3}\frac{ds}{\sqrt{(s-q_1)(s-q_2)(q_3-s)(q_4-s)}}, 
\end{equation}
with $\phi_0$ denoting the initial phase. This integral yields the explicit periodic solution of the defocusing mKdV equation (\ref{mkdv}) in terms of Jacobi elliptic functions $\mathrm{sn}(u|m)$ and $\mathrm{cn}(u|m)$: 
\begin{equation}\label{periodic-solu-1}
	q_{per}(x,t)=q(\phi)=q_2+\frac{(q_3-q_2)\mathrm{cn}^2(\sqrt{(q_3-q_1)(q_4-q_2)}(x-\mathcal{V}t-\phi_0)/2|m)}{1-\left(\frac{q_3-q_2}{q_4-q_2}\right)\mathrm{sn}^2(\sqrt{(q_3-q_1)(q_4-q_2)}(x-\mathcal{V}t-\phi_0)/2|m)},
\end{equation}
where the modulus $m$ of the Jacobi elliptic function is defined by 
\begin{equation}
	m^2=\frac{(q_3-q_2)(q_4-q_1)}{(q_4-q_2)(q_3-q_1)}. 
\end{equation}
\par
Introducing the variable transformation $\{q_1,q_2,q_3,q_4\} \mapsto \{z_1,z_2,z_3\}$ via
\begin{equation}
	q_1=-z_1-z_2-z_3, \quad q_2=z_1+z_2-z_3, \quad q_3=z_1+z_3-z_2, \quad q_4=z_2+z_3-z_2, 
\end{equation}
with $z_1<z_2<z_3$. This choice of transformation relates the elliptic periodic solution to the band spectrum of the Lax operator in (\ref{lax}). This simplifies the periodic solution (\ref{periodic-solu-1}) into
\begin{equation}
	q_{per}(x,t)=z_1+z_2-z_3+\frac{2(z_3-z_2)(z_3-z_1)\mathrm{cn}^2(\sqrt{z_3^2-z_1^2}(x-\mathcal{V}t-\phi_0)|m)}{(z_3-z_1)-(z_3-z_2)\mathrm{sn}^2(\sqrt{z_3^2-z_1^2}(x-\mathcal{V}t-\phi_0)|m)},
\end{equation}
where the velocity of the traveling wave and modulus of the Jacobi elliptic function are transformed into 
\begin{equation}
	\mathcal{V}=z_1^2+z_2^2+z_3^3,\quad m^2=\frac{z_3^2-z_2^2}{z_3^2-z_1^2}. 
\end{equation}
For the specific case of $z_1=q_l$, $z_2=z_d$, $z_3=q_r$ and $\phi_0=x_0$, this periodic solution reduces to the leading-order term $q_{dsw}(x,t)$ in (\ref{qdsw}) in Theorem \ref{main} immediately, that is  
\begin{equation}
	q_{dsw}(x,t)=q_{per}(x,t). 
\end{equation}  
\par

\section{Two class of local models}\label{Appendix-B}

For completeness, this section presents the essential details of two canonical local model RHPs, which play a pivotal role in the construction and asymptotic analysis of local parametrices. As these models are well-established in the literature \cite{DZ1993,Jenkins2015}, we focus solely on stating the relevant results to be employed.
\par

\subsection{The parabolic cylinders model} \label{Appendix-B-1}\

Given a complex parameter $\rho \in \mathbb{C}$, define the auxiliary quantity $\nu(\rho)=-(2\pi)^{-1}\ln(1 - |\rho|^2)$, which induces the following RHP: 

\begin{rhp}\label{pc}
	Find a $2\times2$ complex matrix function $\mathcal{M}^{(PC)}(\zeta)=\mathcal{M}^{(PC)}(\zeta;\rho)$ that satisfies: \\
	(1) $\mathcal{M}^{(PC)}(\zeta)$ is holomorphic for $\zeta\in\mathbb{C}\backslash\Gamma^{(PC)}$, where
		$\Gamma^{(PC)}:=(e^{-3i\pi/4}\infty,e^{i\pi/4}\infty)\cup(e^{3i\pi/4}\infty,e^{-i\pi/4}\infty)$
    is shown in Figure \ref{pcline}.\\
	(2) For $\zeta\in\Gamma^{(PC)}$, the jump condition is  
	\begin{equation}
		\mathcal{M}_+^{(PC)}(\zeta)=\mathcal{M}_-^{(PC)}(\zeta)\mathcal{J}^{(PC)}(\zeta), 
	\end{equation}
	where 
	\begin{equation}
		\mathcal{J}^{(PC)}(\zeta;\rho)=\begin{cases}
			\begin{aligned}
				&\begin{pmatrix}
					1&0\\\rho\zeta^{-2i\nu}e^{i\zeta^2/2}&1
				\end{pmatrix},& \zeta\in(0,e^{i\pi/4}\infty),\\
				&\begin{pmatrix}
					1&-\bar{\rho}\zeta^{2i\nu}e^{-i\zeta^2/2}\\0&1
				\end{pmatrix},&\zeta\in (0,e^{-i\pi/4}\infty),\\
				&\begin{pmatrix}
					1&\frac{-\bar{\rho}}{1-|\rho|^2}\zeta^{2i\nu}r^{-i\zeta^2/2}\\0&1
				\end{pmatrix},&\zeta\in(e^{3i\pi/4}\infty,0),\\
				&\begin{pmatrix}
					1&0\\ \frac{\rho}{1-|\rho|^2}&1
				\end{pmatrix},& \zeta\in(e^{-3i\pi/4}\infty,0). 
			\end{aligned}
		\end{cases}
	\end{equation}
	(3) As $\zeta\to\infty$, $\mathcal{M}^{(PC)}(\zeta;\rho)=I+\zeta^{-1}\cdot\mathcal{M}_1^{(PC)}(\rho)+\mathcal{O}(\zeta^{-2})$. 
\end{rhp}
\begin{figure}[ht]
	\begin{minipage}[b]{0.48\textwidth}  % 确保宽度总和不超过 1.0\textwidth
		\centering
		\begin{tikzpicture}[scale=1]
			\draw[line width=1pt,dashed, draw=gray] (-3,0) -- (3,0) node[pos=1, below, font=\small]{$\mathbb{R}$} node[pos=0.5,below,font=\small]{$O$};
			% 第一条线
			\draw[line width=1pt,-{Stealth[length=2mm, width=1.5mm]}] (-2.5,-2.5) -- (-1,-1);
			\draw[line width=1pt,-{Stealth[length=2mm, width=1.5mm]}] (-2,-2) -- (1.3,1.3);
			\draw[line width=1pt] (0,0) -- (2.5,2.5); 
			% 第二条线
			\draw[line width=1pt,-{Stealth[length=2mm, width=1.5mm]}] (-2.5,2.5) -- (-1,1);
			\draw[line width=1pt,-{Stealth[length=2mm, width=1.5mm]}] (-2,2) -- (1.3,-1.3);
			\draw[line width=1pt] (0,0) -- (2.5,-2.5); 
			\path (-3,-3) rectangle (3,3);
			\draw (0.5,0) arc (0:45:0.5) node[pos=0.7,right] {$\frac{\pi}{4}$};
		\end{tikzpicture}
		\caption{The jump contour for $\mathcal{M}^{(PC)}(\zeta)$}
		\label{pcline}
	\end{minipage}
	\begin{minipage}[b]{0.48\textwidth}  % 确保宽度总和不超过 1.0\textwidth
		\centering
		\begin{tikzpicture}[scale=1]
			\draw[line width=1pt,dashed, draw=gray] (-2.5,0) -- (2.5,0) node[pos=1, below, font=\small]{$\mathbb{R}$} node[pos=0.5,below,font=\small]{$O$};
			\draw[line width=1pt,-{Stealth[length=2mm, width=1.5mm]}] (-2.5,0) -- (-1,0);
			\draw[line width=1pt,-{Stealth[length=2mm, width=1.5mm]}] (-1.25,0) -- (1.25,0);
			\draw[line width=1pt] (1,0) -- (2.5,0); 
			% 第一条线
			\draw[line width=1pt,-{Stealth[length=2mm, width=1.5mm]}] (-1.5,-2.595) -- (-1,-1.73);
			\draw[line width=1pt] (-1.5,-2.595) -- (0,0);
			% 第二条线
			\draw[line width=1pt,-{Stealth[length=2mm, width=1.5mm]}] (-1.5,2.595) -- (-1,1.73);
			\draw[line width=1pt] (-1.5,2.595) -- (0,0);
			\path (-2,-2) rectangle (2,2);
			\draw (0.5,0) arc (0:120:0.5) node[pos=0.7,above] {$\frac{2\pi}{3}$};
		\end{tikzpicture}
		\caption{The jump contour for $\mathcal{M}^{(Ai)}(\zeta)$}
		\label{ailine}
	\end{minipage}
\end{figure}
\par

It has been shown that the RHP \ref{pc} admits a unique solution \cite{DZ1993}. For $\rho=0$, the solution is trivial, while for $\rho\neq0$, it can be explicitly constructed by using the parabolic cylinder function $\mathcal{D}_a(z)$, which satisfies the differential equation $\partial_z^2\mathcal{D}_a + \left(1/2-z^2/4+a\right)\mathcal{D}_a = 0$. To be specific, the solution to RHP \ref{pc} takes the form
\begin{equation}
	\mathcal{M}^{(PC)}(\zeta;\rho)=\Psi_1(\zeta;\rho)\mathcal{P}(\zeta;\rho)e^{i\frac{\zeta^2}{4}\sigma_3}\zeta^{-i\nu\sigma_3}, 
\end{equation}
where
\begin{equation}
	\Psi_1(\zeta;\rho)=\begin{cases}
		\begin{aligned}
			&\begin{pmatrix}
				e^{-\frac{3\pi\nu}{4}}\mathcal{D}_{i\nu}\left(e^{-\frac{3i\pi}{4}}\zeta\right)&-i\beta e^{\frac{\pi}{4}(\nu-i)}\mathcal{D}_{-i\nu-1}\left(e^{-\frac{i\pi}{4}}\zeta\right)\\i\bar{\beta} e^{-\frac{3\pi}{4}(\nu+i)}\mathcal{D}_{i\nu-1}\left(e^{-\frac{3i\pi}{4}}\zeta\right)&e^{\frac{\pi\nu}{4}}\mathcal{D}_{-i\nu}\left(e^{-\frac{i\pi}{4}}\zeta\right)
			\end{pmatrix},&\Im\zeta>0,\\
			&\begin{pmatrix}
				e^{\frac{\pi\nu}{4}}\mathcal{D}_{i\nu}\left(e^{\frac{i\pi}{4}}\zeta\right)&-i\beta e^{-\frac{3\pi}{4}(\nu-i)}\mathcal{D}_{-i\nu-1}\left(e^{\frac{3i\pi}{4}}\zeta\right)\\i\bar{\beta} e^{\frac{\pi}{4}(\nu+i)}\mathcal{D}_{i\nu-1}\left(e^{\frac{i\pi}{4}}\zeta\right)&e^{-\frac{3\pi\nu}{4}}\mathcal{D}_{-i\nu}\left(e^{\frac{3i\pi}{4}}\zeta\right)
			\end{pmatrix},&\Im\zeta<0, 
		\end{aligned}
	\end{cases}
\end{equation}
with $\beta=\beta(\rho)=\frac{\sqrt{2\pi}e^{i\pi/4}e^{-\pi\nu/2}}{\rho\cdot\Gamma(-i\nu)}$ and
\begin{equation}
	\mathcal{P}(\zeta;\rho)=\begin{cases}
		\begin{aligned}
			&\begin{pmatrix}
				1&0\\-\rho&1
			\end{pmatrix}, &\arg\zeta\in(0,\frac{\pi}{4}), \\
			&\begin{pmatrix}
				1&-\bar{\rho}\\0&1
			\end{pmatrix},&\arg\zeta\in(-\frac{\pi}{4},0),\\
			&\begin{pmatrix}
				1&\frac{\bar{\rho}}{1-|\rho|^2}\\0&1, 
			\end{pmatrix}, &\arg\zeta\in(\frac{3\pi}{4},\pi),\\
			&\begin{pmatrix}
				1&0\\ \frac{\rho}{1-|\rho|^2}&1,
			\end{pmatrix}, &\arg\zeta\in(-\pi,-\frac{3\pi}{4}),\\
			&\qquad I,&\arg\zeta\in(\frac{\pi}{4},\frac{3\pi}{4})\cup(-\frac{3\pi}{4},-\frac{\pi}{4}).
		\end{aligned}
	\end{cases} 
\end{equation}
Moreover, the leading-order correction matrix $\mathcal{M}_1^{(PC)}(\rho)$ has the explicit form:
\begin{equation}
	\mathcal{M}_1^{(PC)}(\rho)=\begin{pmatrix}
		0&-i\beta(\rho)\\i\overline{\beta(\rho)}&0
	\end{pmatrix} . 
\end{equation}
\par

\subsection{The Airy model} \label{Appendix-B-2}\

Consider the following RHP for function $\mathcal{M}^{(Ai)}(\zeta)$:

\begin{rhp}\label{ai}
	Find a $2\times2$ complex matrix-valued function $\mathcal{M}^{(Ai)}(\zeta)$ that satisfies: \\
	(1) $\mathcal{M}^{(Ai)}(\zeta)$ is holomorphic for $\zeta\in\mathbb{C}\backslash\Gamma^{(Ai)}$, where $\Gamma^{(Ai)}=\mathbb{R}\cup e^{2i\pi/3}\mathbb{R}_+\cup e^{-2i\pi/3}\mathbb{R}_+$ is shown in Figure \ref{ailine}. \\
	(2) For $\zeta\in\Gamma^{(Ai)}$, the jump condition is
	\begin{equation}
		\mathcal{M}^{(Ai)}_+(\zeta)=\mathcal{M}^{(Ai)}_-(\zeta)\mathcal{J}^{(Ai)}(\zeta), 
	\end{equation} 
	where 
	\begin{equation}
		\mathcal{J}^{(Ai)}(\zeta)=\begin{cases}
			\begin{aligned}
				&\begin{pmatrix}
					1&e^{-\frac{4}{3}\zeta^{3/2}}\\0&1
				\end{pmatrix}, &\zeta\in(0,+\infty), \\
				&\begin{pmatrix}
					1&0\\e^{\frac{4}{3}\zeta^{3/2}}&1
				\end{pmatrix}, &\zeta\in(0,e^{2i\pi/3}\infty)\cup(0,e^{-2i\pi/3}\infty), \\
				&\begin{pmatrix}
					0&1\\-1&0
				\end{pmatrix}, &\zeta\in(-\infty,0). 
			\end{aligned}
		\end{cases}
	\end{equation}
\end{rhp}
This RHP admits a unique solution expressed by Airy function:
\begin{equation}
	\mathcal{M}^{(Ai)}(\zeta)=\Psi_2(\zeta)\cdot\mathcal{S}(\zeta)\cdot e^{\frac{2}{3}\zeta^{3/2}\sigma_3},
\end{equation}
where the piecewise constant matrix $\mathcal{S}(\zeta)$ is defined by
\begin{equation}
	\mathcal{S}(\zeta)=\begin{cases}
		\quad I, & \arg\zeta\in(-\frac{2\pi}{3},\frac{2\pi}{3}), \\
		\begin{pmatrix}1&0\\-1&1\end{pmatrix}, & \arg\zeta\in(\frac{2\pi}{3},\pi), \\
		\begin{pmatrix}1&0\\1&1\end{pmatrix}, & \arg\zeta\in(-\pi,-\frac{2\pi}{3}),
	\end{cases}
\end{equation}
and the fundamental solution $\Psi_2(\zeta)$ takes the form
\begin{equation}
	\Psi_2(\zeta)=\begin{cases}
		\begin{pmatrix}
			\text{Ai}(\zeta) & \omega^2\text{Ai}(\omega^2\zeta) \\
			\text{Ai}'(\zeta) & \omega^2\text{Ai}'(\omega^2\zeta)
		\end{pmatrix}e^{-i\pi\sigma_3/6}, & \Im\zeta>0, \\
		\begin{pmatrix}
			\text{Ai}(\zeta) & -\omega\text{Ai}(\omega\zeta) \\
			\text{Ai}'(\zeta) & -\text{Ai}'(\omega\zeta)
		\end{pmatrix}e^{-i\pi\sigma_3/6}, & \Im\zeta<0,
	\end{cases}
\end{equation}
where $\omega=e^{2i\pi/3}$ and $\text{Ai}(\zeta)$ satisfies the Airy equation $\text{Ai}''(\zeta)=\zeta\text{Ai}(\zeta)$. 
\par
As $\zeta\to\infty$, the asymptotic expansion of Airy function is formulated by
\begin{equation}
	\text{Ai}(\zeta)\sim\frac{e^{-i\pi/12}}{2\sqrt{\pi}}\zeta^{-1/4}\left[\begin{pmatrix}
		1 & i \\ -i & 1
	\end{pmatrix}+\sum_{n=1}^\infty\begin{pmatrix}
		(-1)^ns_n & is_n \\ (-1)^{n+1}it_n & t_n
	\end{pmatrix}\left(\frac{4}{3}\zeta^{3/2}\right)^{-n}\right]e^{-\frac{2}{3}\zeta^{3/2}},
\end{equation}
with coefficients
\begin{equation}
	s_n=\frac{\Gamma(3n+\frac{1}{2})}{54^n n!\,\Gamma(n+\frac{1}{2})}, \quad t_n=-\frac{6n+1}{6n-1}s_n.
\end{equation}
\par

\section{The detailed calculations of reconstruction formula in (\ref{redsw})}\label{Appendix-C}

This section calculates the term $\sigma$ in reconstruction formula (\ref{redsw}) and establishes that the leading-order term $2i\lim_{z\to\infty}z\mathcal{N}_{12}(z)+2i\sigma$ in (\ref{redsw}) coincides precisely with the dispersive shock wave solution $q_{dsw}(x,t)$ given in (\ref{qdsw}). This equivalence is derived from fundamental identities governed by the Jacobi theta functions and Jacobi elliptic functions.  
\par
In what follows, we begin by computing the quantity $\sigma$. From relation (\ref{delta2}), it is derived that 
\begin{equation}
	\mathcal{A}(iq_l)=\int_{q_l}^{iq_l}\eta=-\frac{1}{2}+\frac{\tau}{4}+\frac{\delta_2}{\pi}, 
\end{equation}
which yields the expression for $\sigma$ as: 
\begin{equation}
	\begin{aligned}
		\sigma&=-\frac{iq_l}{2}\frac{[\alpha(iq_l)+\alpha^{-1}(iq_l)]\frac{\Theta(-\frac{1}{2}-\frac{\delta}{\pi})}{\Theta(-\frac{1}{2}+\frac{\delta_2}{\pi})}+i[\alpha(iq_l)-\alpha^{-1}(iq_l)]\frac{\Theta(-\frac{1}{2}-\frac{\delta}{\pi}+\frac{\tau}{2})}{\Theta(-\frac{1}{2}+\frac{\delta_2}{\pi}+\frac{\tau}{2})}e^{-i\tilde{\delta}}}{[\alpha(iq_l)+\alpha^{-1}(iq_l)]\frac{\Theta(-\frac{1}{2}-\frac{\delta}{\pi})}{\Theta(-\frac{1}{2}+\frac{\delta_2}{\pi})}-i[\alpha(iq_l)-\alpha^{-1}(iq_l)]\frac{\Theta(-\frac{1}{2}-\frac{\delta}{\pi}+\frac{\tau}{2})}{\Theta(-\frac{1}{2}+\frac{\delta_2}{\pi}+\frac{\tau}{2})}e^{-i\tilde{\delta}}}\\
		&=-\frac{iq_l}{2}\frac{i\frac{\alpha(iq_l)+\alpha^{-1}(iq_l)}{\alpha(iq_l)-\alpha^{-1}(iq_l)}\frac{\Theta(-\frac{1}{2}+\frac{\delta_2}{\pi}+\frac{\tau}{2})}{\Theta(-\frac{1}{2}+\frac{\delta_2}{\pi})}e^{i\delta_2}-\frac{\Theta(-\frac{1}{2}-\frac{\delta}{\pi}+\frac{\tau}{2})}{\Theta(-\frac{1}{2}-\frac{\delta}{\pi})}e^{-i\delta}}{i\frac{\alpha(iq_l)+\alpha^{-1}(iq_l)}{\alpha(iq_l)\alpha^{-1}(iq_l)}\frac{\Theta(-\frac{1}{2}+\frac{\delta_2}{\pi}+\frac{\tau}{2})}{\Theta(-\frac{1}{2}+\frac{\delta_2}{\pi})}e^{i\delta_2}+\frac{\Theta(-\frac{1}{2}-\frac{\delta}{\pi}+\frac{\tau}{2})}{\Theta(-\frac{1}{2}-\frac{\delta}{\pi})}e^{-i\delta}}. 
	\end{aligned}
\end{equation}
To simplify this expression, introduce the four Jacobi theta functions: 
\begin{equation}
	\begin{aligned}
		\theta_1(z | \tau)&:=-i\sum_{n=-\infty}^{\infty}(-1)^n q^{(n+\frac{1}{2})^2}e^{i(2n+1)z},& \quad &\theta_2(z | \tau):=\sum_{n=-\infty}^{\infty} q^{(n+\frac{1}{2})^2}e^{i(2n+1)z},\\
		\theta_3(z | \tau)&:=\sum_{n=-\infty}^{\infty} q^{n^2}e^{2inz}=\Theta(\frac{z}{\pi}),& \quad  &\theta_4(z | \tau):=\sum_{n=-\infty}^{\infty}(-1)^n  q^{n^2}e^{2inz},
	\end{aligned}
\end{equation}
where the nome $q=e^{i\pi \tau}$ and the parameter $\tau$ given by (\ref{tau}). Recalling the connection between Jacobi theta functions and Jacobi elliptic functions, we have
\begin{equation}
	\mathrm{sn}\left(u|\tilde{m}\right)=\frac{\theta_3(0|\tau)}{\theta_2(0|\tau)}\cdot\frac{\theta_1(z|\tau)}{\theta_4(z|\tau)},\quad \mathrm{cn}(u|\tilde{m})=\frac{\theta_4(0|\tau)}{\theta_2(0|\tau)}\cdot\frac{\theta_2(z|\tau)}{\theta_4(z|\tau)}, \quad \mathrm{dn}(u|\tilde{m})=\frac{\theta_4(0|\tau)}{\theta_3(0|\tau)}\cdot\frac{\theta_3(z|\tau)}{\theta_4(z|\tau)}, 
\end{equation}
with $u=\theta^2_3(0|\tau)\cdot z=\frac{2K(\tilde{m})}{\pi}\cdot z$. 
\par
In addition, applying Landen’s transformations and half-argument formulas, it yields that 
\begin{equation}
	\begin{cases}
		\begin{aligned}
			\mathrm{sn}\left((1+m')u|\tilde{m}\right)=&(1+m')\mathrm{sn}(u|m)\mathrm{cd}(u|m), \\
			\mathrm{cn}\left((1+m')u|\tilde{m}\right)=&\mathrm{nd}(u|m)-(1+m')\mathrm{sn}(u|m)\mathrm{sd}(u|m), \\
			\mathrm{dn}\left((1+m')u|\tilde{m}\right)=&\mathrm{nd}(u|m)+(m'-1)\mathrm{sn}(u|m)\mathrm{sd}(u|m), 
		\end{aligned}
	\end{cases} \qquad \tilde{m}=\frac{1-m'}{1+m'}, 
\end{equation}
and 
\begin{equation}
	\begin{cases}
		\begin{aligned}
			&\mathrm{sn}^2\left(\frac{u}{2}\Big|m\right)=\frac{1-\mathrm{cn}(u|m)}{1+\mathrm{dn}(u|m)},\\
			&\mathrm{cn}^2\left(\frac{u}{2}\Big|m\right)=\frac{\mathrm{cn}(u|m)+\mathrm{dn}(u|m)}{1+\mathrm{dn}(u|m)},\\   &\mathrm{dn}^2\left(\frac{u}{2}\Big|m\right)=\frac{\mathrm{cn}(u|m)+\mathrm{dn}(u|m)}{1+\mathrm{cn}(u|m)}.
		\end{aligned}
	\end{cases} 
\end{equation}
Then using the periodicity properties of Jacobi theta functions, it is obtained that
\begin{equation}
	\begin{aligned}
		\frac{\Theta(0)\Theta(-\frac{1}{2}-\frac{\delta}{\pi}+\frac{\tau}{2})}{\Theta(\frac{\tau}{2})\Theta(-\frac{1}{2}-\frac{\delta}{\pi})}e^{-i\delta}= &\frac{\theta_3(0|\tau)\theta_3(-\delta+\frac{\pi}{2}+\frac{\pi\tau}{2}|\tau)}{\theta_3(\frac{\pi\tau}{2}|\tau)\theta_3(-\delta+\frac{\pi}{2}|\tau)}e^{-i\delta} 
		=\frac{\theta_3(0|\tau)\theta_1(-\delta|\tau)iq^{-\frac{1}{4}}e^{i\delta}}{\theta_2(0|\tau)\theta_4(-\delta|\tau)q^{-\frac{1}{4}}}e^{-i\delta}\\
		=&(-i)\cdot\frac{\theta_3(0|\tau)\theta_1(\delta|\tau)}{\theta_2(0|\tau)\theta_4(\delta|\tau)}
		=(-i)\cdot \mathrm{sn}\left(2K(\tilde{m})\frac{\delta}{\pi}\bigg|\tilde{m}\right)\\
		=&(-i)\cdot(1+m') \mathrm{sn}\left(K(m)\frac{\delta}{\pi}\bigg|m\right) \mathrm{cd}\left(K(m)\frac{\delta}{\pi}\bigg|m\right). 
	\end{aligned}
\end{equation}
Furthermore, applying equation (\ref{delta2}) with $F:=F({q_l}/{z_d},m)$ gives
\begin{equation}
	\begin{aligned}
		\frac{\Theta(0)\Theta(-\frac{1}{2}+\frac{\delta_2}{\pi}+\frac{\tau}{2})}{\Theta(\frac{\tau}{2})\Theta(-\frac{1}{2}+\frac{\delta_2}{\pi})}e^{i\delta_2}=&i\cdot \mathrm{sn}\left(2K(\tilde{m})\frac{\delta_2}{\pi}|\tilde{m}\right)=i\cdot(1+m')\mathrm{sn}\left(\frac{F}{2}\Bigg|m\right)\mathrm{cd}\left(\frac{F}{2}\Bigg|m\right)\\
		=&i\cdot\frac{(1+m')\mathrm{sn}\left(F\Big|m\right)}{1+\mathrm{dn}\left(F\Big|m\right)}=i\cdot\frac{q_l(\tilde{q}+\tilde{d})}{\tilde{q}z_d+\tilde{d}q_r}:=i\cdot \mathcal{K}_1. 
	\end{aligned}
\end{equation}
Also the following identity holds 
\begin{equation}
	i\cdot\frac{\alpha(iq_l)+\alpha^{-1}(iq_l)}{\alpha(iq_l)-\alpha^{-1}(iq_l)}=\frac{q_l^2+\tilde{q}\tilde{d}+z_dq_r}{q_l(\tilde{q}-\tilde{d})}=\frac{q_r\tilde{d}+z_d\tilde{q}}{q_l(q_r-z_d)}=\frac{\tilde{q}+\tilde{d}}{\mathcal{K}_1\cdot(q_r-z_d)}. 
\end{equation}
Combining these results, the term $\sigma$ is reduced to 
\begin{equation}
	\begin{aligned}
		\sigma=&-\frac{iq_l}{2}\cdot\frac{i\cdot \frac{\tilde{q}+\tilde{d}}{q_r-z_d}-(-i)\cdot(1+m')\mathrm{sn}\left(K(m)\frac{\delta}{\pi}|m\right) \mathrm{cd}\left(K(m)\frac{\delta}{\pi}|m\right)}{i\cdot\frac{\tilde{q}+\tilde{d}}{q_r-z_d}+(-i)\cdot(1+m')\mathrm{sn}\left(K(m)\frac{\delta}{\pi}|m\right) \mathrm{cd}\left(K(m)\frac{\delta}{\pi}|m\right)}\\
		=&-\frac{iq_l}{2}\cdot\frac{\tilde{q}+(q_r-z_d)\mathrm{sn}\left(K(m)\frac{\delta}{\pi}|m\right) \mathrm{cd}\left(K(m)\frac{\delta}{\pi}|m\right)}{\tilde{q}-(q_r-z_d)\mathrm{sn}\left(K(m)\frac{\delta}{\pi}|m\right) \mathrm{cd}\left(K(m)\frac{\delta}{\pi}|m\right)}. 
	\end{aligned}
\end{equation}
\par
We now return to compute the term $2i\lim_{z\to\infty} z \mathcal{N}_{12}(w(z))$. By exploiting the connection between Jacobi theta functions and elliptic functions, it is derived that 
\begin{equation}
	\begin{aligned}
		2i\lim_{z\to\infty} z \mathcal{N}_{12}(w(z))&=(\tilde{q}-\tilde{d})\cdot\frac{\Theta(0)\Theta\left(-\frac{\tilde{\delta}}{\pi}-\frac{\tau}{2}\right)}{\Theta(\frac{\tilde{\delta}}{\pi})\Theta\left(-\frac{\tau}{2}\right)}e^{i\tilde{\delta}}=(\tilde{q}-\tilde{d})\cdot\frac{\theta_3(0)\theta_3(\tilde{\delta}+\frac{\pi\tau}{2})}{\theta_3(\frac{\pi\tau}{2})\theta_3(\tilde{\delta})}e^{i\tilde{\delta}}\\
		&=(\tilde{q}-\tilde{d})\cdot\frac{\theta_3(0)\theta_2(\tilde{\delta})q^{-\frac{1}{4}}e^{-i\tilde{\delta}}}{\theta_2(0)\theta_3(\tilde{\delta})q^{-\frac{1}{4}}}e^{i\tilde{\delta}}=(\tilde{q}-\tilde{d})\cdot\frac{\theta_3(0)\theta_2(\tilde{\delta})}{\theta_2(0)\theta_3(\tilde{\delta})}\\
		&=(\tilde{q}-\tilde{d})\cdot\mathrm{cd}\left(2K(\tilde{m})\frac{\tilde{\delta}}{\pi}\Bigg|\tilde{m}\right)=(\tilde{q}-\tilde{d})\cdot\mathrm{cd}\left(2K(\tilde{m})\frac{\delta+\delta_2}{\pi}\Bigg|\tilde{m}\right). 
	\end{aligned}
\end{equation}
\par
To proceed, recall the basic formulas of Jacobi elliptic function below
\begin{equation}
	\mathrm{cd}(u+v|\tilde{m})=\frac{\mathrm{cn}(u|\tilde{m})\mathrm{dn}(u|\tilde{m})\mathrm{cn}(v|\tilde{m})\mathrm{dn}(v|\tilde{m})-\tilde{m}'^2\mathrm{sn}(u|\tilde{m})\mathrm{sn}(v|\tilde{m})}{1-\tilde{m}^2\mathrm{sn}^2(u|\tilde{m})-\tilde{m}^2\mathrm{sn}^2(v|\tilde{m})-\tilde{m}^2\mathrm{sn}^2(u|\tilde{m})\mathrm{sn}^2(v|\tilde{m})}, 
\end{equation}

\begin{equation}
\mathrm{sn}^2(u|m)+\mathrm{cn}^2(u|m)=1,\qquad \mathrm{sn}^2(u|m)+m^2\mathrm{dn}^2(u|m)=1. 
\end{equation}
\par
Furthermore, after some non-trivial computations, key auxiliary identities are obtained as follows
\begin{equation}
	\begin{aligned}		              
        &\mathrm{cn}\left(2K(\tilde{m})\frac{\delta_2}{\pi}\bigg|\tilde{m}\right)\mathrm{dn}\left(2K(\tilde{m})\frac{\delta_2}{\pi}\bigg|\tilde{m}\right)=\mathrm{nd}^2\left(\frac{F}{2}\Bigg|m\right)-2\mathrm{sd}^2\left(\frac{F}{2}\Bigg|m\right)+m^2\mathrm{sn}^2\left(\frac{F}{2}\Bigg|m\right)\mathrm{sd}^2\left(\frac{F}{2}\Bigg|m\right)\\
		=&\frac{1+\mathrm{cn}(F|m)}{\mathrm{cn}(F|m)+\mathrm{dn}(F|m)}-\frac{2\mathrm{sn}^2(F|m)}{(\mathrm{cn}(F|m)+\mathrm{dn}(F|m))(1+\mathrm{dn}(F|m))}+\frac{m^2(1-\mathrm{cn}(F|m))\mathrm{sn}^2(F|m)}{(\mathrm{cn}(F|m)+\mathrm{dn}(F|m))(1+dn(F|m))^2}\\
		=&\frac{\tilde{q}(z_d+\tilde{d})}{\tilde{d}(q_r+\tilde{q})}-\frac{2q_l^2\tilde{q}^2}{\tilde{d}(\tilde{q}+q_r)(\tilde{q}z_d+dq_r)}+\frac{\tilde{q}q_l^2(z_d-\tilde{d})(\tilde{q}^2-\tilde{d}^2)}{\tilde{d}(\tilde{q}+q_r)(\tilde{q}z_d+\tilde{d}q_r)^2}:=\mathcal{K}_2, 
	\end{aligned}
\end{equation}
and
\begin{equation}
	\begin{cases}
		(q_r-z_d)\cdot\left[1-\tilde{m}^2\cdot {(\mathcal{K}_1)}^2\right]=\tilde{q}^2\cdot\tilde{m}^2(1+m')^2\cdot\left[1-{(\mathcal{K}_1)}^2\right], \\
		\tilde{q}(\tilde{q}-\tilde{d})\cdot\tilde{m}^2(1+m')^2\cdot {\mathcal{K}_1}=2q_l(q_r-z_d)\cdot\left[1-\tilde{m}^2\cdot {(\mathcal{K}_1)}^2\right],\\
		(\tilde{q}-\tilde{d})\cdot \mathcal{K}_2=(q_r-z_d)\cdot\left[1-\tilde{m}^2\cdot {(\mathcal{K}_1)}^2\right]. 
	\end{cases}
\end{equation}
\par
Substituting all these identities into the reconstruction formula $2i\lim_{z\to\infty}z\mathcal{N}_{12}(w(z))+2i\sigma$, we ultimately obtain that
\allowdisplaybreaks
\begin{equation}
    \begin{split}
	\begin{aligned}
		&2i\lim_{z\to\infty}z\mathcal{N}_{12}(w(z))+2i\sigma\\
		=&(\tilde{q}-\tilde{d})\cdot{\mathcal{K}_2}\cdot\frac{\mathrm{nd}^2\left(K(m)\frac{\delta}{\pi}|m\right)-2\mathrm{sd}^2\left(K(m)\frac{\delta}{\pi}|m\right)+m^2\mathrm{sn}^2\left(K(m)\frac{\delta}{\pi}|m\right)\mathrm{sd}^2\left(K(m)\frac{\delta}{\pi}|m\right)}{\left[1-\tilde{m}^2\cdot {(\mathcal{K}_1)}^2\right]-\tilde{m}^2(1+m')^2\cdot \left[1-{(\mathcal{K}_1)}^2\right]\cdot \mathrm{sn}^2\left(K(m)\frac{\delta}{\pi}|m\right)\mathrm{cd}^2\left(K(m)\frac{\delta}{\pi}|m\right)}\\
		&-(\tilde{q}-\tilde{d})\cdot {\mathcal{K}_1}\cdot\frac{\tilde{m}^2(1+m')\cdot \mathrm{sn}\left(K(m)\frac{\delta}{\pi}|m\right)\mathrm{cd}\left(K(m)\frac{\delta}{\pi}|m\right)}{\left[1-\tilde{m}^2\cdot {(\mathcal{K}_1)}^2\right]-\tilde{m}^2(1+m')^2\cdot\left[1-{(\mathcal{K}_1)}^2\right]\cdot \mathrm{sn}^2\left(K(m)\frac{\delta}{\pi}|m\right)\mathrm{cd}^2\left(K(m)\frac{\delta}{\pi}|m\right)}\\
		&+{q_l} \cdot \frac{\tilde{q}+(q_r-z_d)\mathrm{sn}\left(K(m)\frac{\delta}{\pi}|m\right) \mathrm{cd}\left(K(m)\frac{\delta}{\pi}|m\right)}{\tilde{q}-(q_r-z_d)\mathrm{sn}\left(K(m)\frac{\delta}{\pi}|m\right) \mathrm{cd}\left(K(m)\frac{\delta}{\pi}|m\right)}\\
		=&(\tilde{q}-\tilde{d})\cdot {\mathcal{K}_2}\cdot\frac{\tilde{q}^{2}\cdot\left[1-2\mathrm{sn}^2\left(K(m)\frac{\delta}{\pi}|m\right)+m^2\mathrm{sn}^4\left(K(m)\frac{\delta}{\pi}|m\right)\right]}{\left[1-\tilde{m}^2\cdot {(\mathcal{K}_1)}^2\right]\cdot\left[\tilde{q}^2-2q_r(q_r-z_d)\mathrm{sn}^2\left(K(m)\frac{\delta}{\pi}|m\right) +(q_r-z_d)^2\mathrm{sn}^4\left(K(m)\frac{\delta}{\pi}|m\right) \right]}\\
		&-\tilde{q}^{2}\cdot\left\{(\tilde{q}-\tilde{d})\cdot\tilde{m}^2(1+m')^2\cdot {\mathcal{K}_1}-2q_l\tilde{q}^{-1}(q_r-z_d)\left[1-\tilde{m}^2\cdot {(\mathcal{K}_1)}^2\right]\right\}\\
		&\quad\times\frac{\mathrm{sn}\left(K(m)\frac{\delta}{\pi}|m\right)\mathrm{cd}\left(K(m)\frac{\delta}{\pi}|m\right)}{\left[1-\tilde{m}^2\cdot {(\mathcal{K}_1)}^2\right]\cdot\left[\tilde{q}^2-(q_r-z_d)^2\mathrm{sn}^2\left(K(m)\frac{\delta}{\pi}|m\right) \mathrm{cd}^2\left(K(m)\frac{\delta}{\pi}|m\right)\right]}\\
		&+{q_l}\cdot\frac{\tilde{q}^2+(q_r-z_d)^2\mathrm{sn}^2\left(K(m)\frac{\delta}{\pi}|m\right) \mathrm{cd}^2\left(K(m)\frac{\delta}{\pi}|m\right)}{\tilde{q}^2-(q_r-z_d)^2\mathrm{sn}^2\left(K(m)\frac{\delta}{\pi}|m\right) \mathrm{cd}^2\left(K(m)\frac{\delta}{\pi}|m\right)}\\
		=&-{q_l}+\frac{\tilde{q}^2(q_r-z_d)\left[1-2\mathrm{sn}^2\left(K(m)\frac{\delta}{\pi}|m\right)+m^2\mathrm{sn}^4\left(K(m)\frac{\delta}{\pi}|m\right)\right]}{\tilde{q}^2-2q_r(q_r-z_d)\mathrm{sn}^2\left(K(m)\frac{\delta}{\pi}|m\right) +(q_r-z_d)^2\mathrm{sn}^4\left(K(m)\frac{\delta}{\pi}|m\right) }\\
		&+\frac{2q_l\tilde{q}^2\left[1-m^2\mathrm{sn}^2\left(K(m)\frac{\delta}{\pi}|m\right)\right]}{\tilde{q}^2-2q_r(q_r-z_d)\mathrm{sn}^2\left(K(m)\frac{\delta}{\pi}|m\right) +(q_r-z_d)^2\mathrm{sn}^4\left(K(m)\frac{\delta}{\pi}|m\right)}\\
		=&{q_r}+{q_l}-{z_d}+\frac{2(q_r-z_d)(q_l-z_d)\left[(q_l+q_r)\mathrm{sn}^2\left(K(m)\frac{\delta}{\pi}|m\right)-(q_r-z_d)\mathrm{sn}^4\left(K(m)\frac{\delta}{\pi}|m\right)\right]}{(q_r+q_l)(q_r-q_l)-(q_r-z_d)^2\mathrm{sn}^2\left(K(m)\frac{\delta}{\pi}|m\right) -(q_r-z_d)^2\mathrm{sn}^4\left(K(m)\frac{\delta}{\pi}|m\right)}\\
		=&{q_l}+{z_d}-{q_r}+\frac{2(q_r-z_d)(q_r-q_l)\mathrm{cn}^2\left(K(m)\frac{\delta}{\pi}|m\right)}{(q_r-q_l)-(q_r-z_d)\mathrm{sn}^2\left(K(m)\frac{\delta}{\pi}|m\right)},
	\end{aligned}
    \end{split}
\end{equation}
which is just the leading-order term $q_{dsw}(x,t)$ given in (\ref{qdsw}). \\
\par

{\bf Acknowledgments}
\par
This work was supported by the National Natural Science Foundation of China through grant Nos. 12371247 and 12431008.\\
\par
{\bf Data availability}
\par
 Data sharing is not applicable to this paper as no new data were created or analyzed in this study.\\
 \par
{\bf Declarations}  \\
 \par
{\bf Conflict of interest}
\par
  The authors declare that they have no conflict of interest.
\par
%\bibliography{main}
%\bibliographystyle{plain}
%\bibliography{is}

\begin{thebibliography}{99}
    \bibitem{AKNS_1973} M. J. Ablowitz, D. J. Kaup, A. C. Newell, H. Segur, Nonlinear-Evolution Equations of Physical Significance, Phys. Rev. Lett. 31 (2) (1973) 125-127. 
    \bibitem{KdV1977} M. J. Ablowitz, H. Segur, Asymptotic Solutions of the Korteweg-deVries Equation, Stud. Appl. Math. 57 (1) (1977) 13-44. 
    \bibitem{JDE2016} K. Andreiev, I. Egorova, T. L. Lange, G. Teschl, Rarefaction waves of the Korteweg–de Vries equation via nonlinear steepest descent, J. Differ. Equ. 261 (10) (2016) 5371-5410. 
    \bibitem{BC1984} R. Beals, R. R. Coifman, Scattering and inverse scattering for first order systems, Comm. Pure Appl. Math. 37 (1) (1984) 39-90. 
    \bibitem{Minakov2019} M. Bertola, A. Minakov, Laguerre polynomials and transitional asymptotics of the modified Korteweg–de Vries equation for step-like initial data, Anal. Math. Phys. 9 (4) (2019) 1761-1818. 
    \bibitem{Biondini2021} G. Biondini, S. T. Li, D. Mantzavinos, Long-Time Asymptotics for the Focusing Nonlinear {Schr{\"o}dinger} Equation with Nonzero Boundary Conditions in the Presence of a Discrete Spectrum, Commun. Math. Phys. 382 (2021) 1495-1577. 
    \bibitem{CMP2009} A. Boutet de Monvel, A. Its, V. Kotlyarov, Long-Time Asymptotics for the Focusing NLS Equation with Time-Periodic Boundary Condition on the Half-Line, Commun. Math. Phys. 290 (2009) 479-522. 
    \bibitem{IMRN2011} A. Boutet de Monvel, V. P. Kotlyarov, D. Shepelsky, Focusing NLS Equation: Long-Time Dynamics of Step-Like Initial Data, Int. Math. Res. Not. 2011 (7) (2011) 1613–1653. 
    \bibitem{Lenells2021} A. Boutet de Monvel, J. Lenells, D. Shepelsky, The Focusing {NLS} Equation with Step-Like Oscillating Background: Scenarios of Long-Time Asymptotics, Commun. Math. Phys. 383 (2) (2021) 893-952. 
    \bibitem{Buckingham2007} R. Buckingham, S. Venakides, Long-time asymptotics of the nonlinear Schrödinger equation shock problem, Comms. Pure Appl. Math. 60 (9) (2007) 1349-1414. 
    \bibitem{Jenkins2016} S. Cuccagna and R. Jenkins, On the Asymptotic Stability of {$N$}-Soliton Solutions of the Defocusing Nonlinear {Schr\"odinger} Equation, Commun. Math. Phys. 343 (2016) 921-969. 
    \bibitem{DZ1993} P. Deift, X. Zhou, A steepest descent method for oscillatory Riemann-Hilbert problems. Asymptotics for the MKdV equation, Ann. of Math. 137 (2) (1993) 295-368.
    \bibitem{Whitham1976} C. F. Driscoll, T. M. O’Neil, Modulational instability of cnoidal wave solutions of the modified Korteweg–de Vries equation, J. Math. Phys. 17 (7) (1976) 1196-1200. 
    \bibitem{KdV2013} I. Egorova, Z. Gladka, V. Kotlyarov, G. Teschl, Long-time asymptotics for the Korteweg–de Vries equation with step-like initial data, Nonlinearity 26 (2013) 1839. 

    \bibitem{El-SIAM-2017}G. A. El, M. A. Hoefer, M. Shearer, Dispersive and diffusive-dispersive shock waves for nonconvex conservation laws,  SIAM Review, 59 (2017) 3-61. 

   \bibitem{Lenells2025}S. Fromm, J. Lenells, R. Quirchmayr, The defocusing nonlinear Schrödinger equation with step-like oscillatory initial data, Adv. Diff. Equations, 30 (2025) 455-525.

    
    \bibitem{Minakov2020} T. Grava, A. Minakov, On the Long-Time Asymptotic Behavior of the Modified Korteweg--de Vries Equation with Step-like Initial Data, SIAM J. Math. Anal. 52 (6) (2020) 5892-5993. 
    \bibitem{Grava2002} T. Grava, F. R. Tian, The generation, propagation, and extinction of multiphases in the KdV zero-dispersion limit, Comm. Pure Appl. Math. 55 (12) (2002) 1569-1639. 
    \bibitem{KdV1973} A. V. Gurevich, L. P. Pitaevskii, Decay of Initial Discontinuity in the Korteweg-de Vries Equation, Sov. J. Exp. Theor. Phys. Lett. 17 (1973) 193. 
    \bibitem{JFM1984} K. R. Helfrich, W. K. Melville, J. W. Miles, On interfacial solitary waves over slowly varying topography, J. Fluid Mech. 149 (1984) 305-317. 
    \bibitem{Jenkins2015} R. Jenkins, Regularization of a sharp shock by the defocusing nonlinear Schrödinger equation, Nonlinearity 28 (7) (2015) 2131.  
    \bibitem{Kakutani1978} T. Kakutani, N. Yamasaki, Solitary Waves on a Two-Layer Fluid, J. Phys. Soc. Jpn. 45 (2) (1978) 674-679. 

   \bibitem{AMK-2004} A. M. Kamchatnov, A. Spire, V. V. Konotop, On dissipationless shock waves in a discrete nonlinear Schrödinger equation, J. Phys. A: Math. Gen. 37 (2004) 5547-5568.
    
    \bibitem{Khruslov1976} E. Ya. Khruslov, Asymptotics of the solution of the Cauchy problem for the Korteweg--de Vries equation with initial data of step type, Math. USSR-Sb. 28 (2) (1976) 229-248. 
    \bibitem{Minakov2010} V. Kotlyarov, A. Minakov, Riemann–Hilbert problem to the modified Korteveg–de Vries equation: Long-time dynamics of the steplike initial data, J. Math. Phys. 51 (9) (2010) 093506. 
    \bibitem{Minakov2012} V. Kotlyarov, A. Minakov, Step-Initial Function to the MKdV Equation: Hyper-Elliptic Long-Time Asymptotics of the Solution, J. Math. Phys. Anal. Geom. 8 (1) (2012) 38–62. 
    \bibitem{Minakov2015} V. Kotlyarov, A. Minakov, Modulated elliptic wave and asymptotic solitons in a shock problem to the modified Korteweg-de Vries equation, J. Math. Anal. Appl. 48 (1) (2015) 81-104. 
    \bibitem{Kotlyarov1989} V. P. Kotlyarov, E. Ya. Khruslov, Asymptotic solitons of the modified Korteweg-de Vries equation, Inverse Problems 5 (6) (1989) 1075. 
    \bibitem{Levermore1988} C. D. Levermore, The hyperbolic nature of the zero dispersion Kdv limit, Commun. Partial Differential Equations 13 (4) (1988) 495-514. 
    \bibitem{Minakov2016} A. Minakov, Asymptotics of step-like solutions for the Camassa–Holm equation J. Differ. Equ. 261 (11) (2016) 6055-6098. 
    \bibitem{Miura1968} R. M. Miura, Korteweg‐de Vries Equation and Generalizations. I. A Remarkable Explicit Nonlinear Transformation, J. Math. Phys. 9 (8) (1968) 1202-1204. 
    \bibitem{Muskhelishvili1958} N. I. Muskhelishvili, Singular Integral Equations: Boundary Problems of Functions Theory and Their Applications to Mathematical Physics, Wolters-Noordhoff (1958). 
    \bibitem{Study2021} Y. Rybalko, D. Shepelsky, Curved wedges in the long-time asymptotics for the integrable nonlocal nonlinear Schrödinger equation, Stud. Appl. Math. 147 (3) (2021) 872-903. 
    \bibitem{Tajiri1985} M. Tajiri, K. Nishihara, Solitons and Shock Waves in Two-Electron-Temperature Plasmas, J. Phys. Soc. Jpn. 54 (1985) 572-578. 
    \bibitem{Watanabe1984} S. Watanabe, Ion Acoustic Soliton in Plasma with Negative Ion, J. Phys. Soc. Jpn. 53 (3) (1984) 950-956. 
    \bibitem{Xu2021} T.Y. Xu, On the large-time asymptotics of the defocusing mKdV equation with step-like initial data, arXiv:2204.01299. 
    \bibitem{Zakharov1976} V. E. Zakharov, S. V. Manakov, Asymptotic behavior of non-linear wave systems integrated by the inverse scattering method, Zh. Eksp. Teor. Fiz. 71 (1) (1976) 203. 
    \bibitem{Zhou1989} X. Zhou, The Riemann–Hilbert Problem and Inverse Scattering, SIAM J. Math. Anal. 20 (4) (1989) 966-986. 

\end{thebibliography}

\bibliographystyle{plain}

\end{document}